\documentclass{tran-l}
\usepackage{amsfonts}
\usepackage{mathrsfs}
\usepackage{amssymb}
\usepackage{graphicx}
\usepackage{subfigure}
\usepackage[colorlinks,
            linkcolor=blue,
            anchorcolor=blue,
            citecolor=blue]{hyperref}
\usepackage[modulo]{lineno}
\usepackage{amsthm}          

\newtheorem{corollary}{Corollary}
\newtheorem{theorem}{Theorem}[section]
\newtheorem{lemma}[theorem]{Lemma}
\theoremstyle{definition}
\newtheorem{definition}[theorem]{Definition}
\newtheorem{example}[theorem]{Example}

\theoremstyle{remark}

\numberwithin{equation}{section}

\begin{document}

\title{Some useful inequalities for nabla tempered fractional calculus}


\author{Yiheng Wei}
\address{School of Mathematics, Southeast University,
Nanjing 210096, China}
\curraddr{}
\email{neudawei@seu.edu.cn}
\thanks{The work described in this paper was fully supported by the National Natural Science Foundation of China (62273092) and the National Key R$\&$D Project of China (2020YFA0714300).}

\author{Linlin Zhao}
\address{School of Business, Nanjing Audit University, Nanjing 211815, China}
\curraddr{}
\email{zll2016@nau.edu.cn}
\thanks{}

\author{Kai Cao}
\address{School of Mathematics, Southeast University,
Nanjing 210096, China}
\curraddr{}
\email{kcao@seu.edu.cn}
\thanks{}

\author{Jinde Cao}
\address{School of Mathematics, Southeast University,
Nanjing 210096, China; Yonsei Frontier Lab, Yonsei University, Seoul 03722, South Korea}
\curraddr{}
\email{jdcao@seu.edu.cn}
\thanks{}

\date{}

\dedicatory{}

\begin{abstract}
This paper gives particular emphasis to the nabla tempered fractional calculus, which involves the multiplication of the rising function kernel with tempered functions, and provides a more flexible alternative with considerable promise for practical applications. Some remarkable inequalities for such nabla fractional calculus are developed and analyzed, which greatly enrich the mathematical theory of nabla tempered fractional calculus. Numerical results confirm the validity of the developed properties once again, which also reveals that the introduction of tempered function provides high value and huge potential.
\end{abstract}

\maketitle




\section{Introduction}\label{Section1}
Fractional calculus, as a useful analytical toolbox, has attracted an increasing attention from scholars \cite{Diethelm:2022NODY}. Due to its nonlocal property, fractional calculus play a key role in diverse application of science and engineering, such as diffusion model \cite{Jin:2022JAM}, compartment model \cite{Angstmann:2021SR}, automatic control \cite{Padula:2020JCO}, neural network \cite{Pang:2019JSC}, etc. Especially, tempered fractional calculus, which introduces an extra tempered function, has many merits, and henceforth many researchers have deployed themselves to explore valuables results for such class of fractional calculus. A huge work has been done for this subject \cite{Sabzikar:2015JCP,Almeida:2019JCAM,Fernandez:2020JCAM,Mali:2022MMA}, which makes a positive and profound impact.

For the continuous time case, many diffusion models involving tempered fractional calculus were established and the solvability of the resulting equation becomes important and difficult. To achieve this aim, a lot of work has been carried out. A Chebyshev pseudospectral scheme was developed to discretize the space-time tempered fractional diffusion equation in \cite{Hanert:2014JSC}. A time discretization method was established for approximating the mild solution of the tempered fractional Feynman--Kac equation in \cite{Deng:2018JNA}. An efficient and stable finite difference scheme was proposed for solving space tempered fractional diffusion equations in \cite{Guo:2018JSC}. The tempered L\'{e}vy flights were introduced to process the anomalous diffusion problem in \cite{Deng:2018MMS}. The equivalence between the tempered fractional derivative operator and the Hadamard finite-part integral was first proved and then the fractional linear multistep method was extended to the tempered fractional integral and derivative operators in \cite{Guo:2019JSC}. A class of tempered fractional neural networks was proposed and the conditions for attractivity and Mittag--Leffler stability were provided in \cite{Gu:2021Fractals}.

For the discrete time case, the research is just in its infancy and some properties has been explored preliminarily. The memory effect of delta tempered fractional calculus was investigated and applied to image processing \cite{Abdeljawad:2020Optik}. The tempered fractional derivative on an isolated time scale was defined and a new method was presented based on the time scale theory for numerical discretization in \cite{Fu:2021Fractals}. A general definition for nabla discrete time tempered fractional calculus was presented in \cite{Ferreira:2021NODY}. The tempered function was chosen as the nonzero case instead of the discrete exponential function, which greatly enrich the potential of the tempered fractional calculus. Compared with the continuous time case, the discrete time case performs better in  computing, storage, transport, etc. and it has greater potential in the digital era. Though the study on discrete time tempered fractional calculus is still in sufficient, a proliferation of results reported on discrete time fractional calculus \cite{Cheng:2011Book,Ostalczyk:2015Book,Goodrich:2015Book} could give us a lot of helpful inspiration and reference.

The basic arithmetic and equivalence relations of fractional difference and fractional sum were discussed in \cite{Cheng:2011Book,Ostalczyk:2015Book,Goodrich:2015Book,Wei:2019JCND}. The monotonicity of fractional difference was explored in  \cite{Eloe:2019BKMS,Goodrich:2015Book,Almeida:2019RMJM}. The comparison principles were developed in \cite{Wyrwas:2015Kybe,Baleanu:2017CNSNS,Wu:2017AMC,Liu:2018DSA,Liu:2019TJM,Liu:2021Filomat}, which are dominant in the stability analysis of discrete time fractional order systems and derive fascinating consequences. The fractional difference inequalities were developed in \cite{Baleanu:2017CNSNS,Wu:2017AMC,Wei:2021NODY}, and the research was not comprehensive and some new inequalities were still expected, including the non-convex case, the Cauchy inequality, the Jensen inequality and the H\"{o}lder inequality, etc. Firstly, some similar properties like the classical case can be checked for the tempered case. Secondly, efforts can be made for the unsolved problems of the classical case. Thirdly, this work could try to produce some new remarkable results.

Since there is still a gap in the literature concerning nabla tempered fractional calculus, the main purpose of this study is to investigate the nabla tempered fractional calculus further and derive some fundamental properties which are also the main contributions of this work, including \textsf{i) the monotonicity of nabla tempered fractional difference is explored; ii) three comparison principles for nabla tempered fractional order systems are developed; iii) a series of tempered fractional difference inequalities are derived; and iv) many tempered fractional sum inequalities are built.} Notably, it is not an easy task to generalize the existing properties to the new field, since the introduction of the tempered function brings some unexpected difficulty and damage some accustomed properties. Furthermore, it is more difficult to discover some innovative valuable properties.

The remainder of this paper is organized as follows. Section \ref{Section2} presents some preliminaries on the classical nabla fractional calculus and the tempered case. Section \ref{Section3} is devoted to deriving the main results. Section \ref{Section4} provides three numerical examples to verify the elaborated theoretical results. Finally, this paper ends in Section \ref{Section5} with some concluding remarks.

\section{Preliminaries}\label{Section2}
In this section, the basic definitions for nabla fractional calculus and nabla tempered fractional calculus are presented.

For $x: \mathbb{N}_{a+1-n} \to \mathbb{R}$, its $n$-th nabla difference is defined by
\begin{equation}\label{Eq2.1}
{\textstyle \nabla^{n} x(k):=\sum_{i=0}^{n}(-1)^{i}\big(\begin{smallmatrix}
n \\
i
\end{smallmatrix}\big) x(k-i),}
\end{equation}
where $n \in \mathbb{Z}_{+}$, $k \in \mathbb{N}_{a+1}:=\{a+1,a+2,\cdots\}$, $a \in \mathbb{R}$, $\left(\begin{smallmatrix}p \\ q\end{smallmatrix}\right):=\frac{\Gamma(p+1)}{\Gamma(q+1) \Gamma(p-q+1)}$ is the generalized binomial coefficient and $\Gamma(\cdot)$ is the Gamma function.

For $x: \mathbb{N}_{a+1} \to \mathbb{R}$, its $\alpha$-th Gr\"{u}nwald--Letnikov difference/sum is defined by \cite{Ostalczyk:2015Book,Atici:2021FDC}
\begin{equation}\label{Eq2.2}
{\textstyle {}_{a}^{\rm G} \nabla_{k}^{\alpha} x(k):=\sum_{i=0}^{k-a-1}(-1)^{i}\big(\begin{smallmatrix}
\alpha \\
i
\end{smallmatrix}\big) x(k-i),}
\end{equation}
where $\alpha \in \mathbb{R}$, $k \in \mathbb{N}_{a+1}$ and $a \in \mathbb{R}$. When $\alpha>0$, ${ }_{a}^{\rm G} \nabla_{k}^{\alpha} x(k)$ represents the difference operation. When $\alpha<0$, ${}_{a}^{\rm G} \nabla_{k}^{\alpha} x(k)$ represents the sum operation including the fractional order case and the integer order case. Specially, ${}_{a}^{\rm G} \nabla_{k}^{0} x(k)=x(k)$. Even though $\alpha=n \in \mathbb{Z}_{+}$, ${ }_{a}^{\rm G} \nabla_{k}^{\alpha} x(k) \not \equiv \nabla^{n} x(k)$ for all $k \in \mathbb{N}_{a+1}$.

Defining the rising function $p\overline {^q} := \frac{\Gamma(p+q)}{\Gamma(p)}$, $p\in\mathbb{Z}_+$, $q\in\mathbb{R}$, (\ref{Eq2.2}) can rewritten as
\begin{equation}\label{Eq2.3}
{\textstyle
\begin{array}{rl}
{ }_{a}^{\mathrm{G}} \nabla_{k}^{\alpha} x(k)=&\hspace{-6pt}\sum_{i=0}^{k-a-1}(-1)^{i} \frac{\Gamma(1+\alpha)}{\Gamma(i+1) \Gamma(1+\alpha-i)} x(k-i) \\
=&\hspace{-6pt}\sum_{i=0}^{k-a-1} \frac{\Gamma(i-\alpha)}{\Gamma(i+1) \Gamma(-\alpha)} x(k-i) \\
=&\hspace{-6pt}\sum_{i=0}^{k-a-1} \frac{(i+1)^{\overline{-\alpha-1}}}{\Gamma(-\alpha)} x(k-i) \\
=&\hspace{-6pt}\sum_{i=a+1}^{k} \frac{(k-i+1)^{\overline{-\alpha-1}}}{\Gamma(-\alpha)} x(i),
\end{array}}
\end{equation}
where $\Gamma(\theta) \Gamma(1-\theta)=\frac{\pi}{\sin (\pi \theta)}$, $\theta \in \mathbb{R}$ is adopted.

From the previous definitions, the $\alpha$-th Riemann--Liouville fractional difference and Caputo fractional difference for $x: \mathbb{N}_{a+1-n} \to \mathbb{R}$, $\alpha \in(n-1, n)$, $n \in \mathbb{Z}_{+}$, $k \in \mathbb{N}_{a+1}$ and $a \in \mathbb{R}$ are defined by \cite{Goodrich:2015Book}
\begin{equation}\label{Eq2.4}
{\textstyle
{}_{a}^{\mathrm{R}} \nabla_{k}^{\alpha} x(k):=\nabla^n{}_{a}^{\rm G} \nabla_{k}^{\alpha-n} x(k), }
\end{equation}
\begin{equation}\label{Eq2.5}
{\textstyle
{ }_{a}^{\mathrm{C}} \nabla_{k}^{\alpha} x(k):={ }_{a}^{\rm G} \nabla_{k}^{\alpha-n} \nabla^{n} x(k) .}
\end{equation}
On this basis, the following properties hold.

\begin{lemma}\label{Lemma2.1}\cite{Goodrich:2015Book,Wei:2019JCND}
For any function $x: \mathbb{N}_{a+1-n} \to \mathbb{R}$, $n \in \mathbb{Z}_{+}$, $a \in \mathbb{R}$, one has
\begin{equation}\label{Eq2.6}
{\textstyle
{ }_{a}^{\rm R} \nabla_{k}^\alpha{}_ a^{ \rm G} \nabla_{k}^{-\alpha} x(k)={ }_{a}^{\rm C} \nabla_{k}^{\alpha}{}_a^{ \rm G} \nabla_{k}^{-\alpha} x(k)=x(k),}
\end{equation}
\begin{equation}\label{Eq2.7}
{\textstyle
{ }_{a}^{\rm G} \nabla_{k}^{-\alpha} {}_a^{\rm R} \nabla_{k}^{\alpha} x(k)=x(k)-\sum_{i=0}^{n-1} \frac{(k-a) \overline{^{\alpha-i-1}}}{\Gamma(\alpha-i)}[{ }_{a}^{\rm R} \nabla_{k}^{\alpha-i-1} x(k)]_{k=a},
}
\end{equation}
\begin{equation}\label{Eq2.8}
{\textstyle
{ }_{a}^{\rm G} \nabla_{k}^{-\alpha}{ }_{a}^{\rm C} \nabla_{k}^{\alpha} x(k)=x(k)-\sum_{i=0}^{n-1} \frac{(k-a) \overline{^i}}{i !}[\nabla^{i} x(k)]_{k=a} , }
\end{equation}
where $k \in \mathbb{N}_{a+1}$, $\alpha \in(n-1, n)$.
\end{lemma}

By introducing a tempered function $w: \mathbb{N}_{a+1} \to \mathbb{R} \backslash\{0\}$, the concept of nabla fractional calculus can be extended further.

For $x: \mathbb{N}_{a+1} \to \mathbb{R}$, its $\alpha$-th Gr\"{u}nwald--Letnikov tempered difference/sum is defined by \cite{Ferreira:2021NODY}
\begin{equation}\label{Eq2.9}
{\textstyle
{ }_{a}^{\rm G} \nabla_{k}^{\alpha, w(k)} x(k):=w^{-1}(k)_{a}^{\rm G} \nabla_{k}^{\alpha}[w(k) x(k)],
}
\end{equation}
where $\alpha \in \mathbb{R}, k \in \mathbb{N}_{a+1}, a \in \mathbb{R}$ and $w: \mathbb{N}_{a+1} \to \mathbb{R} \backslash\{0\}$.

The $n$-th nabla tempered difference, the $\alpha$-th Riemann--Liouville tempered fractional difference and Caputo tempered fractional difference of $x: \mathbb{N}_{a+1-n} \to \mathbb{R}$ can be defined by
\begin{equation}\label{Eq2.10}
{\textstyle
\nabla^{n, w(k)} x(k):=w^{-1}(k) \nabla^{n}[w(k) x(k)],}
\end{equation}
\begin{equation}\label{Eq2.11}
{\textstyle
{ }_{a}^{\mathrm{R}} \nabla_{k}^{\alpha, w(k)} x(k):=w^{-1}(k){}_{a}^{\mathrm{R}} \nabla_{k}^{\alpha}[w(k) x(k)],}
\end{equation}
\begin{equation}\label{Eq2.12}
{\textstyle
{ }_{a}^{\rm C} \nabla_{k}^{\alpha, w(k)} x(k):=w^{-1}(k){}_{a}^{\rm C} \nabla_{k}^{\alpha}[w(k) x(k)],}
\end{equation}
respectively, where $\alpha \in(n-1, n), n \in \mathbb{Z}_{+}, k \in \mathbb{N}_{a+1}, a \in \mathbb{R}$ and $w: \mathbb{N}_{a+1} \to \mathbb{R} \backslash\{0\}$. On this basis, the following relationships hold
\begin{equation}\label{Eq2.13}
{\textstyle
{ }_{a}^{\rm R} \nabla_{k}^{\alpha, w(k)} x(k)=\nabla^{n, w(k)}{}_{a}^{\rm G} \nabla_{k}^{\alpha-n, w(k)} x(k),}
\end{equation}

\begin{equation}\label{Eq2.14}
{\textstyle
{ }_{a}^{\rm C} \nabla_{k}^{\alpha, w(k)} x(k)={ }_{a}^{\rm G} \nabla_{k}^{\alpha-n, w(k)} \nabla^{n, w(k)} x(k) .}
\end{equation}

The equivalent condition of $w: \mathbb{N}_{a+1} \to \mathbb{R} \backslash\{0\}$ is finite nonzero. In this work, when $w(k)=(1-\lambda)^{k-a}$, $\lambda \in \mathbb{R} \backslash\{1\}$, the operations $\nabla^{n, w(k)} x(k),{ }_{a}^{\rm G} \nabla_{k}^{\alpha, w(k)} x(k)$, ${ }_{a}^{\rm R} \nabla_{k}^{\alpha, w(k)} x(k),{ }_{a}^{\rm C} \nabla_{k}^{\alpha, w(k)} x(k)$ could be abbreviate as $\nabla^{n, \lambda} x(k),{ }_{a}^{\rm G} \nabla_{k}^{\alpha, \lambda} x(k),{ }_{a}^{\rm R} \nabla_{k}^{\alpha, \lambda} x(k)$, ${ }_{a}^{\rm C} \nabla_{k}^{\alpha, \lambda} x(k)$, respectively. Notably, this special case is different from the one in \cite{Ferreira:2021NODY}, which facilitates the use and analysis. \textsf{Compared to existing results, the tempered function $w(k)$ is no longer limited to the exponential function, which makes this work more general and practical.}

By using the linearity, the following lemma can be derived immediately, which is simple while useful for understanding such fractional calculus.
\begin{lemma}\label{Lemma2.2}
For any function $x: \mathbb{N}_{a+1-n} \to \mathbb{R}$, $n \in \mathbb{Z}_{+}$, $a \in \mathbb{R}$, finite nonzero $w(k)$, $\alpha\in(n-1,n)$, $k \in \mathbb{N}_{a+1}$, $\lambda \in\mathbb{R}\backslash\{0\}$, one has
\begin{equation}\label{Eq2.15}
{\textstyle
\left\{ \begin{array}{rl}
{\nabla ^{n,w(k)}}x(k) =&\hspace{-6pt} {\nabla ^{n,\rho w(k)}}x(k),\\
{}_a^{\rm{G}}\nabla _k^{\alpha ,w(k)}x(k) =&\hspace{-6pt} {}_a^{\rm{G}}\nabla _k^{\alpha ,\lambda w(k)}x(k),\\
{}_a^{\rm{G}}\nabla _k^{ - \alpha ,w(k)}x(k) =&\hspace{-6pt} {}_a^{\rm{G}}\nabla _k^{ - \alpha ,\lambda w(k)}x(k),\\
{}_a^{\rm{R}}\nabla _k^{\alpha ,w(k)}x(k) =&\hspace{-6pt} {}_a^{\rm{R}}\nabla _k^{\alpha ,\lambda w(k)}x(k),\\
{}_a^{\rm{C}}\nabla _k^{\alpha ,w(k)}x(k) =&\hspace{-6pt} {}_a^{\rm{C}}\nabla _k^{\alpha ,\lambda w(k)}x(k).
\end{array} \right.}
\end{equation}
\end{lemma}
Note that Lemma \ref{Lemma2.2} is indeed the scale invariance. When $\lambda=-1$, the sign of $w(k)$ is just reversed to $\lambda w(k)$. From this, one is ready to claim that if a property on tempered calculus holds for $w(k)>0$, it also holds for $w(k)<0$.

\section{Main Results}\label{Section3}
In this section, a series of inequalities concerning nabla tempered fractional calculus will be developed.

\begin{theorem}\label{Theorem3.18}
For any $\alpha\in(0,1)$, $w:\mathbb{N}_a\to\mathbb{R}\backslash\{0\}$ with $\frac{w(a)}{w(k)}>0$, $k\in\mathbb{N}_{a+1}$, $a\in\mathbb{R}$, one has
\begin{equation}\label{Eq3.114}
{\textstyle {}_a^{\rm C}\nabla _{k}^{\alpha,w(k)} x(k) \le {}_{a - 1}^{\hspace{9pt}\rm{R}}\nabla _{k}^{\alpha,w(k)} x(k) \le {}_a^{\rm R}\nabla _{k}^{\alpha,w(k)} x(k), x ( a  ) \ge 0,}
\end{equation}
\begin{equation}\label{Eq3.115}
{\textstyle {}_a^{\rm C}\nabla _{k}^{\alpha,w(k)} x(k) \ge {}_{a - 1}^{\hspace{9pt}\rm{R}}\nabla _{k}^{\alpha,w(k)} x(k) \ge {}_a^{\rm R}\nabla _{k}^{\alpha,w(k)} x(k), x ( a  ) \le 0.}
\end{equation}
\end{theorem}
\begin{proof}
Let $z(k) := w(k)x(k)$. By using (\ref{Eq2.3}) and (\ref{Eq2.4}), one has
\begin{equation}\label{Eq3.116}
\begin{array}{rl}
{}_{a - 1}^{\hspace{9pt}\rm{R}}\nabla _k^\alpha z(k) =&\hspace{-6pt} \nabla _{a - 1}^{\hspace{9pt}\rm{G}}\nabla _k^{\alpha  - 1}z(k)\\
 =&\hspace{-6pt} \nabla \sum\nolimits_{j = a}^k {\frac{{ ( {k - j + 1}  )\overline {^{ - \alpha }} }}{{\Gamma  ( {1 - \alpha }  )}}z(j)} \\
 =&\hspace{-6pt} {}_a^{\rm R}\nabla _k^\alpha z(k) + \nabla \frac{{ ( {k - a + 1}  )\overline {^{ - \alpha }} }}{{\Gamma  ( {1 - \alpha }  )}}z ( a  )\\
 =&\hspace{-6pt} {}_a^{\rm R}\nabla _k^\alpha z(k) + \frac{{ ( {k - a + 1}  )\overline {^{ - \alpha  - 1}} }}{{\Gamma  ( { - \alpha }  )}}z ( a  ).
\end{array}
\end{equation}
Combining (\ref{Eq3.116}) and (\ref{Eq2.11}) yields
\begin{equation}\label{Eq3.117}
{\textstyle {}_{a - 1}^{\hspace{9pt}\rm{R}}\nabla _k^{\alpha ,w(k) }x(k) = {}_a^{\rm R}\nabla _k^{\alpha ,w(k) }x(k) + \frac{{ ( {k - a + 1}  )\overline {^{ - \alpha  - 1}} }}{{\Gamma  ( { - \alpha } )}}\frac{w(a)}{w(k)}x ( a  ).}
\end{equation}

From the definition, one has
\begin{equation}\label{Eq3.118}
{\textstyle \begin{array}{l}
{}_a^{\rm{R}}\nabla _k^{\alpha ,w(k)}x(k)\\
 = {\nabla ^{1,w(k)}}{}_a^{\rm{G}}\nabla _k^{\alpha  - 1,w(k)}x(k)\\
 = {w^{ - 1}}(k)\nabla \sum\nolimits_{j = a + 1}^k {\frac{{(k - j + 1)\overline {^{ - \alpha }} }}{{\Gamma ( {1 - \alpha }  )}}} w(j)x(j)\\
 ={w^{ - 1}}(k)\sum\nolimits_{j = a + 1}^k {\frac{{(k - j + 1)\overline {^{ - \alpha }} }}{{\Gamma  ( {1 - \alpha }  )}}\nabla } [w(j)x(j)] + \frac{{(k - a)\overline {^{ - \alpha }} }}{{\Gamma  ( {1 - \alpha } )}}\frac{{w(a)}}{{w(k)}}x(a)\\
 = {w^{ - 1}}(k)\sum\nolimits_{j = a + 1}^k {\frac{{(k - j + 1)\overline {^{ - \alpha }} }}{{\Gamma  ( {1 - \alpha }  )}}w(j){\nabla ^{1,w(j)}}} x(j) + \frac{{(k - a)\overline {^{ - \alpha }} }}{{\Gamma  ( {1 - \alpha }  )}}\frac{{w(a)}}{{w(k)}}x(a)\\
 ={w^{ - 1}}(k){}_a^{\rm{G}}\nabla _k^{\alpha  - 1,w(k)}{\nabla ^{1,w(k)}}x(k) + \frac{{(k - a)\overline {^{ - \alpha }} }}{{\Gamma  ( {1 - \alpha }  )}}\frac{{w(a)}}{{w(k)}}x(a)\\
 ={}_a^{\rm{C}}\nabla _k^{\alpha ,w(k)}x(k) + \frac{{(k - a)\overline {^{ - \alpha }} }}{{\Gamma  ( {1 - \alpha }  )}}\frac{{w(a)}}{{w(k)}}x(a).
\end{array}}
\end{equation}

Under given conditions, it is not difficult to obtain that $\frac{w(a)}{w(k)}>0$, $\frac{{ ( {k - a + 1} )\overline {^{ - \alpha  - 1}} }}{{\Gamma  ( { - \alpha }  )}}=\frac{\Gamma  ( { k-a-\alpha } )}{\Gamma  ( { - \alpha }  )\Gamma  ( { k-a+1}  )}<0$, $\frac{{ ( {k - a}  )\overline {^{ - \alpha }} }}{{\Gamma  ( {1 - \alpha }  )}}=\frac{\Gamma  ( { k-a-\alpha }  )}{\Gamma  ( {1 - \alpha }  )\Gamma  ( { k-a }  )}>0$, $k\in\mathbb{N}_{a+1}$. Besides, the following relationship holds
\begin{equation}\label{Eq3.119}
{\textstyle \frac{{ ( {k - a + 1}  )\overline {^{ - \alpha  - 1}} }}{{\Gamma  ( { - \alpha }  )}} + \frac{{ ( {k - a}  )\overline {^{ - \alpha }} }}{{\Gamma  ( {1 - \alpha }  )}} = \frac{{\Gamma ( {k - a - \alpha  + 1}  )}}{{\Gamma  ( {1 - \alpha }  )\Gamma  ( {k - a+1}  )}} > 0.}
\end{equation}
From this, the magnitude of ${}_a^{\rm C}\nabla _{k}^{\alpha,w(k)} x(k)$, ${}_{a - 1}^{\hspace{9pt}\rm{R}}\nabla _{k}^{\alpha,w(k)} x(k)$ and  ${}_a^{\rm R}\nabla _{k}^{\alpha,w(k)} x(k)$ mainly depend on the sign of $u ( a  )$. Thus, (\ref{Eq3.114}) and (\ref{Eq3.115}) establish immediately.
\end{proof}

When $\frac{w(a)}{w(k)}<0$ is adopted, the sign of inequality should reverse. Additionally, if $w(k)=1$, $k\in\mathbb{N}_{a+1}$, (\ref{Eq3.114}) reduces to \cite[Lemma 2.2]{Eloe:2019BKMS}. The condition $w:\mathbb{N}_a\to\mathbb{R}\backslash\{0\}$ with $\frac{w(a)}{w(k)}>0$ is equivalent to $w(k)>0$ or $w(k)<0$, $\forall k\in\mathbb{N}_a$.

\begin{theorem}\label{Theorem3.19}
If $\nabla^{1,w(k)} x(k)$ is nonnegative, then for any $\alpha\in(0,1)$, $w:\mathbb{N}_a\to\mathbb{R}\backslash\{0\}$ with $\frac{w(a)}{w(k)}>0$, $k\in\mathbb{N}_{a+1}$, $a\in\mathbb{R}$, one has ${}_a^{\rm C}\nabla _k^{\alpha,w(k)}x(k)$ is nonnegative. Besides, if $x(a)\ge0$, then ${}_a^{\rm R}\nabla _k^{\alpha,w(k)}x(k)$ is also nonnegative.
\end{theorem}
\begin{proof}
From the definition, one has
\begin{equation}\label{Eq3.120}
{\textstyle
\begin{array}{rl}
{}_a^{\rm{C}}\nabla _k^{\alpha ,w(k)}x(k) =&\hspace{-6pt} {}_a^{\rm{G}}\nabla _k^{\alpha  - 1,w(k)}{\nabla ^{1,w(k)}}x(k)\\
 =&\hspace{-6pt}\sum\nolimits_{j = a + 1}^k {\frac{{(k - j + 1)\overline {^{ - \alpha }} }}{{\Gamma  ( {1 - \alpha }  )}}} \frac{w(j)}{w(k)}{\nabla ^{1,w(j)}}x(j).
\end{array}}
\end{equation}

Since $\frac{w(j)}{w(k)}>0$, $\nabla^{1,w(k)} x(k)\ge0$, $k\in\mathbb{N}_{a+1}$, and $\frac{{(k - j + 1)\overline {^{ - \alpha }} }}{{\Gamma  ( {1 - \alpha }  )}} = \frac{{\Gamma (k - j + 1 - \alpha )}}{{\Gamma (k - j + 1)\Gamma  ( {1 - \alpha }  )}} > 0$, for any $\alpha\in(0,1)$, $j\in\mathbb{N}_{a+1}^k$, one has ${}_a^{\rm{C}}\nabla _k^{\alpha ,w(k)}x(k)\ge0$.

Under the given conditions, it is not difficult to obtain ${}_a^{\rm{C}}\nabla _k^{\alpha ,w(k)}x(k)\ge0$, $\frac{{ ( {k - a}  )\overline {^{ - \alpha }} }}{{\Gamma  ( {1 - \alpha }  )}}$ $= \frac{{\Gamma (k - a - \alpha )}}{{\Gamma (k - a)\Gamma  ( {1 - \alpha }  )}} > 0$, $\frac{w(a)}{w(k)}>0$, $\forall \alpha\in(0,1)$, $j\in\mathbb{N}_{a+1}^k$. If $x(a)\ge0$, the relationship ${}_a^{\rm R}\nabla _k^{\alpha ,w(k) }x(k) = {}_a^{\rm C}\nabla _k^{\alpha ,w(k) }x(k) + \frac{{ ( {k - a}  )\overline {^{ - \alpha }} }}{{\Gamma  ( {1 - \alpha }  )}}\frac{w(a)}{w(k)}x ( a  )$ will lead to the desired result ${}_a^{\rm{R}}\nabla _k^{\alpha ,w(k)}x(k)\ge0$ immediately.
\end{proof}

Theorem \ref{Theorem3.19} can be regarded as the generalization of subsection 3.18 of \cite{Goodrich:2015Book} and Theorem 21 of \cite{Almeida:2019RMJM}. Actually, the range of the order can be wider. If $\nabla^{n,w(k)} x(k)$ is nonnegative, then for any $\alpha\in(n-1,n)$, $n\in\mathbb{Z}_+$, $w:\mathbb{N}_a\to\mathbb{R}\backslash\{0\}$ with $\frac{w(a)}{w(k)}>0$, $k\in\mathbb{N}_{a+1}$, $a\in\mathbb{R}$, one has ${}_a^{\rm C}\nabla _k^{\alpha,w(k)}x(k)$ is nonnegative. It will be more interesting to construct the sufficient condition for $\nabla^{1, w(k)} x(k) \geq 0$ with the assumption ${ }_{a}^{\mathrm{C}} \nabla_{k}^{\alpha, w(k)} x(k) \geq 0$ or ${ }_{a}^{\mathrm{R}} \nabla_{k}^{\alpha, w(k)} x(k) \geq 0$

\begin{theorem}\label{Theorem3.20}
If ${}_a^{\rm C}\nabla _{k}^{\alpha,w(k)} x(k)\ge{}_a^{\rm C}\nabla _{k}^{\alpha,w(k)} y(k)$ with $x(a)=y(a)$, then for any  $\alpha\in(0,1)$, $w:\mathbb{N}_a\to\mathbb{R}\backslash\{0\}$ with $\frac{w(a)}{w(k)}>0$, $k\in\mathbb{N}_{a+1}$, $a\in\mathbb{R}$, one has $x(k)\ge y(k)$.
\end{theorem}
\begin{proof}
Following from ${}_a^{\rm C}\nabla _{k}^{\alpha,w(k)} x(k)\ge{}_a^{\rm C}\nabla _{k}^{\alpha,w(k)} y(k)$, there must exist a compensation sequence $c(k)\ge 0$ such that
\begin{equation}\label{Eq3.121}
{\textstyle
{}_a^{\rm C}\nabla _k^{\alpha ,w(k) }x(k) = {}_a^{\rm C}\nabla _k^{\alpha ,w(k) }y(k) + c(k).}
\end{equation}
Taking $\alpha$-th Gr\"{u}nwald--Letnikov tempered sum for both sides of (\ref{Eq3.121}) yields
\begin{equation}\label{Eq3.122}
{\textstyle
x(k) - \frac{w(a)}{w(k)}x ( a  ) = y(k) - \frac{w(a)}{w(k)}y ( a  ) + {}_a^{\rm{G}}\nabla _k^{ - \alpha,w(k) }c(k).}
\end{equation}
Since $x(a)=y(a)$, the key item is ${}_a^{\rm{G}}\nabla _k^{ - \alpha,w(k) }c(k)$. According to the nonnegativeness of $c(k)$ and the invariance of the sign of $w(k)$, one has ${}_a^{\rm{G}}\nabla _k^{ - \alpha,w(k) }c(k)\ge0$. As a result, $x(k)\ge y(k)$ is implied in (\ref{Eq3.122}).
\end{proof}

Notably, in Theorem \ref{Theorem3.20}, if $x(a)=y(a)$ is replaced by $x(a)\ge y(a)$, the conclusion still holds. If $x(a)=y(a)$ is replaced by ${[ {{}_a^{\rm R}\nabla _k^{\alpha  - 1,w(k)}x( k )} ]_{k = a}} ={[ {{}_a^{\rm R}\nabla _k^{\alpha  - 1,w(k)}y( k )} ]_{k = a}} $, the conclusion still holds. By using ${}_a^{\rm R}\nabla _k^{\alpha ,w(k) }x(k) = {}_a^{\rm C}\nabla _k^{\alpha ,w(k) }x(k) + \frac{{ ( {k - a} )\overline {^{ - \alpha }} }}{{\Gamma  ( {1 - \alpha }  )}}\frac{w(a)}{w(k)}x ( a  )$, one obtains that if ${}_a^{\rm C}\nabla _{k}^{\alpha,w(k)} x(k)\ge{}_a^{\rm C}\nabla _{k}^{\alpha,w(k)} y(k)$ is replaced by ${}_a^{\rm R}\nabla _{k}^{\alpha,w(k)} x(k)\ge{}_a^{\rm R}\nabla _{k}^{\alpha,w(k)} y(k)$, the conclusion still can be guaranteed. It is generally difficult to obtain the relationship of two sequences' fractional difference and therefore the corresponding inequalities are considered.

\begin{theorem}\label{Theorem3.21}
If ${}_a^{\rm C}\nabla _k^{\alpha,w(k)} x(k)\ge\mu x(k)+\gamma $, ${}_a^{\rm C}\nabla _k^{\alpha,w(k)} y(k)\le\mu y(k)+\gamma$, where $x(a)=y(a)$, then for any $\alpha\in(0,1)$, $w:\mathbb{N}_a\to\mathbb{R}\backslash\{0\}$ with $\frac{w(a)}{w(k)}>0$, $\mu<0$, $\gamma\in\mathbb{R}$, $k\in {\mathbb N}_{a+1}$, $a\in\mathbb{R}$, one has $x(k)\ge y(k)$.
\end{theorem}
\begin{proof}
For ${}_a^{\rm C}\nabla _k^{\alpha,w(k)} u(k)\ge\mu x(k)+\gamma $, there must exist a compensation sequence $c_1(k) \ge 0$ such that
\begin{equation}\label{Eq3.123}
{\textstyle
{}_a^{\rm C}\nabla _k^{\alpha,w(k)} x(k)= \mu x(k)+\gamma +c_1(k).}
\end{equation}

Letting $\hat x(k)=w(k)x(k)$, $\hat c_1(k)=w(k)c_1(k)$, then one has ${}_a^{\rm C}\nabla _k^{\alpha} \hat x(k)= \mu \hat x(k)+\gamma w(k)+\hat c_1(k)$. Therefore, by using the nabla linear system theory, it follows
\begin{equation}\label{Eq3.124}
{\textstyle
\begin{array}{rl}
\hat x(k) =&\hspace{-6pt} \hat x ( a  ){{\mathcal F}_{\alpha ,1}} ( {\mu,k,a}  )+ {\rm \gamma}w(k)\ast {{\mathcal F}_{\alpha ,\alpha}} ( {\mu,k,a}  )\\
 &\hspace{-6pt}+ \hat c_1(k) \ast {{\mathcal F}_{\alpha ,\alpha}} ( {\mu,k,a}  ).
\end{array}}
\end{equation}

In a similar way, letting $\hat y(k)=w(k)y(k)$, $\hat c_2(k)=w(k)c_2(k)$ where $c_2(k)$ is a nonnegative compensation sequence, then one has $\hat y(a)=\hat x(a)$ and ${}_a^{\rm C}\nabla _k^{\alpha} \hat y(k) = \mu \hat y(k)+\gamma w(k)-\hat c_2(k)$, which implies
\begin{equation}\label{Eq3.125}
{\textstyle
\begin{array}{rl}
\hat y(k) =&\hspace{-6pt} \hat y ( a  ){{\mathcal F}_{\alpha ,1}} ( {\mu,k,a}  )+ {\rm \gamma}w(k)\ast {{\mathcal F}_{\alpha ,\alpha}} ( {\mu,k,a}  )\\
&\hspace{-6pt}- \hat c_2(k) \ast {{\mathcal F}_{\alpha ,\alpha}} ( {\mu,k,a}  ).
\end{array}}
\end{equation}

Since ${{\mathcal F}_{\alpha ,1}} ( {\mu,k,a}  )>0$, ${{\mathcal F}_{\alpha ,\alpha}} ( {\mu,k,a} )>0$ hold for any $\alpha\in(0,1)$, $\mu<0$, $k\in {\mathbb N}_{a+1}$, $a\in\mathbb{R}$, when $w(k)>0$, one has $\hat c_1(k)\ge0$, $\hat c_2(k)\ge0$ and $\hat x(k)\ge\hat y(k)$, which implies $ x(k)\ge y(k)$. When $w(k)<0$, one has $\hat c_1(k)\le0$, $\hat c_2(k)\le0$ and $\hat x(k)\le\hat y(k)$. In a similar way, the desired result $ x(k)\ge y(k)$ follows. The case of $w(k)<0$ can also be proved by applying Lemma \ref{Lemma2.2}.
\end{proof}

To be more practical, the following nonlinear case is discussed.
\begin{theorem}\label{Theorem3.22}
If ${}_a^{\rm{C}}\nabla _k^{\alpha,w(k)} x(k) \ge  - \gamma  ( {x(k)}  )+h(k)$, ${}_a^{\rm{C}}\nabla _k^{\alpha,w(k)} y(k) \le  - \gamma  ( {y(k)}  )$ $+h(k)$, where $x(a)=y(a)$, $x(k), y(k) \ge 0$, then for any $\alpha\in(0,1)$, $w:\mathbb{N}_a\to\mathbb{R}\backslash\{0\}$ with $\frac{w(a)}{w(k)}>0$, locally Lipschitz class $\mathcal K$ function $\gamma$, bounded $h(k)$, $k\in {\mathbb N}_{a+1}$, $a\in\mathbb{R}$, one has $x(k) \ge y(k)$.
\end{theorem}
\begin{proof}
When $w(k)>0$, by combining (\ref{Eq2.12}) and the given conditions, one has
\begin{equation}\label{Eq3.126}
{\textstyle {}_a^{\rm{C}}\nabla _k^\alpha [w(k)x(k)] \ge  - w(k)\gamma (x(k)) + w(k)h(k), }
\end{equation}
\begin{equation}\label{Eq3.127}
{\textstyle {}_a^{\rm{C}}\nabla _k^\alpha [w(k)y(k)] \le  - w(k)\gamma (y(k)) + w(k)h(k). }
\end{equation}

By combining (\ref{Eq3.126}) and (\ref{Eq3.127}), one has
\begin{equation}\label{Eq3.128}
{\textstyle {}_a^{\rm{C}}\nabla _k^\alpha [w(k)y(k)-w(k)x(k)] \le  - w(k)[\gamma (y(k)) - \gamma (x(k))]. }
\end{equation}

By substituting $k=a+1$ into (\ref{Eq3.128}), one has $w(a+1)y(a+1)-w(a+1)x(a+1) \le  - w(a+1)[\gamma (y(a+1)) - \gamma (x(a+1))]$ which implies
\begin{equation}\label{Eq3.129}
{\textstyle y ( a+1  )+\gamma  ( {y ( a+1  )}  )\le x ( a+1  )+\gamma  ( {x ( a+1 )}  ). }
\end{equation}
The monotonicity of $f(z)=z+\gamma(z)$, $z\ge0$ results in $y ( a+1  )\le x ( a+1  )$.

Now, let us assume that there exist a constant $k_1\in\mathbb{N}_{a+1}$ such that
\begin{equation}\label{Eq3.130}
 \left\{ \begin{array}{l}
y(k) \le x(k),k \in \mathbb{N}_a^{{k_1} - 1},\\
y(k) > x(k), k=k_1.
\end{array}  \right.
\end{equation}

From (\ref{Eq3.130}), the definition of Caputo fractional difference and the formula of summation by parts give
\begin{equation}\label{Eq3.131}
\begin{array}{l}
\{{}_a^{\rm{C}}\nabla _k^\alpha [w(k)y(k)-w(k)x(k)]\}_{k = {k_1}}\\
 = \sum\nolimits_{j = a + 1}^{{k_1}} {\frac{{ ( {{k_1} - j + 1}  )\overline {^{ - \alpha }} }}{{\Gamma ( {1 - \alpha }  )}}\{ {\nabla [w(j)y(j)] - \nabla [w(j)x(j)]} \}} \\
 = \sum\nolimits_{j = a + 1}^{{k_1}} {\frac{{ ( {{k_1} - j + 1}  )\overline {^{ - \alpha }} }}{{\Gamma ( {1 - \alpha }  )}} \nabla \{w(j)[y(j)- x(j)] \}} \\
 = \frac{{ ( {{k_1} - j}  )\overline {^{ - \alpha }} }}{{\Gamma  ( {1 - \alpha }  )}}w(j) [ {y(j) - x(j)}  ]\big|_{j = a}^{j = {k_1}}\\
 \hspace*{12pt}+ \sum\nolimits_{j = a + 1}^{{k_1}} {\frac{{ ( {{k_1} - j + 1}  )\overline {^{ - \alpha  - 1}} }}{{\Gamma  ( { - \alpha }  )}}w(j) [ {y (j  ) - x (j  )}  ]} \\
 =  \sum\nolimits_{j = a + 1}^{{k_1-1}} {\frac{{ ( {{k_1} - j + 1}  )\overline {^{ - \alpha  - 1}} }}{{\Gamma  ( { - \alpha }  )}}w(j) [ {y (j  ) - x (j  )}  ]} \\
 \hspace*{12pt}+ w(k_1)[y(k_1)- x(k_1)]\\
 > 0,
\end{array}
\end{equation}
where $\frac{{0\overline {^{ - \alpha }} }}{{\Gamma  ( {1 - \alpha }  )}}=0$, $x(a)=y(a)$, $\frac{{ ( {{k_1} - j + 1}  )\overline {^{ - \alpha  - 1}} }}{{\Gamma  ( { - \alpha }  )}}=\frac{\Gamma  ( { k_1-j-\alpha}  )}{\Gamma  ( { k_1-j+1 }  )\Gamma  ( { - \alpha }  )}<0$, $i\in\mathbb{N}_{a+1}^{k_1-1}$ and $w(k)>0$, $k\in\mathbb{N}_{a+1}$ are adopted.

By using (\ref{Eq3.128}) and (\ref{Eq3.130}), one has
\begin{equation}\label{Eq3.132}
\begin{array}{l}
\{{}_a^{\rm{C}}\nabla _k^\alpha [w(k)y(k)-w(k)x(k)]\}_{k = {k_1}}\\
 \le  - w(k_1)[\gamma (y(k_1)) - \gamma (x(k_1))]\\
 \le 0,
\end{array}
\end{equation}
which contradicts (\ref{Eq3.131}). Consequently, $k_1$ does not exist. From this, it is not difficult to obtain the desired result $x(k) \ge y(k)$ for all $k\in {\mathbb N}_{a+1}$.

When $w(k)<0$, $\forall k\in {\mathbb N}_{a+1}$, the desired result $x(k) \ge y(k)$, $\forall k\in {\mathbb N}_{a+1}$ can also be obtained by applying Lemma \ref{Lemma2.2}.
\end{proof}

The comparison principle in Theorem \ref{Theorem3.23} - Theorem \ref{Theorem3.25} are inspired by Lemma 2.10 in \cite{Baleanu:2017CNSNS}, Lemma 3.14 in \cite{Liu:2021Filomat}, Lemma 3.1 in \cite{Liu:2019TJM}, Lemma 3.4 in \cite{Wu:2017AMC}, Lemma 2.16 in \cite{Wyrwas:2015Kybe}, Lemma 3.3 in \cite{Liu:2018DSA}, etc. Compared with the existing ones, the newly developed theorems are more practical.

\begin{theorem}\label{Theorem3.23}
If $W: \mathbb R^{n} \to \mathbb{R}$ is differentiable convex, $W ( {0}  )=0$, $V(x(k)):=w^{-1}(k)W(z(k))$, $z(k):=w(k)x(k)$ and $\frac{{\rm d} W ( {z(k)}  )}{{{\rm d} z(k)}}=\frac{{\rm d}V ( {x(k)}  ) }{{{\rm d} x(k)}}$ holds almost everywhere, then for any $\alpha\in(0,1)$, $w:\mathbb{N}_a\to\mathbb{R}\backslash\{0\}$ with $\frac{w(a)}{w(k)}>0$, $k \in {\mathbb N_{a + 1}}$, $a\in\mathbb{R}$, one has
\begin{equation}\label{Eq3.133}
{\textstyle {}_a^{\rm C}\nabla _{k}^{\alpha,w(k)} V ( {x(k)}  ) \le \frac{{{\rm d} V ( {x(k)} )}}{{{\rm d}  {x^{\rm{T}}}(k)}}{}_a^{\rm C}\nabla _{k}^{\alpha,w(k)} x(k),}
\end{equation}
\begin{equation}\label{Eq3.134}
{\textstyle {}_a^{\rm R}\nabla _{k}^{\alpha,w(k)} V ( {x(k)}  ) \le \frac{{{\rm d}  V ( {x(k)} )}}{{{\rm d}  {x^{\rm{T}}}(k)}}{}_a^{\rm R}\nabla _{k}^{\alpha,w(k)} x(k),}
\end{equation}
\begin{equation}\label{Eq3.135}
{\textstyle {}_a^{\rm G}\nabla _{k}^{\alpha,w(k)} V ( {x(k)}  ) \le \frac{{{\rm d}  V ( {x(k)} )}}{{{\rm d}  {x^{\rm{T}}}(k)}}{}_a^{\rm G}\nabla _{k}^{\alpha,w(k)} x(k).}
\end{equation}
\end{theorem}
\begin{proof}
Similarly, only the case of $w(k)>0$ is considered here. Owing to the differentiable convex condition, one has $W ( {z(j)}  ) - W ( {z(k)}  ) - \frac{{{\rm d} W ( {z(k)}  )}}{{{\rm d}  {z^{\rm{T}}}(k)}} [ {z(j) - z(k)}  ]\ge0$, $j\in\mathbb{N}_{a}^{k-1}$, $k\in\mathbb{N}_{a+1}$.

On the basis of these conclusions, the subsequent proof of this theorem will be divided into three parts.

$\blacktriangleright$ Part 1: the Caputo case

By using (\ref{Eq2.3}), (\ref{Eq2.12}) and the formula of summation by parts, it follows
\begin{equation}\label{Eq3.136}
\begin{array}{l}
w(k)[{}_a^{\rm C}\nabla _{k}^{\alpha,w(k)} V ( {x(k)}  ) - \frac{{{\rm d} V ( {x(k)}  )}}{{{\rm d}  {x^{\rm{T}}}(k)}}{}_a^{\rm C}\nabla _{k}^{\alpha,w(k)} x(k)]\\
={}_a^{\rm C}\nabla _k^\alpha W ( {z(k)}  ) - \frac{{{\rm d} W ( {z(k)}  )}}{{{\rm d}  {z^{\rm{T}}}(k)}}{}_a^{\rm C}\nabla _k^\alpha z(k)\\
 = {}_a^{\rm G}\nabla _k^{\alpha  - 1}\big[ {\nabla W ( {z(k)}  ) - \frac{{{\rm d} W ( {z(k)} )}}{{{\rm d}  {z^{\rm{T}}}(k)}}\nabla z(k)} \big]\\
 = \sum\nolimits_{j = a + 1}^k {\frac{{ ( {k - j + 1}  )\overline {^{ - \alpha }} }}{{\Gamma  ( {1 - \alpha }  )}}} \big[ {\nabla W ( {z(j)}  ) - \frac{{{\rm d} W ( {z(k)}  )}}{{{\rm d}  {z^{\rm{T}}}(k)}}\nabla z(j)} \big]\\
 = \sum\nolimits_{j = a + 1}^k {f ( {j - 1}  )\nabla g(j)} \\
 =  {f ( {j}  )g(j)}  |_{j = a}^{j = k} - \sum\nolimits_{j = a + 1}^k {\nabla f(j)g(j)},
\end{array}
\end{equation}
where $f(j) := \frac{{ ( {k - j}  )\overline {^{ - \alpha }} }}{{\Gamma  ( {1 - \alpha }  )}}$, $g(j): = W ( {z(j)}  ) - W ( {z(k)}  ) - \frac{{{\rm d} W ( {z(k)}  )}}{{{\rm d}  {z^{\rm{T}}}(k)}} [ {z(j) - z(k)}  ]$.

By using the property of rising function and taking first order difference with respect to the variable $j$, one has
$ \nabla f(j) =  - \frac{{ ( {k - j + 1}  )\overline {^{ - \alpha  - 1}} }}{{\Gamma  ( { - \alpha } )}}$. According to the sign of Gamma function, one has $\nabla f(j)\ge0$, $j\in\mathbb{N}_{a+1}^{k-1}$, $[\nabla f(j)]_{j=k}=-1$, $f ( a  )>0$ and $f(k)=0$. Due to the differentiable convex property of $W ( z(k)  )$ with regard to $z(k)$ (see \cite{Zhang:2022JO}), one obtains $g(j)\ge0$, $j\in\mathbb{N}_{a}^{k-1}$ and $g(k)=0$. Then the equation in (\ref{Eq3.139}) can be rewritten as
\begin{equation}\label{Eq3.137}
{\textstyle \begin{array}{l}
w(k)[{}_a^{\rm C}\nabla _{k}^{\alpha,w(k)} V ( {x(k)}  ) - \frac{{{\rm d} V ( {x(k)}  )}}{{{\rm d}  {x^{\rm{T}}}(k)}}{}_a^{\rm C}\nabla _{k}^{\alpha,w(k)} x(k)]\\
 =  - f(a)g ( a  ) - \sum\nolimits_{j = a + 1}^{k - 1} {\nabla f(j)g(j)}\\
 \le0.
\end{array}}
\end{equation}
Because of the positivity of $w(k)$, (\ref{Eq3.137}) implies (\ref{Eq3.133}) firmly.

$\blacktriangleright$ Part 2: the Riemann--Liouville case

By using the derived result in (\ref{Eq3.137}) and the property in (\ref{Eq2.11}), one has
\begin{equation}\label{Eq3.138}
\begin{array}{l}
w(k)[{}_a^{\rm R}\nabla _{k}^{\alpha,w(k)} V ( {x(k)}  ) - \frac{{{\rm d} V ( {x(k)}  )}}{{{\rm d}  {x^{\rm{T}}}(k)}}{}_a^{\rm R}\nabla _{k}^{\alpha,w(k)} x(k)]\\
={}_a^{\rm R}\nabla _k^\alpha W ( {z(k)}  ) - \frac{{{\rm d} W ( {z(k)}  )}}{{{\rm d}  {z^{\rm{T}}}(k)}}{}_a^{\rm R}\nabla _k^\alpha z(k)\\
={}_a^{\rm C}\nabla _k^\alpha W ( {z(k)}  ) - \frac{{{\rm d} W ( {z(k)}  )}}{{{\rm d}  {z^{\rm{T}}}(k)}}{}_a^{\rm C}\nabla _k^\alpha z(k)\\
   \hspace{12pt}+ \frac{{ ( {k - a}  )\overline {^{ - \alpha }} }}{{\Gamma  ( {1 - \alpha }  )}}W ( {z ( a  )}  )-\frac{{{\rm d} W ( {z(k)}  )}}{{{\rm d}  {z^{\rm{T}}}(k)}}\frac{{ ( {k - a} )\overline {^{ - \alpha }} }}{{\Gamma  ( {1 - \alpha }  )}}z ( a  )\\
 = - \sum\nolimits_{j = a + 1}^{k-1} {\nabla f(j)g(j)}+ f(a)[W ( {z(k)}  )-\frac{{{\rm d} W ( {z(k)} )}}{{{\rm d}  {z^{\rm{T}}}(k)}}z(k)].
\end{array}
\end{equation}

Letting $z(j)=0$ in $g(j)$, $W ( {0}  )=0$ gives that $W ( {z(k)} )-\frac{{{\rm d} W ( {x(k)}  )}}{{{\rm d}  {z^{\rm{T}}}(k)}}z(k)\le0$. By applying (\ref{Eq3.137}), (\ref{Eq3.138}) and $w(k)>0$, it follows
\begin{equation}\label{Eq3.139}
{\textstyle
{}_a^{\rm R}\nabla _{k}^{\alpha,w(k)} V ( {x(k)}  ) - \frac{{{\rm d} V ( {x(k)}  )}}{{{\rm d}  {x^{\rm{T}}}(k)}}{}_a^{\rm R}\nabla _{k}^{\alpha,w(k)} x(k)\le0.}
\end{equation}

$\blacktriangleright$ Part 3: the Gr\"{u}nwald--Letnikov case

Based on (\ref{Eq2.3}), one has
\begin{equation}\label{Eq3.140}
\textstyle{\begin{array}{l}
w(k)[{}_a^{\rm G}\nabla _{k}^{\alpha,w(k)} V ( {x(k)}  ) - \frac{{{\rm d} V ( {x(k)}  )}}{{{\rm d}  {x^{\rm{T}}}(k)}}{}_a^{\rm G}\nabla _{k}^{\alpha,w(k)} x(k)]\\
={}_a^{\rm G}\nabla _k^\alpha W ( {z(k)}  ) - \frac{{{\rm d} W ( {z(k)}  )}}{{{\rm d}  {z^{\rm{T}}}(k)}}{}_a^{\rm G}\nabla _k^\alpha z(k)\\
 = \sum\nolimits_{j = a + 1}^k {\frac{{(k - j + 1)\overline {^{ - \alpha  - 1}} }}{{\Gamma ( - \alpha )}}\big[ {W ( {z(j)}  ) - \frac{{{\rm d} W ( {z(k)}  )}}{{{\rm d}  {z^{\rm{T}}}(k)}}z(j)} \big].}
\end{array}}
\end{equation}

Since it is difficult to judge the sign of ${W ( {z(j)}  ) - \frac{{{\rm d} W ( {z(k)}  )}}{{{\rm d}  {z^{\rm{T}}}(k)}}z(j)}$, an effective way is to use the skills of magnifying and shrinking of inequality, transforming the unknown unfamiliar case into the known familiar case. As such, $W ( {0}  )=0$ and $W ( {z(k)}  )-\frac{{{\rm d} W ( {z(k)}  )}}{{{\rm d}  {z^{\rm{T}}}(k)}}z(k)\le0$ can be guaranteed. Since $\nabla \frac{{(k - j)\overline {^{ - \alpha }} }}{{\Gamma (1 - \alpha )}} =  - \frac{{(k - j + 1)\overline {^{ - \alpha  - 1}} }}{{\Gamma ( - \alpha )}}$, it becomes ${\sum\nolimits_{j = a + 1}^k {\frac{{(k - j + 1)\overline {^{ - \alpha  - 1}} }}{{\Gamma ( - \alpha )}}}  = \frac{{(k - a)\overline {^{ - \alpha }} }}{{\Gamma (1 - \alpha )}} \ge 0}$. Applying those properties in (\ref{Eq3.136}), the following desired results can be obtained
\begin{equation}\label{Eq3.141}
\begin{array}{l}
w(k)[{}_a^{\rm G}\nabla _{k}^{\alpha,w(k)} V ( {x(k)}  ) - \frac{{{\rm d} V ( {x(k)}  )}}{{{\rm d}  {x^{\rm{T}}}(k)}}{}_a^{\rm G}\nabla _{k}^{\alpha,w(k)} x(k)]\\
 \le \sum\nolimits_{j = a + 1}^k {\frac{{(k - j + 1)\overline {^{ - \alpha  - 1}} }}{{\Gamma ( - \alpha )}}\big[ {W ( {z(j)}  ) -\frac{{{\rm d} W ( {z(k)}  )}}{{{\rm d}  {z^{\rm{T}}}(k)}}z(j)} \big]} \\
  \hspace{12pt}- \frac{{(k - a)\overline {^{ - \alpha }} }}{{\Gamma (1 - \alpha )}}\big[ {W ( {z(k)}  ) - \frac{{{\rm d} W ( {z(k)}  )}}{{{\rm d}  {z^{\rm{T}}}(k)}}z(k)} \big]\\
 =  \sum\nolimits_{j = a + 1}^k {\frac{{(k - j + 1)\overline {^{ - \alpha  - 1}} }}{{\Gamma ( - \alpha )}}g (j)} \\
 =  - \sum\nolimits_{j = a + 1}^{k-1} {\nabla f(j)g(j)}- g(k)\\
 \le 0,
\end{array}
\end{equation}
which implies (\ref{Eq3.135}). Till now, the proof has been completed.
\end{proof}

Theorem \ref{Theorem3.23} is inspired by \cite[Theorem 1]{Chen:2017CTA} and \cite[Lemma 1]{Badri:2019CTA}. Note that the condition $W ( {0}  )=0$ is not necessary for the Caputo case. In Theorem \ref{Theorem3.23}, the considered function $W(z(k))$ is general and it can be Volterra function, power function, logarithmic function, integral upper limit function and tangent function, etc. For multivariate composite functions, the following corollary follows.

\begin{corollary}\label{Corollary3.2}
If $W: \mathbb R^{p} \times \mathbb R^{q} \to \mathbb{R}$ is differentiable convex, $W ( {0,0} )=0$,  $V(x(k)):=w^{-1}(k)W(\hat u(k),\hat v(k))$, $\hat u(k):=w(k)u(k)$, $\hat v(k):=w(k)v(k)$, $\frac{{\partial W ( {\hat u(k),\hat v(k)} )}}{{\partial \hat u(k)}}$ $=\frac{{\partial V ( {u(k),v(k)} )}}{{\partial u(k)}}$ and $\frac{{\partial W ( {\hat u(k),\hat v(k)} )}}{{\partial \hat u(k)}}=\frac{{\partial V ( {u(k),v(k)} )}}{{\partial v(k)}}$ hold almost everywhere, then for any $\alpha\in(0,1)$, $w:\mathbb{N}_a\to\mathbb{R}\backslash\{0\}$ with $\frac{w(a)}{w(k)}>0$, $k \in {\mathbb N_{a + 1}}$, $a\in\mathbb{R}$, one has
\begin{equation}\label{Eq3.142}
\textstyle{
\begin{array}{rl}
{}_a^{\rm C}\nabla _k^{\alpha,w(k)} V ( {u(k),v(k)}  ) \le&\hspace{-6pt} \frac{{\partial V ( {u(k),v(k)} )}}{{\partial u^{\rm T}(k)}}{}_a^{\rm C}\nabla _k^{\alpha,w(k)} u(k)\\
&\hspace{-6pt}+\frac{{\partial V ( {u(k),v(k)}  )}}{{\partial v^{\rm T}(k)}}{}_a^{\rm C}\nabla _k^{\alpha,w(k)} v(k),
\end{array}}
\end{equation}
\begin{equation}\label{Eq3.143}
\textstyle{
\begin{array}{rl}
{}_a^{\rm R}\nabla _k^{\alpha,w(k)} V ( {u(k),v(k)}  ) \le&\hspace{-6pt} \frac{{\partial V ( {u(k),v(k)} )}}{{\partial u^{\rm T}(k)}}{}_a^{\rm R}\nabla _k^{\alpha,w(k)} u(k)\\
&\hspace{-6pt}+\frac{{\partial V ( {u(k),v(k)}  )}}{{\partial v^{\rm T}(k)}}{}_a^{\rm R}\nabla _k^{\alpha,w(k)} v(k),
\end{array}}
\end{equation}
\begin{equation}\label{Eq3.144}
\textstyle{
\hspace{-6pt}\begin{array}{rl}
{}_a^{\rm G}\nabla _k^{\alpha,w(k)} V ( {u(k),v(k)}  ) \le&\hspace{-6pt} \frac{{\partial V ( {u(k),v(k)} )}}{{\partial u^{\rm T}(k)}}{}_a^{\rm G}\nabla _k^{\alpha,w(k)} u(k)\\
&\hspace{-6pt}+\frac{{\partial V ( {u(k),v(k)}  )}}{{\partial v^{\rm T}(k)}}{}_a^{\rm G}\nabla _k^{\alpha,w(k)} v(k).
\end{array}}
\end{equation}
\end{corollary}

\textsf{In Theorem \ref{Theorem3.23}, $W: \mathbb{R}^n \to \mathbb{R}$ is assume to be differentiable convex. If it is convex but not differentiable, the following theorem can be developed.}

\begin{theorem}\label{Theorem3.24}
If $W: \mathbb R^{n} \to \mathbb{R}$ is convex, $W ( {0}  )=0$, $z(k):=w(k)x(k)$, $V(x(k)):=w^{-1}(k)W(w(k)x(k))$ and $\zeta(z(k))=\zeta(x(k))$ holds almost everywhere, then for any $\alpha\in(0,1)$, $w:\mathbb{N}_a\to\mathbb{R}\backslash\{0\}$ with $\frac{w(a)}{w(k)}>0$, $k \in {\mathbb N_{a + 1}}$, $a\in\mathbb{R}$, one has
\begin{equation}\label{Eq3.145}
{\textstyle {}_a^{\rm C}\nabla _k^{\alpha,w(k)} V ( {x(k)}  ) \le \zeta^{\rm T}(x(k)){}_a^{\rm C}\nabla _k^{\alpha,w(k)} x(k),}
\end{equation}
\begin{equation}\label{Eq3.146}
{\textstyle {}_a^{\rm R}\nabla _k^{\alpha,w(k)} V ( {x(k)}  ) \le \zeta^{\rm T}(x(k)){}_a^{\rm R}\nabla _k^{\alpha,w(k)} x(k),}
\end{equation}
\begin{equation}\label{Eq3.147}
{\textstyle {}_a^{\rm G}\nabla _k^{\alpha,w(k)} V ( {x(k)}  ) \le \zeta^{\rm T}(x(k)){}_a^{\rm G}\nabla _k^{\alpha,w(k)} x(k),}
\end{equation}
where $\zeta(z(k))\in\partial W(z(k))$ and $\zeta(x(k))\in\partial V(x(k))$.
\end{theorem}
\begin{proof}
The main idea of this proof is also to construct non-negative terms by using the properties of convex functions. Different from Theorem \ref{Theorem3.23}, letting $g(j): = W ( {z(j)}  ) - W ( {z(k)} ) - \zeta^{\rm T}(z(k)) [ {z(j) - z(k)}  ]$. By using the basic property of subgradient \cite{Zhang:2022JO}, one finds $g(j)\ge0$, $j\in{\mathbb N_{a}^k}$ and $g(k)=0$ when $w(k)>0$, which coincide with the property of $g(j)$ in Theorem \ref{Theorem3.23}. Along this way, the proof could be completed similarly. For simplity, it is omitted here.
\end{proof}

Theorem \ref{Theorem3.24} is inspired by \cite[Theorem 2]{Zhang:2015NAHS}, \cite[Lemma 10]{Munoz:2018TIMC} and \cite[Theorem 3, Theorem 7]{Wei:2021NODY}. When $W ( {z(k)}  )$ is differentiable, the subgradient $\zeta(z(k))$ is strengthened as the true gradient $\frac{{{\rm d} W ( {z(k)}  )}}{{{\rm d} z(k)}} $. Therefore, Theorem \ref{Theorem3.23} is a special case of Theorem \ref{Theorem3.24}. When $n=1$, $w(k)=1$, $V ( {x(k)}  )=|x(k)|$, it follows ${\mathop{\rm sgn}}  ( {x(k)}  )\in\partial V(x(k))$. At this point, the following corollary can be obtained.

\begin{corollary}\label{Corollary3.3}
For any $x: \mathbb{N}_{a+1} \to \mathbb{R}$, $\alpha\in(0,1)$, $w:\mathbb{N}_a\to\mathbb{R}\backslash\{0\}$ with $\frac{w(a)}{w(k)}>0$, $k \in {\mathbb N_{a + 1}}$, $a\in\mathbb{R}$, one has
\begin{equation}\label{Eq3.148}
{\textstyle {}_a^{\rm C}\nabla _k^{\alpha,w(k)} |x(k)| \le {\rm sgn}(x(k)){}_a^{\rm C}\nabla _k^{\alpha,w(k)} x(k),}
\end{equation}
\begin{equation}\label{Eq3.149}
{\textstyle {}_a^{\rm R}\nabla _k^{\alpha,w(k)} |x(k)| \le {\rm sgn}(x(k)){}_a^{\rm R}\nabla _k^{\alpha,w(k)} x(k),}
\end{equation}
\begin{equation}\label{Eq3.150}
{\textstyle {}_a^{\rm G}\nabla _k^{\alpha,w(k)} |x(k)| \le {\rm sgn}(x(k)){}_a^{\rm G}\nabla _k^{\alpha,w(k)} x(k).}
\end{equation}
\end{corollary}
Notably, the complex condition of $W(z(k))$ in Theorem \ref{Theorem3.23} and Theorem \ref{Theorem3.24} has been removed in Corollary \ref{Corollary3.3}. It would be more practical to reduce the complexity of the given condition for previous theorems. Also, the scalar case can be extended to the vector case. Along this way, the fractional difference inequalities on $\|x(k)\|_1$ or $\|x(k)\|_\infty$ can be developed.

\textsf{If it is differentiable but not convex, three cases will be discussed, i.e., the concave function, the monotone function and the synchronous function. The latter two cases are neither convex nor concave.}
\begin{theorem}\label{Theorem3.25}
If $W: \mathbb R^{n} \to \mathbb{R}$ is differentiable concave, $W ( {0}  )=0$, $V(x(k)):=w^{-1}(k)W(z(k))$, $z(k):=w(k)x(k)$ and $\frac{{\rm d} W ( {z(k)}  )}{{{\rm d} z(k)}}=\frac{{\rm d}V ( {x(k)}  ) }{{{\rm d} x(k)}}$ holds almost everywhere, then for any $\alpha\in(0,1)$, $w:\mathbb{N}_a\to\mathbb{R}\backslash\{0\}$ with $\frac{w(a)}{w(k)}>0$, $k \in {\mathbb N_{a + 1}}$, $a\in\mathbb{R}$, one has
\begin{equation}\label{Eq3.151}
{\textstyle {}_a^{\rm C}\nabla _{k}^{\alpha,w(k)} V ( {x(k)}  ) \ge \frac{{{\rm d} V ( {x(k)} )}}{{{\rm d}  {x^{\rm{T}}}(k)}}{}_a^{\rm C}\nabla _{k}^{\alpha,w(k)} x(k),}
\end{equation}
\begin{equation}\label{Eq3.152}
{\textstyle {}_a^{\rm R}\nabla _{k}^{\alpha,w(k)} V ( {x(k)}  ) \ge \frac{{{\rm d}  V ( {x(k)} )}}{{{\rm d}  {x^{\rm{T}}}(k)}}{}_a^{\rm R}\nabla _{k}^{\alpha,w(k)} x(k),}
\end{equation}
\begin{equation}\label{Eq3.153}
{\textstyle {}_a^{\rm G}\nabla _{k}^{\alpha,w(k)} V ( {x(k)}  ) \ge \frac{{{\rm d}  V ( {x(k)} )}}{{{\rm d}  {x^{\rm{T}}}(k)}}{}_a^{\rm G}\nabla _{k}^{\alpha,w(k)} x(k).}
\end{equation}
\end{theorem}
\begin{proof}
Define $f(j)$, $W(z(k))$ and $z(k)$ like the proof of Theorem \ref{Theorem3.23}. Let $g(j): = W ( {z(j)} ) - W ( {z(k)}  ) - \frac{{{\rm d} W ( {z(k)}  )}}{{{\rm d}  {z^{\rm{T}}}(k)}} [ {z(j) - z(k)}  ]$. When the convex condition is replaced by the concave condition, $g(j)\le0$, $j\in{\mathbb N_{a}^k}$ and $g(k)=0$. In a similar way, the inequalities in (\ref{Eq3.151}) - (\ref{Eq3.153}) can be derived.
\end{proof}

Notably, to remove the coupling and prettify the conclusion, the complicated conditions are introduced in Theorem \ref{Theorem3.23} - Theorem \ref{Theorem3.25}. In future study, a plain condition on $V(x(k))$ is expected.

\begin{theorem}\label{Theorem3.26}
If the function $\phi: {\mathbb N_{a + 1}} \to \mathbb{R}$ is monotonically decreasing and $x: {\mathbb N_{a + 1}} \to \mathbb{R}^n$ is non-negative, then for any $\alpha\in(0,1)$, $w:\mathbb{N}_a\to\mathbb{R}\backslash\{0\}$ with $\frac{w(a)}{w(k)}>0$, $k \in {\mathbb N_{a + 1}}$, $a\in\mathbb{R}$, one has
\begin{equation}\label{Eq3.154}
\textstyle{{}_a^{\rm C}\nabla _k^{\alpha,w(k)} [\phi^{\rm T}(k)x(k)] \le \phi^{\rm T}(k){}_a^{\rm C}\nabla _k^{\alpha,w(k)} x(k),}
\end{equation}
\begin{equation}\label{Eq3.155}
\textstyle{{}_a^{\rm R}\nabla _k^{\alpha,w(k)} [\phi^{\rm T}(k)x(k)] \le \phi^{\rm T}(k){}_a^{\rm R}\nabla _k^{\alpha,w(k)} x(k),}
\end{equation}
\begin{equation}\label{Eq3.156}
\textstyle{{}_a^{\rm G}\nabla _k^{\alpha,w(k)} [\phi^{\rm T}(k)x(k)] \le \phi^{\rm T}(k){}_a^{\rm G}\nabla _k^{\alpha,w(k)} x(k).}
\end{equation}
\end{theorem}
\begin{proof}
Setting $w(k)>0$ and defining $f(j) := \frac{{ ( {k - j}  )\overline {^{ - \alpha }} }}{{\Gamma  ( {1 - \alpha }  )}}$ as previous, it follows $\nabla f(j)\ge0$, $j\in\mathbb{N}_{a+1}^{k-1}$, $[\nabla f(j)]_{j=k}=-1$, $f ( a )>0$ and $f(k)=0$. Letting $z(k):=w(k)x(k)$, $g(j) :=  [ {\phi (j) - \phi (k)}  ]^{\rm T}z(j)$, one has $g(j)\ge0$, $j\in\mathbb{N}_{a}^{k}$ and $g(k)=0$. Then the proof can be proceed as follows
\begin{equation}\label{Eq3.157}
\begin{array}{l}
w(k)\{{}_a^{\rm C}\nabla _k^{\alpha,w(k)} [\phi^{\rm T} (k)x(k)] - \phi^{\rm T} (k){}_a^{\rm C}\nabla _k^{\alpha,w(k)} x(k)\}\\
={}_a^{\rm C}\nabla _k^{\alpha} [\phi^{\rm T} (k)z(k)] - \phi^{\rm T} (k){}_a^{\rm C}\nabla _k^{\alpha}z(k)\\
 = \sum\nolimits_{j = a + 1}^k {\frac{{{{ ( {k - j + 1}  )}^{\overline { - \alpha } }}}}{{\Gamma  ( { - \alpha  + 1}  )}}\nabla [\phi^{\rm T}(j)z(j)]}\\
   \hspace{12pt}- \phi^{\rm T} (k)\sum\nolimits_{j = a + 1}^k {\frac{{{{ ( {k - j + 1}  )}^{\overline { - \alpha } }}}}{{\Gamma  ( { - \alpha  + 1}  )}}\nabla z(j)}\\
 = \sum\nolimits_{j = a + 1}^k {f(j - 1)\nabla g(j)} \\
 = f(j)g(j)|_{j=a}^{j=k} - \sum\nolimits_{j = a + 1}^k {\nabla f(j)g(j)} \\
 = -f ( a  )g ( a  )- \sum\nolimits_{j = a + 1}^{k-1} {\nabla f(j)g(j)}\\
 \le0,
\end{array}
\end{equation}
which implies the correctness of (\ref{Eq3.154}). Along this way, the remainder proof can be completed smoothly.
\end{proof}

Theorem \ref{Theorem3.26} is inspired by \cite[Section 3]{Lenka:2019IJAM}.  
In the proof of Theorem \ref{Theorem3.26}, $g(j)\ge0$, $j\in\mathbb{N}_{a}^{k}$ and $g(k)=0$ hold since $\phi$ is monotonically decreasing. If $\phi$ is monotonically increasing, one has $g(j)\le0$, $j\in\mathbb{N}_{a}^{k}$ and $g(k)=0$. On this basis, the sign of inequality should reverse.

Before introducing the third case, a key definition is provided here.
\begin{definition}\label{Definition3.4}
Let $u,v:\mathbb{N}_{a}\to\mathbb{R}^n$, $a\in\mathbb{R}$, $n\in\mathbb{Z}_+$. $u$ and $v$ are said synchronous if $ [ {u(j) - u(k)}  ]^{\rm T} [ {v(j) - v(k)}  ] \ge 0$ holds for any $j,k\in\mathbb{N}_{a}$. $u$ and $v$ are said asynchronous, if $ [ {u(j) - u(k)}  ]^{\rm T} [ {v(j) - v(k)}  ]$ $ \le 0$ holds for any $j,k\in\mathbb{N}_{a}$.
\end{definition}

\begin{theorem}\label{Theorem3.27}
If $u,v:\mathbb{N}_{a}\to\mathbb{R}^n$ are synchronous, $a\in\mathbb{R}$, $n\in\mathbb{Z}_+$, then for any $\alpha\in(0,1)$, $w:\mathbb{N}_a\to\mathbb{R}\backslash\{0\}$ with $\frac{w(a)}{w(k)}>0$, $k \in {\mathbb N_{a + 1}}$,one has
\begin{equation}\label{Eq3.158}
\begin{array}{l}
{}_a^{\rm C}\nabla _k^{\alpha,w(k)} [ {u^{\rm T}(k)v(k)} ]+[ {u^{\rm T}(k)v(k)} ]{}_a^{\rm C}\nabla _k^{\alpha,w(k)}1\\
\le u^{\rm T}(k){}_a^{\rm C}\nabla _k^{\alpha,w(k)}  {v(k)}+v^{\rm T}(k){}_a^{\rm C}\nabla _k^{\alpha,w(k)} {u(k)} ,
\end{array}
\end{equation}
\begin{equation}\label{Eq3.159}
\begin{array}{l}
{}_a^{\rm R}\nabla _k^{\alpha,w(k)} [ {u^{\rm T}(k)v(k)} ]+[ {u^{\rm T}(k)v(k)} ]{}_a^{\rm R}\nabla _k^{\alpha,w(k)}1\\
\le u^{\rm T}(k){}_a^{\rm R}\nabla _k^{\alpha,w(k)}  {v(k)}+v^{\rm T}(k){}_a^{\rm R}\nabla _k^{\alpha,w(k)} {u(k)} ,
\end{array}
\end{equation}
\begin{equation}\label{Eq3.160}
\begin{array}{l}
{}_a^{\rm G}\nabla _k^{\alpha,w(k)} [ {u^{\rm T}(k)v(k)} ]+[ {u^{\rm T}(k)v(k)} ]{}_a^{\rm G}\nabla _k^{\alpha,w(k)}1\\
\le u^{\rm T}(k){}_a^{\rm G}\nabla _k^{\alpha,w(k)}  {v(k)}+v^{\rm T}(k){}_a^{\rm G}\nabla _k^{\alpha,w(k)} {u(k)} .
\end{array}
\end{equation}
\end{theorem}
\begin{proof}
Define $f(j) := \frac{{ ( {k - j}  )\overline {^{ - \alpha }} }}{{\Gamma  ( {1 - \alpha }  )}}$, $g(j) =  w(j) [ {u(j) - u(k)}  ]^{\rm T} [ {v(j) - v(k)}  ]$. For the case of $w(k)>0$, one has $g(j)\ge0$, $j\in\mathbb{N}_{a}^{k}$ and $g(k)=0$. On this basis, one obtains
\begin{equation}\label{Eq3.161}
\begin{array}{l}
w(k)\{{}_a^{\rm C}\nabla _k^{\alpha,w(k)} [ {u(k)v(k)} ]+[ {u(k)v(k)} ]{}_a^{\rm C}\nabla _k^{\alpha,w(k)}1\}\\
- w(k)[u(k){}_a^{\rm C}\nabla _k^{\alpha,w(k)}  {v(k)}+v(k){}_a^{\rm C}\nabla _k^{\alpha,w(k)} {u(k)}] \\
 = \sum\nolimits_{j = a + 1}^k {\frac{{ ( {k - j + 1}  )\overline {^{ - \alpha }} }}{{\Gamma  ( {1 - \alpha }  )}}\nabla  {[ w(j){u(j)v(j)} ]}} \\
 \hspace{12pt}+[{u(k)v(k)}]\sum\nolimits_{j = a + 1}^k {\frac{{ ( {k - j + 1}  )\overline {^{ - \alpha }} }}{{\Gamma  ( {1 - \alpha }  )}}\nabla  {w(j) }} \\
 \hspace{12pt}- u(k)\sum\nolimits_{j = a + 1}^k {\frac{{ ( {k - j + 1}  )\overline {^{ - \alpha }} }}{{\Gamma  ( {1 - \alpha }  )}}{ {\nabla [w(j)v(j)]}  }} \\
 \hspace{12pt}- v(k)\sum\nolimits_{j = a + 1}^k {\frac{{ ( {k - j + 1}  )\overline {^{ - \alpha }} }}{{\Gamma  ( {1 - \alpha }  )}}{ {\nabla [w(j)u(j)]} } } \\
 = \sum\nolimits_{j = a + 1}^k {\frac{{ ( {k - j + 1}  )\overline {^{ - \alpha }} }}{{\Gamma  ( {1 - \alpha }  )}}\nabla \{ {w(j) [ {u(j) - u(k)}  ] [ {v(j) - v(k)}  ]} \}}\\
= \sum\nolimits_{j = a + 1}^k {f(j - 1)\nabla g(j)} \\
 = f(j)g(j)|_{j=a}^{j=k} - \sum\nolimits_{j = a + 1}^k {\nabla f(j)g(j)} \\
 = -f ( a  )g ( a  )- \sum\nolimits_{j = a + 1}^{k-1} {\nabla f(j)g(j)}\\
 \le0,
\end{array}
\end{equation}
which leads to the desired result in (\ref{Eq3.158}). For the case of $w(k)<0$, the desired result in (\ref{Eq3.158}) can be derived by using Lemma \ref{Lemma2.2}. Similarly, the results in (\ref{Eq3.159}), (\ref{Eq3.160}) can be established successfully.
\end{proof}

From the synchronous concept, one has $u(k)v(k)\ge 0$. When $w(k)$ is constant or increasing, one has ${}_a^{\rm C}\nabla _k^{\alpha,w(k)}1\ge0$. By using Theorem \ref{Theorem3.18} further, one obtains ${}_a^{\rm R}\nabla _k^{\alpha,w(k)}1\ge0$ and ${}_a^{\rm{G}}\nabla _k^{\alpha,w(k)}1\ge0$. To make it more practical, the following corollary can be developed.
\begin{corollary}\label{Corollary3.4}
If $u,v:\mathbb{N}_{a}\to\mathbb{R}^n$ are synchronous, $a\in\mathbb{R}$, $n\in\mathbb{Z}_+$, then for any $\alpha\in(0,1)$, $w:\mathbb{N}_a\to\mathbb{R}\backslash\{0\}$ with $\frac{w(a)}{w(k)}>0$ and ${}_a^{\rm C}\nabla _k^{\alpha,w(k)}1\ge0$, $k \in {\mathbb N_{a + 1}}$, one has
\begin{equation}\label{Eq3.162}
{}_a^{\rm C}\nabla _k^{\alpha,w(k)} [ {u(k)v(k)} ]\le u(k){}_a^{\rm C}\nabla _k^{\alpha,w(k)}  {v(k)}+v(k){}_a^{\rm C}\nabla _k^{\alpha,w(k)} {u(k)},
\end{equation}
\begin{equation}\label{Eq3.163}
{}_a^{\rm R}\nabla _k^{\alpha,w(k)} [ {u(k)v(k)} ]\le u(k){}_a^{\rm R}\nabla _k^{\alpha,w(k)}  {v(k)}+v(k){}_a^{\rm R}\nabla _k^{\alpha,w(k)} {u(k)},
\end{equation}
\begin{equation}\label{Eq3.164}
{}_a^{\rm{G}}\nabla _k^{\alpha,w(k)} [ {u(k)v(k)} ]\le u(k){}_a^{\rm{G}}\nabla _k^{\alpha,w(k)}  {v(k)}+v(k){}_a^{\rm{G}}\nabla _k^{\alpha,w(k)} {u(k)}.
\end{equation}
\end{corollary}

Corollary \ref{Corollary3.4} is inspired by \cite[Lemma 1]{Alsaedi:2015QAM}, \cite[Theorem 3.10]{Alsaedi:2017FCAA} and \cite[Theorem 1]{Wei:2021NODY}. When $v(k)=u(k)$ or $v(k)={\mathop{\rm sgn}} ( u(k) )$, some special cases of Theorem \ref{Theorem3.27} can also be developed. Besides, when the synchronous condition is replaced by the asynchronous case, the sign of inequality should be flipped.

Besides the fractional difference inequalities, the following fractional sum inequalities can be derived.

\begin{theorem}\label{Theorem3.28}
If $u,v:\mathbb{N}_{a}\to\mathbb{R}^n$ are synchronous, $n\in\mathbb{Z}_+$, then for any $\alpha>0$, $w:\mathbb{N}_a\to\mathbb{R}\backslash\{0\}$ with $\frac{w(a)}{w(k)}>0$, $k\in\mathbb{N}_{a+1}$, $a\in\mathbb{R}$, one has
\begin{equation}\label{Eq3.165}
{\textstyle
{}_a^{\rm{G}}\nabla _k^{-\alpha,w(k)} [{u^{\rm T}(k)v(k)} ] {}_a^{\rm{G}}\nabla _k^{-\alpha,w(k)} 1\ge {}_a^{\rm{G}}\nabla _k^{-\alpha,w(k)} {u^{\rm T}(k)}{}_a^{\rm{G}}\nabla _k^{-\alpha,w(k)} {v(k)} .}
\end{equation}
If $u,v:\mathbb{N}_{a}\to\mathbb{R}^n$ are asynchronous, $n\in\mathbb{Z}_+$, then for any $\alpha>0$, $w:\mathbb{N}_a\to\mathbb{R}\backslash\{0\}$ with $\frac{w(a)}{w(k)}>0$, $k\in\mathbb{N}_{a+1}$, $a\in\mathbb{R}$, one has
\begin{equation}\label{Eq3.166}
{\textstyle
{}_a^{\rm{G}}\nabla _k^{-\alpha,w(k)} [{u^{\rm T}(k)v(k)} ] {}_a^{\rm{G}}\nabla _k^{-\alpha,w(k)} 1\le {}_a^{\rm{G}}\nabla _k^{-\alpha,w(k)} {u^{\rm T}(k)}{}_a^{\rm{G}}\nabla _k^{-\alpha,w(k)} {v(k)} .}
\end{equation}
\end{theorem}
\begin{proof}
With the help of Lemma \ref{Lemma2.2}, if Theorem \ref{Theorem3.28} holds for $w(k)>0$, then Theorem \ref{Theorem3.28} holds for $w(k)>0$. As a result, only the case of $w(k)>0$ is considered. From the definition of synchronous functions, one has $u^{\rm T}(j)v(j)+u^{\rm T} ( i  )v ( i  )$ $\ge u^{\rm T}(j)v ( i  )+u^{\rm T} ( i  )v(j)$ for any $i,j\in\mathbb{N}_{a+1}$. Multiplying this inequality by the positive factor $\frac{{ ( {k - j + 1}  )\overline {^{\alpha  - 1}} }}{{\Gamma  ( \alpha   )}}w(j)$ on both left and right hand sides, and then summing both sides with respect to $j$ over the interval $(a+1, k)$ yield
\begin{equation}\label{Eq3.167}
\begin{array}{l}
\sum\nolimits_{j = a + 1}^k {\frac{{ ( {k - j + 1}  )\overline {^{\alpha  - 1}} }}{{\Gamma  ( \alpha  )}}w(j) [ {u^{\rm T}(j)v(j) + u^{\rm T} ( i  )v ( i  )}  ]} \\
 = w(k){}_a^{\rm{G}}\nabla _k^{ - \alpha ,w(k) }[u^{\rm T}(k)v(k)] + w(k)u^{\rm T} ( i  )v ( i ){}_a^{\rm{G}}\nabla _k^{ - \alpha ,w(k) }1\\
 \ge \sum\nolimits_{j = a + 1}^k {\frac{{ ( {k - j + 1}  )\overline {^{\alpha  - 1}} }}{{\Gamma  ( \alpha   )}}w(j) [ {u^{\rm T}(j)v ( i  ) + u^{\rm T} ( i  )v(j)}  ]} \\
 = w(k)v^{\rm T} ( i  ){}_a^{\rm{G}}\nabla _k^{ - \alpha ,w(k) }u(k) + w(k)u^{\rm T} ( i ){}_a^{\rm{G}}\nabla _k^{ - \alpha ,w(k) }v(k).
\end{array}
\end{equation}

Since $w(k)>0$, one has
\begin{equation}\label{Eq3.168}
\begin{array}{l}
 {}_a^{\rm{G}}\nabla _k^{ - \alpha ,w(k) }[u^{\rm T}(k)v(k)] + u^{\rm T} ( i  )v ( i ){}_a^{\rm{G}}\nabla _k^{ - \alpha ,w(k) }1\\
 \ge v^{\rm T} ( i  ){}_a^{\rm{G}}\nabla _k^{ - \alpha ,w(k) }u(k) + u^{\rm T} ( i ){}_a^{\rm{G}}\nabla _k^{ - \alpha ,w(k) }v(k).
\end{array}
\end{equation}

Multiplying the inequality in (\ref{Eq3.168}) by the positive factor $\frac{{ ( {k - i + 1}  )\overline {^{\alpha  - 1}} }}{{\Gamma  ( \alpha   )}}w(i)$ on both left and right hand sides, and then summing both sides with respect to $i$ over the interval $(a+1, k)$, one has
\begin{equation}\label{Eq3.169}
\begin{array}{l}
\sum\nolimits_{i = a + 1}^k {\frac{{ ( {k - i + 1}  )\overline {^{\alpha  - 1}} }}{{\Gamma  ( \alpha  )}}w(i){}_a^{\rm{G}}\nabla _k^{ - \alpha ,w(k) }[u^{\rm T}(k)v(k)]} \\
 + \sum\nolimits_{i = a + 1}^k {\frac{{ ( {k - i + 1}  )\overline {^{\alpha  - 1}} }}{{\Gamma  ( \alpha  )}}w(i)u^{\rm T} ( i  )v ( i  ){}_a^{\rm{G}}\nabla _k^{ - \alpha ,w(k) }1} \\
 = 2w(k){}_a^{\rm{G}}\nabla _k^{ - \alpha ,w(k) }[u^{\rm T}(k)v(k)]{}_a^{\rm{G}}\nabla _k^{ - \alpha ,w(k) }1\\
 \ge \sum\nolimits_{i = a + 1}^k {\frac{{ ( {k - i + 1}  )\overline {^{\alpha  - 1}} }}{{\Gamma  ( \alpha   )}}w(i)v^{\rm T} ( i  ){}_a^{\rm{G}}\nabla _k^{ - \alpha ,w(k) }u(k)} \\
   \hspace{12pt}   + \sum\nolimits_{i = a + 1}^k {\frac{{ ( {k - i + 1}  )\overline {^{\alpha  - 1}} }}{{\Gamma  ( \alpha   )}}w(i)u^{\rm T} ( i  ){}_a^{\rm{G}}\nabla _k^{ - \alpha ,w(k) }v(k)} \\
 = 2w(k){}_a^{\rm{G}}\nabla _k^{ - \alpha ,w(k) }u^{\rm T}(k){}_a^{\rm{G}}\nabla _k^{ - \alpha ,w(k) }v(k),
\end{array}
\end{equation}
which implies (\ref{Eq3.165}). In a similar way, if $u(k)$ and $v(k)$ are asynchronous, $u^{\rm T}(j)v(j)+u^{\rm T} ( i  )v ( i  )\le u^{\rm T}(j)v ( i  )+u^{\rm T} ( i  )v(j)$ and then (\ref{Eq3.166}) can be derived.
\end{proof}

\textsf{The inequalities in Theorem \ref{Theorem3.28} can be further generalized from the viewpoint of order number.}

\begin{theorem}\label{Theorem3.29}
If $u,v:\mathbb{N}_{a}\to\mathbb{R}^n$ are synchronous, $n\in\mathbb{Z}_+$, then for any $\alpha,\beta>0$, $w:\mathbb{N}_a\to\mathbb{R}\backslash\{0\}$ with $\frac{w(a)}{w(k)}>0$, $k\in\mathbb{N}_{a+1}$, $a\in\mathbb{R}$, one has
\begin{equation}\label{Eq3.170}
\begin{array}{l}
{}_a^{\rm{G}}\nabla _k^{-\alpha,w(k)}[ u^{\rm T}(k)v(k) ] {}_a^{\rm{G}}\nabla _k^{-\beta,w(k)}1+{}_a^{\rm{G}}\nabla _k^{-\beta,w(k)} [ u^{\rm T}(k)v(k) ]{}_a^{\rm{G}}\nabla _k^{-\alpha,w(k)}1\\
\ge {}_a^{\rm{G}}\nabla _k^{-\alpha,w(k)} {u^{\rm T}(k)}{}_a^{\rm{G}}\nabla _k^{-\beta,w(k)} {v(k)} +{}_a^{\rm{G}}\nabla _k^{-\beta,w(k)} {u^{\rm T}(k)}{}_a^{\rm{G}}\nabla _k^{-\alpha,w(k)} {v(k)} .
\end{array}
\end{equation}
If $u,v:\mathbb{N}_{a}\to\mathbb{R}^n$ are asynchronous, $n\in\mathbb{Z}_+$, then for any $\alpha,\beta>0$, $w:\mathbb{N}_a\to\mathbb{R}\backslash\{0\}$ with $\frac{w(a)}{w(k)}>0$, $k\in\mathbb{N}_{a+1}$, $a\in\mathbb{R}$, one has
\begin{equation}\label{Eq3.171}
\begin{array}{l}
{}_a^{\rm{G}}\nabla _k^{-\alpha,w(k)}[ u^{\rm T}(k)v(k) ] {}_a^{\rm{G}}\nabla _k^{-\beta,w(k)}1+{}_a^{\rm{G}}\nabla _k^{-\beta,w(k)} [ u^{\rm T}(k)v(k) ]{}_a^{\rm{G}}\nabla _k^{-\alpha,w(k)}1\\
\le {}_a^{\rm{G}}\nabla _k^{-\alpha,w(k)} {u^{\rm T}(k)}{}_a^{\rm{G}}\nabla _k^{-\beta,w(k)} {v(k)} +{}_a^{\rm{G}}\nabla _k^{-\beta,w(k)} {u^{\rm T}(k)}{}_a^{\rm{G}}\nabla _k^{-\alpha,w(k)} {v(k)} .
\end{array}
\end{equation}
\end{theorem}

\begin{proof}
Similarly, only the case of $w(k)>0$ is considered. For the synchronous case, recalling the proof of Theorem \ref{Theorem3.28} yields the inequality (\ref{Eq3.168}). Multiplying this inequality by the positive factor $\frac{{ ( {k - i + 1}  )\overline {^{\beta  - 1}} }}{{\Gamma  ( \beta   )}}w(i)$ on both left and right hand sides, and then summing both sides with respect to $i$ over the interval $(a+1, k)$, one has
\begin{equation}\label{Eq3.172}
\begin{array}{l}
\sum\nolimits_{i = a + 1}^k {\frac{{ ( {k - i + 1}  )\overline {^{\beta  - 1}} }}{{\Gamma  ( \beta  )}}w(i){}_a^{\rm{G}}\nabla _k^{ - \alpha ,w(k) }[u^{\rm T}(k)v(k)]} \\
 + \sum\nolimits_{i = a + 1}^k {\frac{{ ( {k - i + 1}  )\overline {^{\beta  - 1}} }}{{\Gamma  ( \beta  )}}w(i)u^{\rm T} ( i  )v ( i  ){}_a^{\rm{G}}\nabla _k^{ - \alpha ,w(k) }1} \\
 =w(k) {}_a^{\rm{G}}\nabla _k^{ - \alpha ,w(k) }[u^{\rm T}(k)v(k)]{}_a^{\rm{G}}\nabla _k^{ - \beta ,w(k) }1 \\
 \hspace{12pt} +w(k) {}_a^{\rm{G}}\nabla _k^{ - \beta ,w(k) }[u^{\rm T}(k)v(k)]{}_a^{\rm{G}}\nabla _k^{ - \alpha ,w(k) }1\\
 \ge \sum\nolimits_{i = a + 1}^k {\frac{{ ( {k - i + 1}  )\overline {^{\beta  - 1}} }}{{\Gamma  ( \beta  )}}w(i)v^{\rm T} ( i  ){}_a^{\rm{G}}\nabla _k^{ - \alpha ,w(k) }u(k)} \\
 \hspace{12pt}+ \sum\nolimits_{i = a + 1}^k {\frac{{ ( {k - i + 1}  )\overline {^{\beta  - 1}} }}{{\Gamma  ( \beta   )}}w(i)u^{\rm T} ( i  ){}_a^{\rm{G}}\nabla _k^{ - \alpha ,w(k) }v(k)} \\
 =w(k) {}_a^{\rm{G}}\nabla _k^{ - \alpha ,w(k) }u^{\rm T}(k){}_a^{\rm{G}}\nabla _k^{ - \beta ,w(k) }v(k)\\
 \hspace{12pt}  +w(k) {}_a^{\rm{G}}\nabla _k^{ - \beta ,w(k) }u^{\rm T}(k){}_a^{\rm{G}}\nabla _k^{ - \alpha ,w(k) }v(k).
\end{array}
\end{equation}
Considering $w(k)>0$, the inequality (\ref{Eq3.170}) can be derived. Likewise, if $u(k)$ and $v(k)$ are asynchronous, (\ref{Eq3.171}) can be obtained directly.
\end{proof}

\textsf{The inequalities in Theorem \ref{Theorem3.28} can be further generalized from the viewpoint of sequence number.}

\begin{theorem}\label{Theorem3.30}
If $u_i:\mathbb{N}_{a+1}\to\mathbb{R}$, $i=1,2,\cdots,n$, $n\in\mathbb{Z}_+$, are positive increasing functions, then for any $\alpha>0$, $w:\mathbb{N}_a\to\mathbb{R}\backslash\{0\}$ with $\frac{w(a)}{w(k)}>0$, $k\in\mathbb{N}_{a+1}$, $a\in\mathbb{R}$, one has
\begin{equation}\label{Eq3.173}
{\textstyle {}_a^{\rm{G}}\nabla _k^{-\alpha ,w(k) }\prod\nolimits_{i = 1}^n {{u_i}(k)} [{}_a^{\rm{G}}\nabla _k^{ - \alpha ,w(k) }1]^{n - 1} \ge \prod\nolimits_{i = 1}^n {{}_a^{\rm{G}}\nabla _k^{-\alpha ,w(k) }{u_i}(k)}.}
\end{equation}
\end{theorem}

\begin{proof}
The proof will be proven by induction.

$\blacktriangleright$ Part 1: when $n=1$, (\ref{Eq3.173}) reduces to
\begin{equation}\label{Eq3.174}
{\textstyle {}_a^{\rm{G}}\nabla _k^{-\alpha ,w(k) }{u_1}(k)  \ge {}_a^{\rm{G}}\nabla _k^{-\alpha ,w(k) }{u_1}(k),}
\end{equation}
which holds naturally.

$\blacktriangleright$ Part 2: assume that when $n=m\in\mathbb{Z}_+$, (\ref{Eq3.173}) holds, i.e.,
\begin{equation}\label{Eq3.175}
{\textstyle {}_a^{\rm{G}}\nabla _k^{-\alpha ,w(k) }\prod\nolimits_{i = 1}^m {{u_i}(k)} [{}_a^{\rm{G}}\nabla _k^{ - \alpha ,w(k) }1]^{m - 1} \ge \prod\nolimits_{i = 1}^m {{}_a^{\rm{G}}\nabla _k^{-\alpha ,w(k) }{u_i}(k)}.}
\end{equation}
Defining $v_m(k):=\prod\nolimits_{i = 1}^m {{u_i}(k)}$ and using Theorem \ref{Theorem3.28}, one has
\begin{equation}\label{Eq3.176}
\begin{array}{l}
{}_a^{\rm{G}}\nabla _k^{-\alpha ,w(k) }\prod\nolimits_{i = 1}^{m+1} {{u_i}(k)} [{}_a^{\rm{G}}\nabla _k^{ - \alpha ,w(k) }1]^{m}\\
={}_a^{\rm{G}}\nabla _k^{-\alpha ,w(k) }[{u_{m + 1}}(k){v_m}(k)]{}_a^{\rm{G}}\nabla _k^{ - \alpha ,\lambda }1[{}_a^{\rm{G}}\nabla _k^{ - \alpha ,w(k) }1]^{m-1}\\
\ge{}_a^{\rm{G}}\nabla _k^{-\alpha,w(k)} {u_{m + 1}(k)}{}_a^{\rm{G}}\nabla _k^{-\alpha,w(k)} {v_m(k)}[{}_a^{\rm{G}}\nabla _k^{ - \alpha ,w(k) }1]^{m-1}\\
 \ge {}_a^{\rm{G}}\nabla _k^{-\alpha,w(k)} {u_{m + 1}(k)}\prod\nolimits_{i = 1}^m {{}_a^{\rm{G}}\nabla _k^{-\alpha ,w(k) }{u_i}(k)}\\
 = \prod\nolimits_{i = 1}^{m+1} {{}_a^{\rm{G}}\nabla _k^{-\alpha ,w(k) }{u_i}(k)},
\end{array}
\end{equation}
where ${}_a^{\rm{G}}\nabla _k^{ - \alpha ,w(k) }1>0$ is adopted. (\ref{Eq3.176}) means that (\ref{Eq3.173}) holds for $n=m+1$. All of these end the proof.
\end{proof}

Theorem \ref{Theorem3.28} - Theorem \ref{Theorem3.30} are generalized from the Chebyshev inequalities from \cite{Belarbi:2009JIPAM,Alsaedi:2017FCAA,Fernandez:2020JCAM,Wei:2021NODY}.

\begin{theorem}\label{Theorem3.6}
If $x : \mathbb{N}_{a+1}^b \to [m,M]$, $m,M\in\mathbb{R}$, then for any $\alpha>0$, $k,b\in\mathbb{N}_{a+1}$, $a\in\mathbb{R}$, $w:\mathbb{N}_a\to\mathbb{R}\backslash\{0\}$ with $\frac{w(a)}{w(k)}>0$, one has
\begin{equation}\label{Eq3.32}
{\textstyle m {}_a^{\rm G}\nabla _k^{-\alpha,w(k)}1\le{}_a^{\rm G}\nabla _k^{-\alpha,w(k)}x(k) \le M {}_a^{\rm G}\nabla _k^{-\alpha,w(k)}1.}
\end{equation}
\end{theorem}
\begin{proof}
By using the definition and the condition, $w(k)>0$, $\frac{{(k - j + 1)\overline {^{\alpha  - 1}} }}{{\Gamma \left( \alpha  \right)}}>0$, $x(k)\in[m,M]$, one has
\begin{equation}\label{Eq3.33}
{\textstyle \begin{array}{rl}
{}_a^{\rm{G}}\nabla _k^{ - \alpha ,w(k)}{x}(k)=&\hspace{-6pt} {w^{ - 1}}(k)\sum\nolimits_{j = a + 1}^k {\frac{{(k - j + 1)\overline {^{\alpha  - 1}} }}{{\Gamma \left( \alpha  \right)}}w(j){x}(j) } \\
 \ge&\hspace{-6pt} m {w^{ - 1}}(k)\sum\nolimits_{j = a + 1}^k {\frac{{(k - j + 1)\overline {^{\alpha  - 1}} }}{{\Gamma \left( \alpha  \right)}}w(j) } \\
 =&\hspace{-6pt}m {}_a^{\rm G}\nabla _k^{-\alpha,w(k)}1,
\end{array}}
\end{equation}
\begin{equation}\label{Eq3.34}
{\textstyle \begin{array}{rl}
{}_a^{\rm{G}}\nabla _k^{ - \alpha ,w(k)}{x}(k)=&\hspace{-6pt} {w^{ - 1}}(k)\sum\nolimits_{j = a + 1}^k {\frac{{(k - j + 1)\overline {^{\alpha  - 1}} }}{{\Gamma \left( \alpha  \right)}}w(j){x}(j) } \\
 \le&\hspace{-6pt} M {w^{ - 1}}(k)\sum\nolimits_{j = a + 1}^k {\frac{{(k - j + 1)\overline {^{\alpha  - 1}} }}{{\Gamma \left( \alpha  \right)}}w(j) } \\
 =&\hspace{-6pt}M {}_a^{\rm G}\nabla _k^{-\alpha,w(k)}1.
\end{array}}
\end{equation}
The proof is thus completed.
\end{proof}

Theorem \ref{Theorem3.6} is the valuation theorem of nabla tempered fractional sum.

\begin{theorem}\label{Theorem3.31}
For any $\alpha>0$, $u,v:\mathbb{N}_{a+1}\to\mathbb{R}^n$, $n\in\mathbb{Z}_+$, $a\in\mathbb{R}$, $w:\mathbb{N}_a\to\mathbb{R}\backslash\{0\}$ with $\frac{w(a)}{w(k)}>0$, $k \in {\mathbb N_{a + 1}}$, one has
\begin{equation}\label{Eq3.177}
{\textstyle\{{}_a^{\rm{G}}\nabla _k^{ - \alpha,w(k) }[u^{\rm T}(k)v(k)]\}^2\le{}_a^{\rm{G}}\nabla _k^{ - \alpha,w(k) }[u^{\rm T}(k)u(k)]{}_a^{\rm{G}}\nabla _k^{ - \alpha,w(k) }[v^{\rm T}(k)v(k)].}
\end{equation}
\end{theorem}
\begin{proof}
Similarly, the $w(k)>0$ case is equivalent of the $w(k)<0$ case. Consequently, only the case of $w(k)>0$ is considered. Let $\hat u(j):=\sqrt{\frac{{ ( {k - j + 1}  )\overline {^{\alpha  - 1}} }}{{\Gamma  ( \alpha  )}}w(j)}u(j)$, $\hat v(j):=\sqrt{\frac{{ ( {k - j + 1}  )\overline {^{\alpha  - 1}} }}{{\Gamma ( \alpha   )}}w(j)}v(j)$. The following quadratic function is constructed
\begin{equation}\label{Eq3.178}
{\textstyle
\begin{array}{rl}
f(x)=&\hspace{-6pt}\big[\sum\nolimits_{j = a + 1}^k {\hat u^{\rm T}(j)\hat u(j)}\big]x^2+2\big[\sum\nolimits_{j = a + 1}^k {\hat u^{\rm T}(j)\hat v(j)}\big]x\\
&\hspace{-6pt}+\sum\nolimits_{j = a + 1}^k {\hat v^{\rm T}(j)\hat v(j)},
\end{array}}
\end{equation}
where $x\in\mathbb{R}$.

By using basic mathematical derivation, one has
\begin{equation}\label{Eq3.179}
{\textstyle
f(x)=\sum\nolimits_{j = a + 1}^k {[\hat u(j)x+\hat v(j)]^{\rm T}[\hat u(j)x+\hat v(j)]},}
\end{equation}
which means that its discriminant is nonnegative, i.e.,
\begin{equation}\label{Eq3.180}
{\textstyle
\begin{array}{rl}
\Delta=&\hspace{-6pt}4\big[\sum\nolimits_{j = a + 1}^k {\hat u^{\rm T}(j)\hat v(j)}\big]^2-4\big[\sum\nolimits_{j = a + 1}^k {\hat u^{\rm T}(j)\hat u(j)}\big]\big[\sum\nolimits_{j = a + 1}^k {\hat v^{\rm T}(j)\hat v(j)}\big]\\
\le&\hspace{-6pt}0.
\end{array}}
\end{equation}
Multiplying this inequality by the positive factor $\frac{1}{4w^{2}(k)}$ on both left and right hand sides and substituting $\hat u(j)$, $\hat v(j)$ into (\ref{Eq3.180}), one has
\begin{equation}\label{Eq3.181}
{\textstyle
\begin{array}{l}
\{{}_a^{\rm{G}}\nabla _k^{ - \alpha,w(k) }[u^{\rm T}(k)v(k)]\}^2-{}_a^{\rm{G}}\nabla _k^{ - \alpha,w(k) }[u^{\rm T}(k)u(k)]{}_a^{\rm{G}}\nabla _k^{ - \alpha,w(k) }[v^{\rm T}(k)v(k)]\\
=\frac{1}{4w^{2}(k)}\Delta\le0,
\end{array}}
\end{equation}
which completes the proof.
\end{proof}

\begin{theorem}\label{Theorem3.32}
If $f:\mathbb{D}\to\mathbb{R}$ is convex, then for any $\alpha>0$, $w:\mathbb{N}_a\to\mathbb{R}\backslash\{0\}$ with $\frac{w(a)}{w(k)}>0$, $k \in {\mathbb N_{a + 1}}$, $a\in\mathbb{R}$, one has
\begin{equation}\label{Eq3.182}
{\textstyle f\big([{}_a^{\rm{G}}\nabla _k^{ - \alpha,w(k) }1]^{-1}{}_a^{\rm{G}}\nabla _k^{ - \alpha,w(k) }x(k)\big)\le[{}_a^{\rm{G}}\nabla _k^{ - \alpha,w(k) }1]^{-1}{}_a^{\rm{G}}\nabla _k^{ - \alpha,w(k) }f(x(k)),}
\end{equation}
where $\mathbb{D}\subseteq\mathbb R^{n}$.
\end{theorem}
\begin{proof}
Letting $\lambda_j:=[{}_a^{\rm{G}}\nabla _k^{ - \alpha,w(k) }1]^{-1}\frac{{ ( {k - j + 1} )\overline {^{\alpha  - 1}} }}{{\Gamma  ( \alpha  )}}\frac{w(j)}{w(k)}$, it is not difficult to obtain $\sum\nolimits_{j = a + 1}^k{\lambda_j}=1$ and $\lambda_j>0$, $j\in\mathbb{N}_{a+1}^k$, $k\in\mathbb{N}_{a+1}$. By using (\ref{Eq2.9}) and the newly defined $\lambda_j$, one has
\begin{equation}\label{Eq3.183}
{\textstyle f\big([{}_a^{\rm{G}}\nabla _k^{ - \alpha,w(k) }1]^{-1}{}_a^{\rm{G}}\nabla _k^{ - \alpha,w(k) }x(k)\big)=f\big(\sum\nolimits_{j = a + 1}^k{\lambda_jx(j)}\big),}
\end{equation}
\begin{equation}\label{Eq3.184}
{\textstyle [{}_a^{\rm{G}}\nabla _k^{ - \alpha,w(k) }1]^{-1}{}_a^{\rm{G}}\nabla _k^{ - \alpha,w(k) }f(x(k))=\sum\nolimits_{j = a + 1}^k{\lambda_jf (x(j))}.}
\end{equation}
Now, the problem becomes
\begin{equation}\label{Eq3.185}
{\textstyle f\big(\sum\nolimits_{j = a + 1}^k{\lambda_jx(j)}\big)\le\sum\nolimits_{j = a + 1}^k{\lambda_jf (x(j))}.}
\end{equation}

The mathematical induction will be adopted once again.

$\blacktriangleright$ Part 1: when $k=a+1$, $\lambda_{a+1}=1$, (\ref{Eq3.185}) holds naturally.

$\blacktriangleright$ Part 2: when $k=a+2$, $\lambda_{a+1}=\frac{\alpha w(a+1)}{\alpha w(a+1)+w(a+2)}$, $\lambda_{a+2}=\frac{w(a+2)}{\alpha w(a+1)+w(a+2)}$, (\ref{Eq3.185}) becomes
\begin{equation}\label{Eq3.186}
{\textstyle
\begin{array}{l}
f(\lambda_{a+1}x(a+1)+\lambda_{a+2}x(a+2))\\
\le\lambda_{a+1}f(x(a+1))+\lambda_{a+2}f(x(a+2)).
\end{array}}
\end{equation}
By using the convexity of $f$, (\ref{Eq3.182}) can be obtained.

$\blacktriangleright$ Part 3: assume that when $k=m\in\mathbb{N}_{a+2}$, (\ref{Eq3.185}) holds, i.e.,
$f\big(\sum\nolimits_{j = a + 1}^m{\lambda_jx(j)}\big)$ $\le\sum\nolimits_{j = a + 1}^m{\lambda_jf (x(j))}$.
Since $\sum\nolimits_{j = a + 1}^{m+1}{\lambda_j}=1$, it follows
\begin{equation}\label{Eq3.187}
{\textstyle
\sum\nolimits_{j = a + 1}^{m}{\lambda_j}=1-\lambda_{m+1}.}
\end{equation}
Defining $\eta_j:=\frac{\lambda_j}{1-\lambda_{m+1}}$, $j\in\mathbb{N}_{a+1}^m$, then one has
\begin{equation}\label{Eq3.188}
{\textstyle
\sum\nolimits_{j = a + 1}^{m}{\eta_j}=\sum\nolimits_{j = a + 1}^{m}{\frac{\lambda_j}{1-\lambda_{m+1}}}=1,}
\end{equation}
and $\eta_j>0$, $j\in\mathbb{N}_{a+1}^m$.

On this basis, one has
\begin{equation}\label{Eq3.189}
{\textstyle
\begin{array}{l}
f\big(\sum\nolimits_{j = a + 1}^{m+1}{\lambda_jx(j)}\big)\\
=f\big(\lambda_{m+1}x(m+1)+\sum\nolimits_{j = a + 1}^{m}{\lambda_jx(j)}\big)\\
=f\big(\lambda_{m+1}x(m+1)+(1-\lambda_{m+1})\sum\nolimits_{j = a + 1}^{m}{\eta_jx(j)}\big)\\
\le \lambda_{m+1}f(x(m+1))+(1-\lambda_{m+1})f\big(\sum\nolimits_{j = a + 1}^{m}{\eta_jx(j)}\big)\\
\le \lambda_{m+1}f(x(m+1))+(1-\lambda_{m+1})\sum\nolimits_{j = a + 1}^{m}{\eta_jf(x(j))}\\
= \lambda_{m+1}f(x(m+1))+\sum\nolimits_{j = a + 1}^{m}{\lambda_jf(x(j))}\\
\le\sum\nolimits_{j = a + 1}^{m+1}{\lambda_jf (x(j))},
\end{array}}
\end{equation}
which means that (\ref{Eq3.185}) holds for $k=m+1$.

To sum up, the proof is done.
\end{proof}

\begin{theorem}\label{Theorem3.33}
For any $\alpha>0$, $u,v:\mathbb{N}_{a+1}\to\mathbb{R}_+$, $a\in\mathbb{R}$, $w:\mathbb{N}_a\to\mathbb{R}\backslash\{0\}$ with $\frac{w(a)}{w(k)}>0$, $k \in {\mathbb N_{a + 1}}$, $\frac{1}{p}+\frac{1}{q}=1$, if $p>1$, one has
\begin{equation}\label{Eq3.190}
{\textstyle{}_a^{\rm{G}}\nabla _k^{ - \alpha,w(k) }[u(k)v(k)]\le[{}_a^{\rm{G}}\nabla _k^{ - \alpha,w(k) }u^p(k)]^{\frac{1}{p}}[{}_a^{\rm{G}}\nabla _k^{ - \alpha,w(k) }v^q(k)]^{\frac{1}{q}}},
\end{equation}
and if $p\in(0,1)$, one has
\begin{equation}\label{Eq3.191}
{\textstyle{}_a^{\rm{G}}\nabla _k^{ - \alpha,w(k) }[u(k)v(k)]\ge[{}_a^{\rm{G}}\nabla _k^{ - \alpha,w(k) }u^p(k)]^{\frac{1}{p}}[{}_a^{\rm{G}}\nabla _k^{ - \alpha,w(k) }v^q(k)]^{\frac{1}{q}}}.
\end{equation}
\end{theorem}
\begin{proof}
Similarly, consider $w(k)>0$, $k \in {\mathbb N_{a + 1}}$. Let $a_j:=\sqrt[p]{\frac{{ ( {k - j + 1}  )\overline {^{\alpha  - 1}} }}{{\Gamma  ( \alpha  )}}w(j)}u(j)$, $b_j:=\sqrt[q]{\frac{{ ( {k - j + 1}  )\overline {^{\alpha  - 1}} }}{{\Gamma ( \alpha   )}}w(j)}v(j)$. It follows $a_j,b_j\ge0$, $\forall j\in\mathbb{N}_{a+1}^k$. Bearing this in mind, (\ref{Eq3.190}) can be expressed as
\begin{equation}\label{Eq3.192}
{\textstyle\sum\nolimits_{j = a + 1}^k{a_jb_j}\le[\sum\nolimits_{j = a + 1}^k{a_j^p}]^{\frac{1}{p}}[\sum\nolimits_{j = a + 1}^k{b_j^q}}]^{\frac{1}{q}},
\end{equation}
which is actually the discrete time H\"{o}lder inequality. (\ref{Eq3.191}) can be expressed as
\begin{equation}\label{Eq3.193}
{\textstyle\sum\nolimits_{j = a + 1}^k{a_jb_j}\ge[\sum\nolimits_{j = a + 1}^k{a_j^p}]^{\frac{1}{p}}[\sum\nolimits_{j = a + 1}^k{b_j^q}}]^{\frac{1}{q}}.
\end{equation}

When $p\in(0,1)$, one has $q=\frac{p}{p-1}$. Let $p':=\frac{1}{p}$, $q':=\frac{1}{1-p}$, it is not difficult to obtain $p',q'>1$ and $\frac{1}{p'}+\frac{1}{q'}=1$. By applying the discrete time H\"{o}lder inequality once again, it yields
\begin{equation}\label{Eq3.194}
{\textstyle
\begin{array}{rl}
\sum\nolimits_{j = a + 1}^k{a_j^p}=&\hspace{-6pt}\sum\nolimits_{j = a + 1}^k{({a_jb_j})^pb_j^{-p}}\\
\le&\hspace{-6pt}[\sum\nolimits_{j = a + 1}^k{({a_jb_j})^{pp'}}]^{\frac{1}{p'}}[\sum\nolimits_{j = a + 1}^k{b_j^{-pq'}}]^{\frac{1}{q'}}\\
\le&\hspace{-6pt}(\sum\nolimits_{j = a + 1}^k{{a_jb_j}})^{p}[\sum\nolimits_{j = a + 1}^k{b_j^{q}}]^{1-p}.
\end{array}}
\end{equation}

Calculating the $\frac{1}{p}$ power on both sides of (\ref{Eq3.194}) gives
\begin{equation}\label{Eq3.195}
{\textstyle
\begin{array}{rl}
[\sum\nolimits_{j = a + 1}^k{a_j^p}]^{\frac{1}{p}}\le&\hspace{-6pt}\sum\nolimits_{j = a + 1}^k{{a_jb_j}}[\sum\nolimits_{j = a + 1}^k{b_j^{q}}]^{\frac{1-p}{p}}\\
=&\hspace{-6pt}\sum\nolimits_{j = a + 1}^k{{a_jb_j}}[\sum\nolimits_{j = a + 1}^k{b_j^{q}}]^{-\frac{1}{q}},
\end{array}
}
\end{equation}
which implies (\ref{Eq3.193}).
\end{proof}

\begin{theorem}\label{Theorem3.34}
For any $\alpha>0$, $u,v:\mathbb{N}_{a+1}\to\mathbb{R}_+$, $a\in\mathbb{R}$, $w:\mathbb{N}_a\to\mathbb{R}\backslash\{0\}$ with $\frac{w(a)}{w(k)}>0$, $k \in {\mathbb N_{a + 1}}$, if $p\ge1$, one has
\begin{equation}\label{Eq3.196}
{\textstyle\{{}_a^{\rm G}\nabla _k^{ - \alpha,w(k) }[u(k)+v(k)]^p\}^{\frac{1}{p}}\le[{}_a^{\rm G}\nabla _k^{ - \alpha,w(k) }u^p(k)]^{\frac{1}{p}}+[{}_a^{\rm G}\nabla _k^{ - \alpha }v^p(k)]^{\frac{1}{p}}},
\end{equation}
and if $p\in(0,1)$, one has
\begin{equation}\label{Eq3.197}
{\textstyle\{{}_a^{\rm G}\nabla _k^{ - \alpha,w(k) }[u(k)+v(k)]^p\}^{\frac{1}{p}}\ge[{}_a^{\rm G}\nabla _k^{ - \alpha,w(k) }u^p(k)]^{\frac{1}{p}}+[{}_a^{\rm G}\nabla _k^{ - \alpha }v^p(k)]^{\frac{1}{p}}}.
\end{equation}
\end{theorem}
\begin{proof}
Similarly, consider $w(k)>0$, $k \in {\mathbb N_{a + 1}}$. Let $a_j:=\sqrt[p]{\frac{{ ( {k - j + 1}  )\overline {^{\alpha  - 1}} }}{{\Gamma  ( \alpha  )}}w(j)}u(j)$, $b_j:=\sqrt[q]{\frac{{ ( {k - j + 1}  )\overline {^{\alpha  - 1}} }}{{\Gamma ( \alpha   )}}w(j)}v(j)$. It follows $a_j,b_j>0$, $\forall j\in\mathbb{N}_{a+1}^k$. Along this way, (\ref{Eq3.196}) can be expressed as
\begin{equation}\label{Eq3.198}
{\textstyle[\sum\nolimits_{j = a + 1}^k{(a_j+b_j)^p}]^{\frac{1}{p}}\le(\sum\nolimits_{j = a + 1}^k{a_j^p})^{\frac{1}{p}}+(\sum\nolimits_{j = a + 1}^k{b_j^q}})^{\frac{1}{q}},
\end{equation}
which is just the Minkonski inequality in discrete time domain. Similarly, (\ref{Eq3.197}) can be expressed as
\begin{equation}\label{Eq3.199}
{\textstyle[\sum\nolimits_{j = a + 1}^k{(a_j+b_j)^p}]^{\frac{1}{p}}\ge(\sum\nolimits_{j = a + 1}^k{a_j^p})^{\frac{1}{p}}+(\sum\nolimits_{j = a + 1}^k{b_j^q}})^{\frac{1}{q}}.
\end{equation}
The proof completes immediately.
\end{proof}

In a similar way, the number of sequences could be extended to finite.
\begin{corollary}\label{Corollary3.3.1}
For any $\alpha>0$, $u_i:\mathbb{N}_{a+1}\to\mathbb{R}_+$, $i=1,2,\cdots,n$, $n\in\mathbb{Z}_+$, $a\in\mathbb{R}$, $w:\mathbb{N}_a\to\mathbb{R}\backslash\{0\}$ with $\frac{w(a)}{w(k)}>0$, $k \in {\mathbb N_{a + 1}}$, if $p\ge1$, one has
\begin{equation}\label{Eq3.200}
{\textstyle\{{}_a^{\rm G}\nabla _k^{ - \alpha,w(k) }[\sum\nolimits_{i = 1}^n u_i(k)]^p\}^{\frac{1}{p}}\le\sum\nolimits_{i = 1}^n[{}_a^{\rm G}\nabla _k^{ - \alpha,w(k) }u_i^p(k)]^{\frac{1}{p}} },
\end{equation}
and if $p\in(0,1)$, one has
\begin{equation}\label{Eq3.201}
{\textstyle\{{}_a^{\rm G}\nabla _k^{ - \alpha,w(k) }[\sum\nolimits_{i = 1}^n u_i(k)]^p\}^{\frac{1}{p}}\ge\sum\nolimits_{i = 1}^n[{}_a^{\rm G}\nabla _k^{ - \alpha,w(k) }u_i^p(k)]^{\frac{1}{p}} }.
\end{equation}
\end{corollary}

Theorem \ref{Theorem3.31} is the generalization of Cauchy inequality or Carlson inequality. Theorem \ref{Theorem3.32} is the generalization of Jensen inequality. Theorem \ref{Theorem3.33} is the generalization of H\"{o}lder inequality. Theorem \ref{Theorem3.34} is the generalization of Minkonski inequality. In a similar way, some integral like inequalities could be further developed for the nabla tempered fractional sum.


\section{Simulation study}\label{Section4}
In this section, six examples are provided to test the relation between theoretical and the simulated results.

\begin{example}\label{Example2}
To examine Theorem \ref{Theorem3.23} and Theorem \ref{Theorem3.24}, a new variable is introduced as $e(k)={ }_{a}^{\rm G} \nabla_{k}^{\alpha, w(k)} V(x(k))-\zeta^{\rm T}(x(k))_{a}^{\rm G} \nabla_{k}^{\alpha, w(k)} x(k)$. When $V(x(k))$ is differentiable, $\zeta(x(k))$ denotes the true gradient $\frac{{\rm d}V(x(k))}{{\rm d}x(k)}$. Setting $a=0$, $k\in\mathbb{N}_{a+1}^{100}$, $\alpha=0.01,0.02,\cdots,1$, the following four cases are considered.
\[\left\{ \begin{array}{ll}
{\rm{case}}\;1:\;V(x(k)) = {x^2}(k)\,,&\hspace{-6pt}x(k) = \sin (10k);\\
{\rm{case}}\;2:\;V(x(k)) = {x^2}(k)\,,&\hspace{-6pt}x(k) = \cos (10k);\\
{\rm{case}}\;3:\;V(x(k)) = |x(k)|\,,&\hspace{-6pt}x(k) = \sin (10k);\\
{\rm{case}}\;4:\;V(x(k)) = |x(k)|\,,&\hspace{-6pt}x(k) = \cos (10k).
\end{array} \right.\]

Since $[{ }_{a}^{\mathrm{G}} \nabla_{k}^{\alpha} x(k)]_{k=a+1}=x(a+1)$, $e(a+1)=-x^2(a+1)$ for case 1, case 2 and  $e(a+1)=0$ for case 3, case 4. Consequently, for each $\alpha$ the maximum of $e(k)$ is calculated with $k\in\mathbb{N}_{a+2}$ instead of $k\in\mathbb{N}_{a+1}$. Letting the tempered function be positive and linearly increasing as $w(k)=0.5(k-a)$, the simulated results are shown in Figure \ref{Fig4a}. Letting the tempered function be negative and linearly decreasing i.e., $w(k)=0.5(a-k)$, the corresponding simulated results are shown in Figure \ref{Fig4b}.  It is shown that $e(k)\le0$ holds for all the elaborated conditions, which confirms the developed inequalities. To demonstrate more details, the relationship between $W(z(k))$ and $z(k)$ is shown in Figure \ref{Fig5}. It can be found that when $w(k)>0$, $W(z(k))$ is the convex function of $z(k)$ on the whole for case 1, case 2 and $W(z(k))$ is the convex function of $z(k)$ for case 3, case 4. When $w(k)<0$, $W(z(k))$ is the convex function of $z(k)$ on the whole for case 1, case 2 and $W(z(k))$ is the convex function of $z(k)$ for case 3, case 4. All of these coincide with the theoretical analysis.

\begin{figure}[!htbp]
\centering
\setlength{\abovecaptionskip}{-2pt}
	\vspace{-10pt}
	\subfigtopskip=-2pt
	\subfigbottomskip=2pt
	\subfigcapskip=-10pt
\subfigure[$w(k)=0.5(k-a)$.]{
\begin{minipage}[t]{0.48\linewidth}
\includegraphics[width=1.0\hsize]{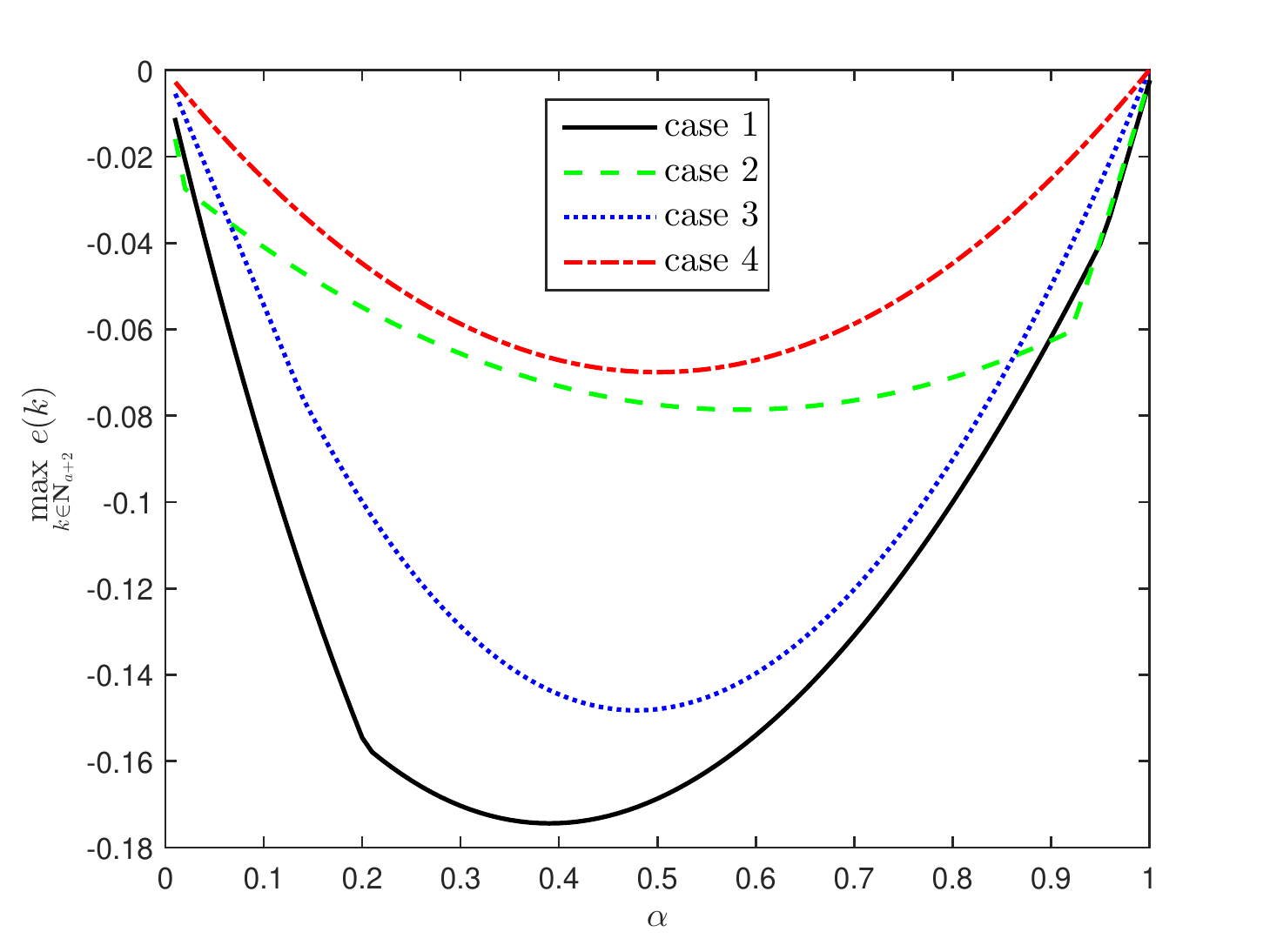}
\label{Fig4a}
\end{minipage}%
}
\subfigure[$w(k)=0.5(a-k)$.]{
\begin{minipage}[t]{0.48\linewidth}
\includegraphics[width=1.0\hsize]{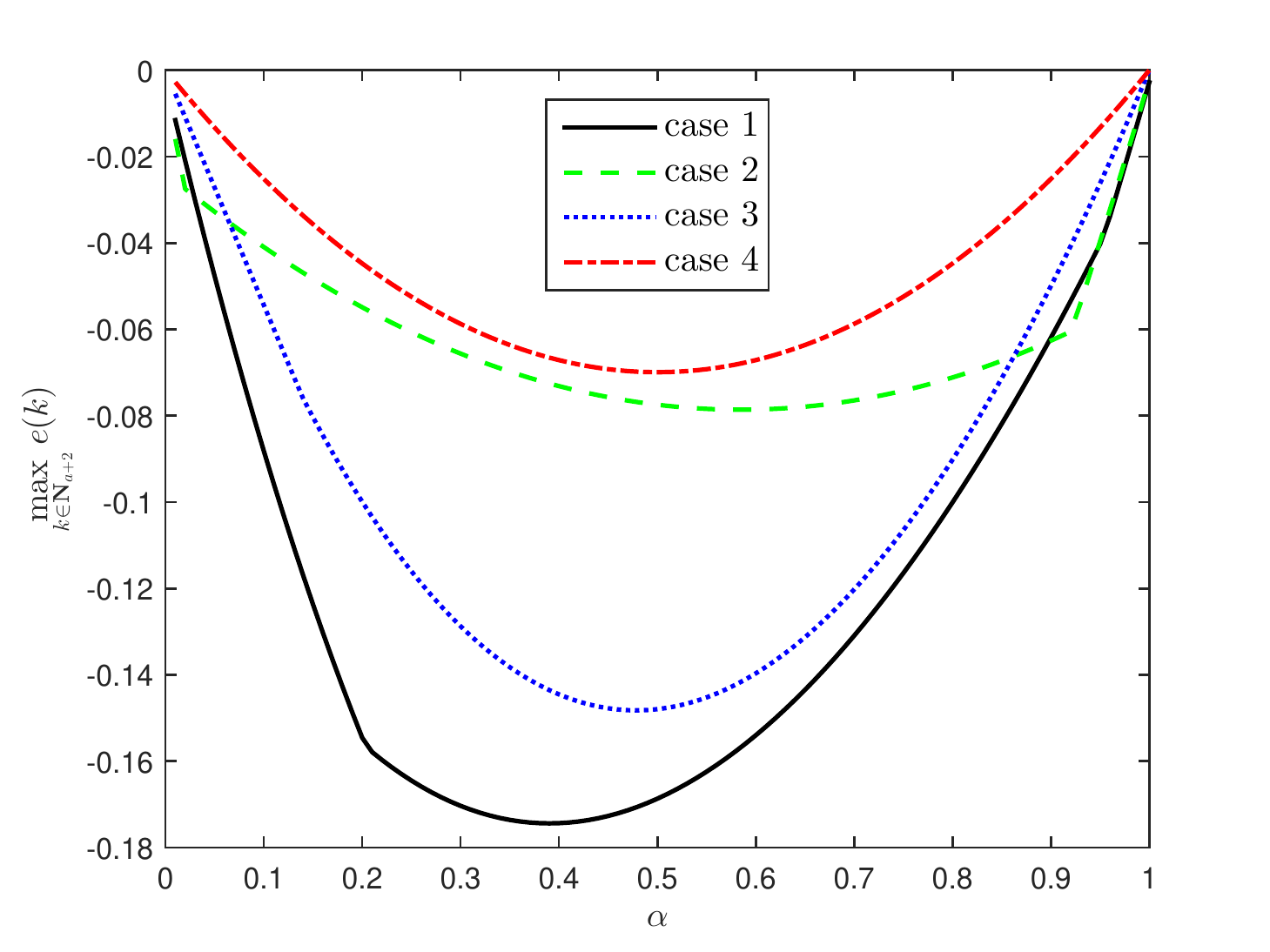}
\label{Fig4b}
\end{minipage}%
}
\centering
\caption{The evolution of the error with respect to the order.}
\label{Fig4}
\end{figure}
\begin{figure}[!htbp]
\centering
\setlength{\abovecaptionskip}{-2pt}
	\vspace{-10pt}
	\subfigtopskip=-2pt
	\subfigbottomskip=2pt
	\subfigcapskip=-10pt
\subfigure[$w(k)=0.5(k-a)$.]{
\begin{minipage}[t]{0.48\linewidth}
\includegraphics[width=1.0\hsize]{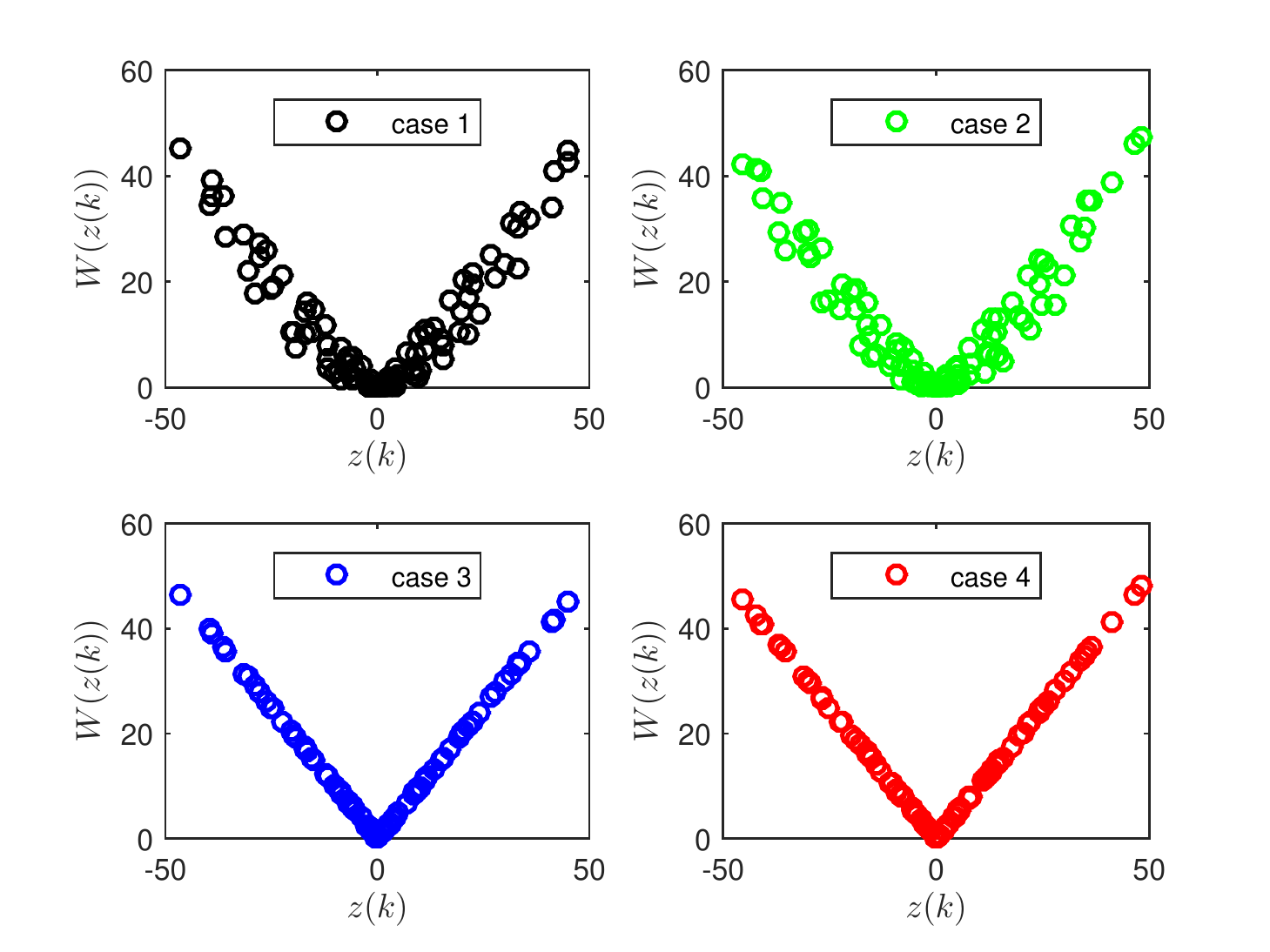}
\label{Fig5a}
\end{minipage}%
}
\subfigure[$w(k)=0.5(a-k)$.]{
\begin{minipage}[t]{0.48\linewidth}
\includegraphics[width=1.0\hsize]{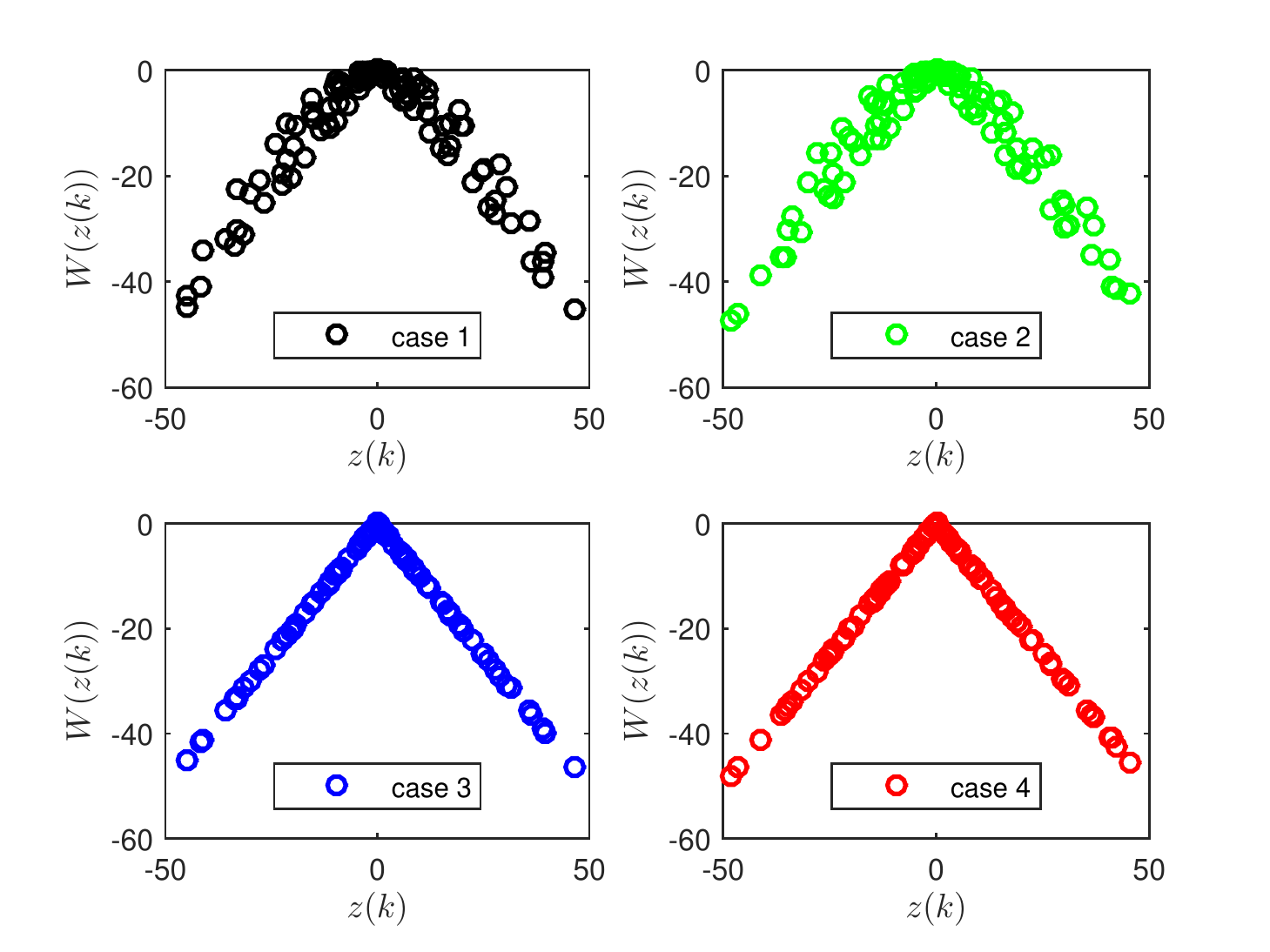}
\label{Fig5b}
\end{minipage}%
}
\centering
\caption{The evolution of $W(z(k))$ with respect to $z(k)$.}
\label{Fig5}
\end{figure}

Letting $w(k)=0.4^{k-a}+0.5$ and $w(k)=0.4^{a-k}+0.5$, the corresponding results are displayed in Figure \ref{Fig5} and Figure \ref{Fig6}, respectively.  Note that $w(k)=0.4^{k-a}+0.5$ is positive and exponentially decreasing. $w(k)=0.4^{a-k}+0.5$ is positive and exponentially increasing. Figure \ref{Fig6} shows that the maximum of $e(k)$ in the mentioned conditions are non-positive, which confirms the correctness of the developed fractional difference inequalities firmly. Figure \ref{Fig7a} shows that $W(z(k))$ is the convex function of $z(k)$ almost everywhere for case 1, case 2 and $W(z(k))$ is the convex function of $z(k)$ for case 3, case 4. Though the points are uneven dispersion in Figure \ref{Fig7b}, the convexity of $W(z(k))$ can be found from the trend closely.

\begin{figure}[!htbp]
\centering
\setlength{\abovecaptionskip}{-2pt}
	\vspace{-10pt}
	\subfigtopskip=-2pt
	\subfigbottomskip=2pt
	\subfigcapskip=-2pt
\subfigure[$w(k)=0.4^{k-a}+0.5$.]{
\begin{minipage}[t]{0.48\linewidth}
\includegraphics[width=1.0\hsize]{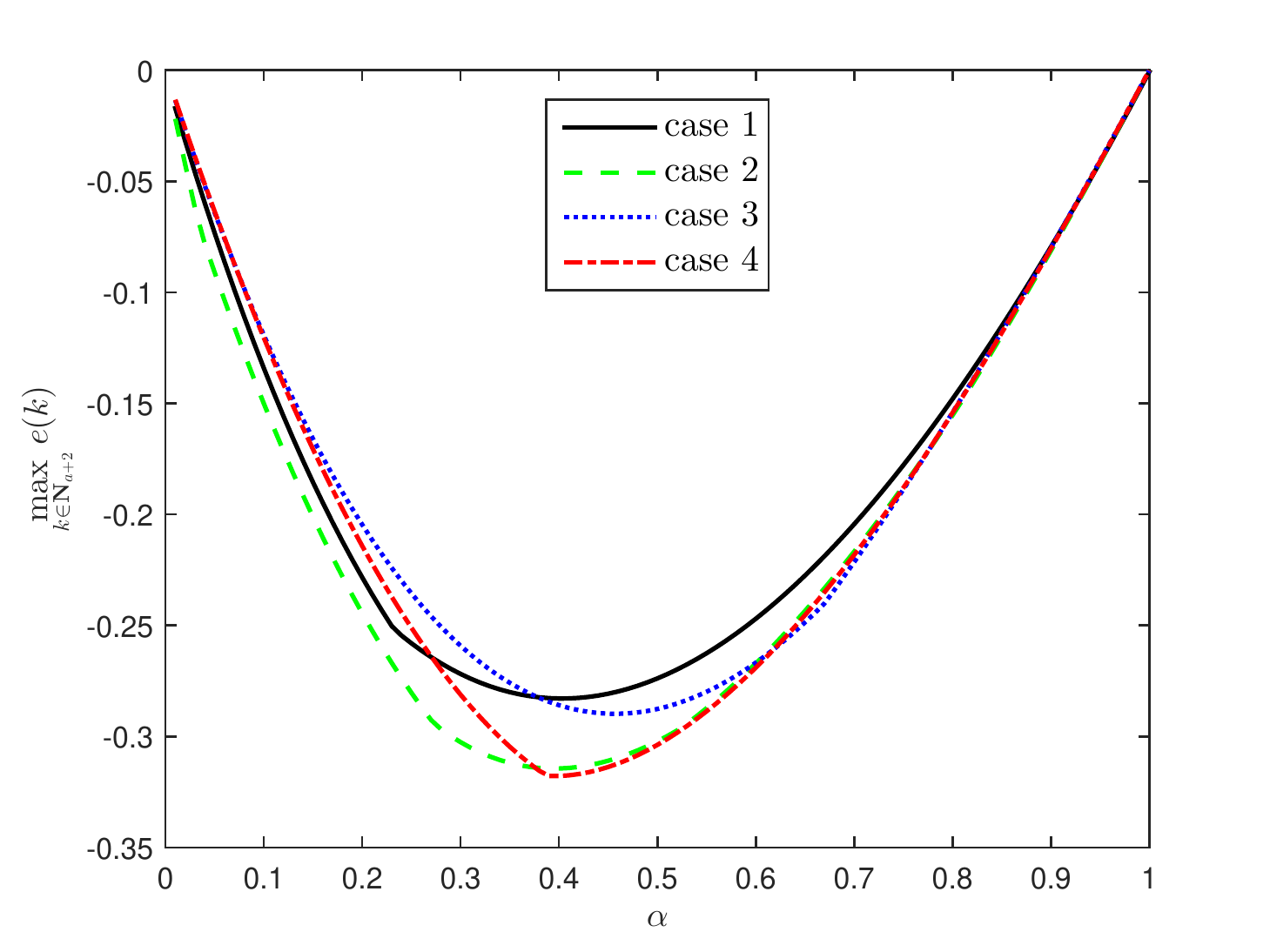}
\label{Fig6a}
\end{minipage}%
}
\subfigure[$w(k)=0.4^{a-k}+0.5$.]{
\begin{minipage}[t]{0.48\linewidth}
\includegraphics[width=1.0\hsize]{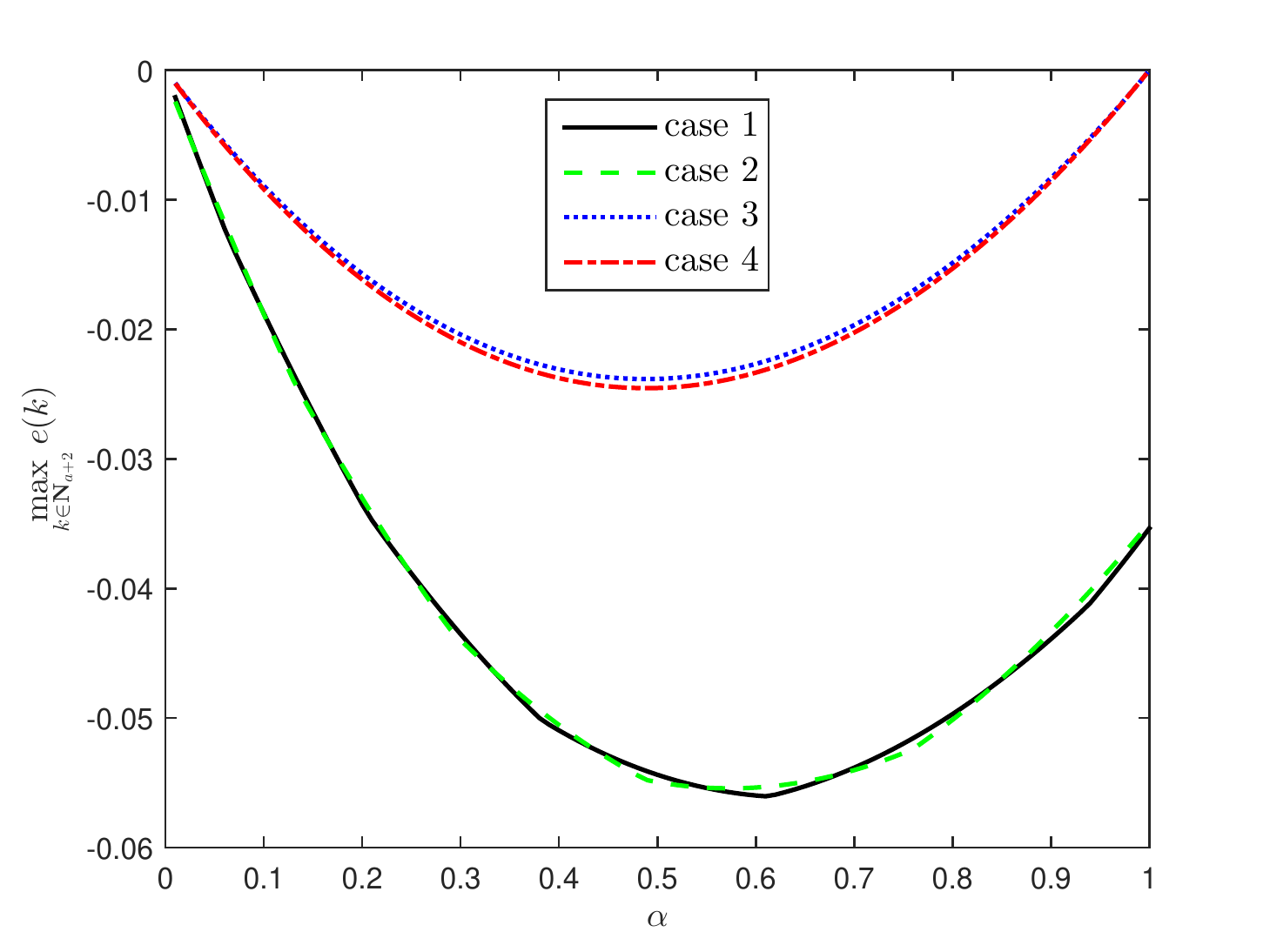}
\label{Fig6b}
\end{minipage}%
}
\centering
\caption{The evolution of the error with respect to the order.}
\label{Fig6}
\end{figure}
\begin{figure}[!htbp]
\centering
\setlength{\abovecaptionskip}{-2pt}
	\vspace{-10pt}
	\subfigtopskip=-2pt
	\subfigbottomskip=2pt
	\subfigcapskip=-2pt
\subfigure[$w(k)=0.4^{k-a}+0.5$.]{
\begin{minipage}[t]{0.48\linewidth}
\includegraphics[width=1.0\hsize]{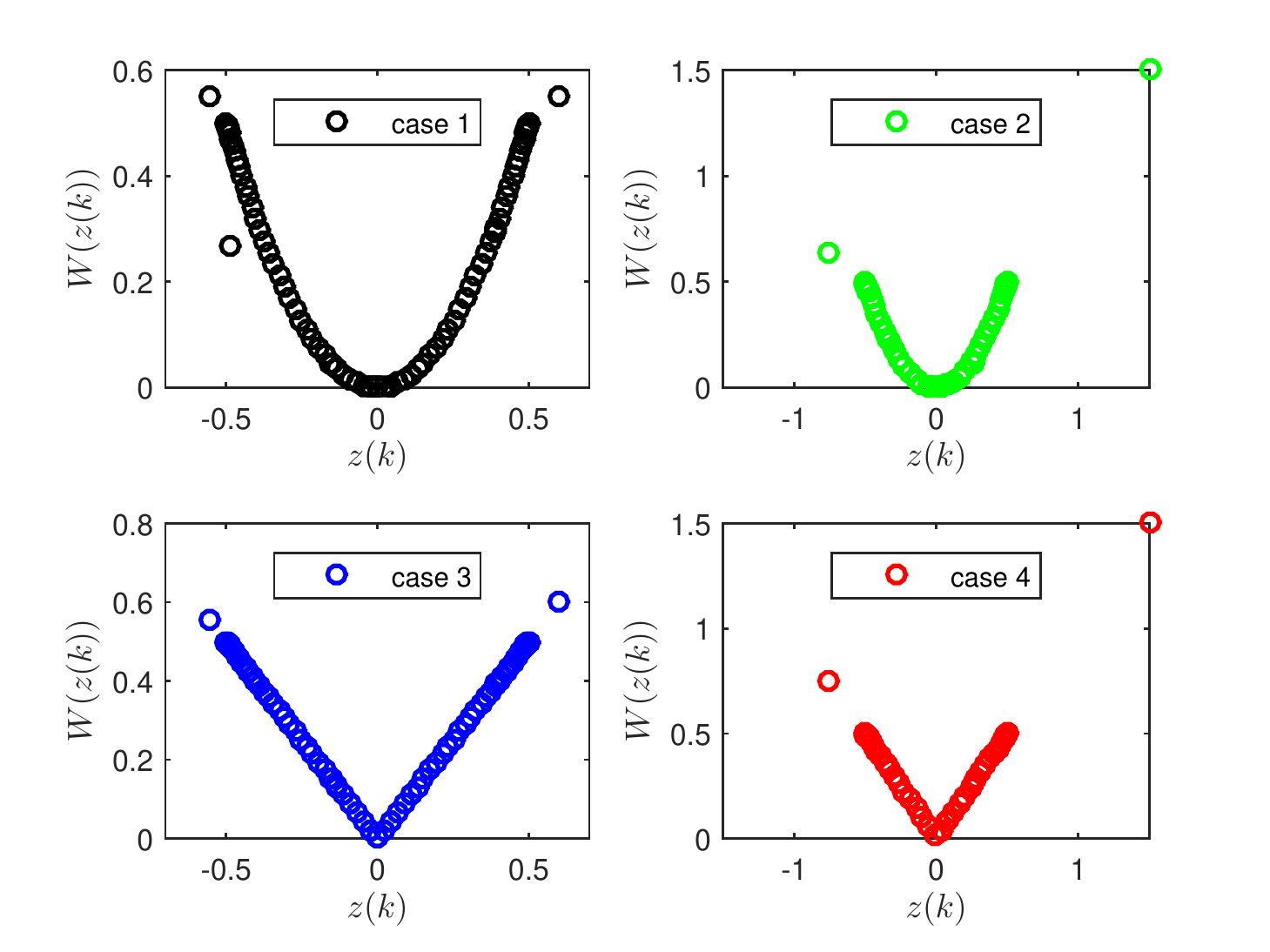}
\label{Fig7a}
\end{minipage}%
}
\subfigure[$w(k)=0.4^{a-k}+0.5$.]{
\begin{minipage}[t]{0.48\linewidth}
\includegraphics[width=1.0\hsize]{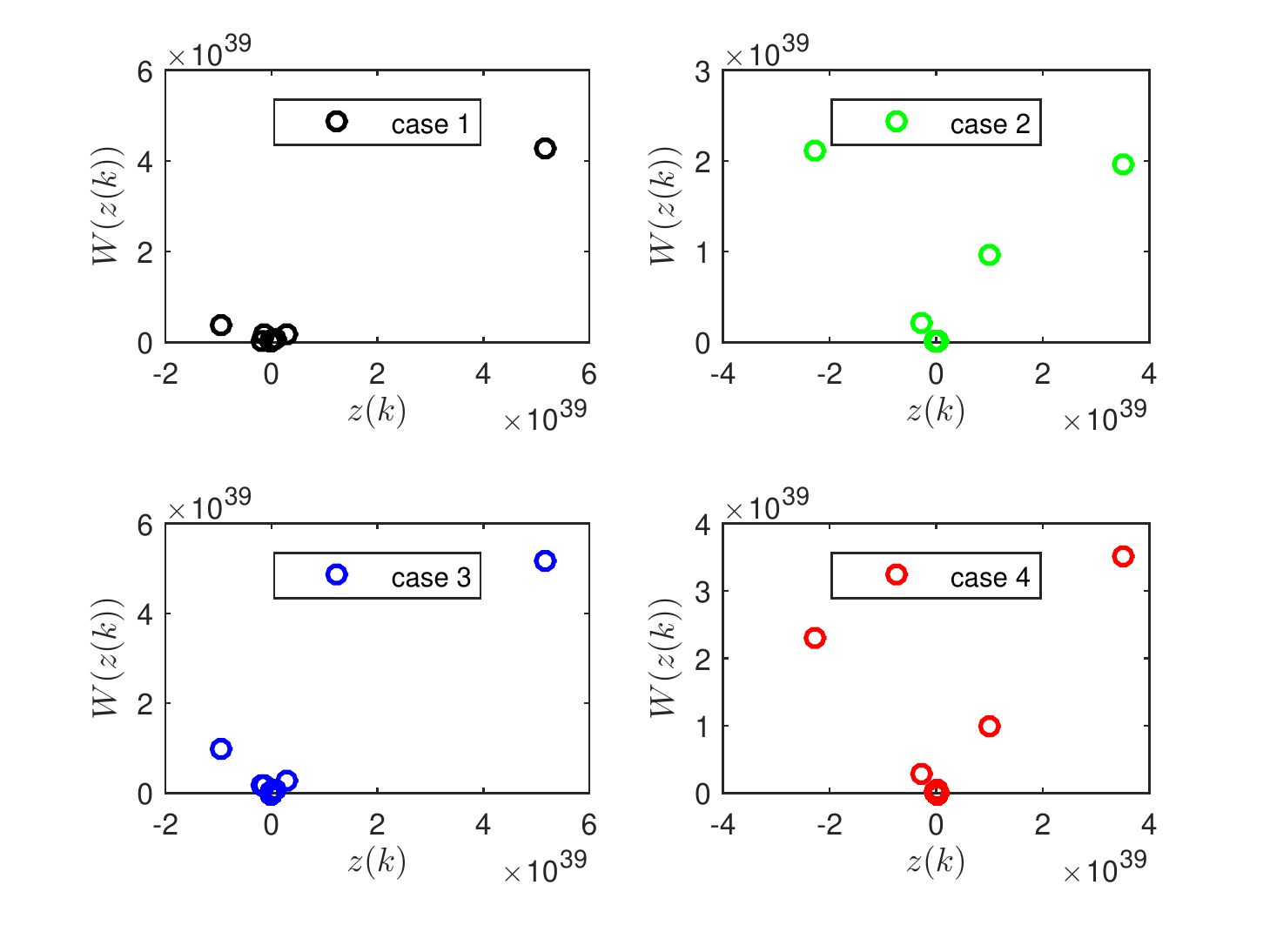}
\label{Fig7b}
\end{minipage}%
}
\centering
\caption{The evolution of $W(z(k))$ with respect to $z(k)$.}
\label{Fig7}
\end{figure}
\end{example}
\begin{example}\label{Example3}
To examine the Chebyshev fractional sum inequality, define $e(k):={}_a^{\rm G}\nabla _k^{-\alpha,w(k)} [ {u^{\rm T}(k)v(k)} ]{}_a^{\rm G}\nabla _k^{ - \alpha,w(k) }1 - {}_a^{\rm G}\nabla _k^{-\alpha,w(k)} {u^{\rm T}(k)}{}_a^{\rm G}\nabla _k^{-\alpha,w(k)} {v(k)}$. Setting $a=0$, $\alpha=0.01,0.02,\cdots,4$, considering different $w(k)$ and the following four cases
\[\left\{ \begin{array}{l}
{\rm{case}}\;1:\;u(k)=\sin (k-a), v(k)={\rm sgn}(\sin (k-a));\\
{\rm{case}}\;2:\;u(k)=\sin (k-a), v(k)=2\sin (k-a)+1;\\
{\rm{case}}\;3:\;u(k)=k-a, v(k)=(k-a)^2;\\
{\rm{case}}\;4:\;u(k)=\frac{1}{k-a}, v(k)=\frac{1}{(k-a)^2},
\end{array} \right.\]
the simulated results are shown in Figure \ref{Fig8} and Figure \ref{Fig9}.
\begin{figure}[!htbp]
  \centering
  \includegraphics[width=0.5\textwidth]{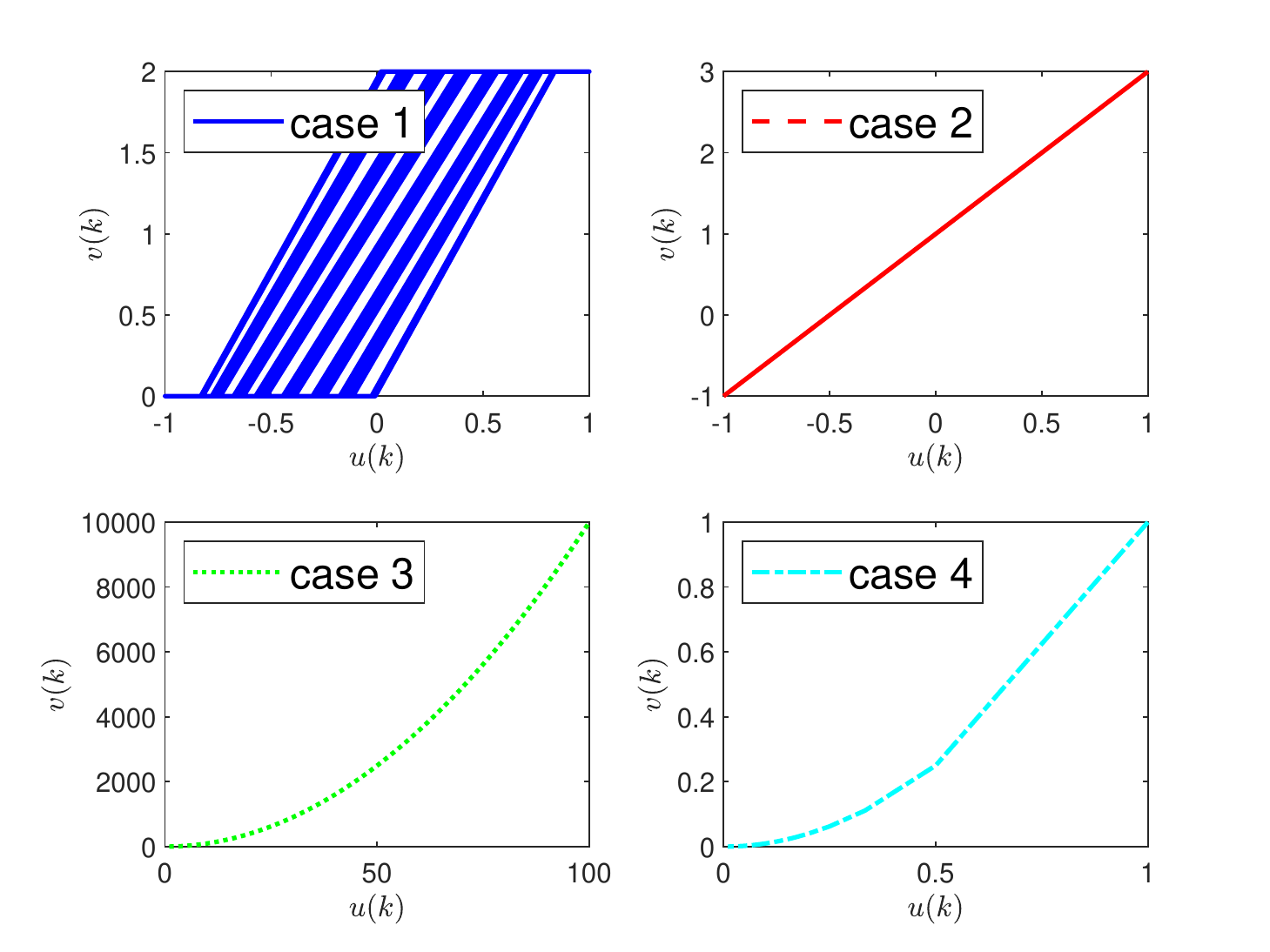}
  \vspace{-10pt}
  \caption{The diagram of $u(k)$ and $v(k)$.}\label{Fig8}
\end{figure}
\begin{figure}[!htbp]
	\centering
	\setlength{\abovecaptionskip}{-2pt}
	\vspace{-10pt}
	\subfigtopskip=-2pt
	\subfigbottomskip=2pt
	\subfigcapskip=-2pt
	\subfigure[$w(k) = {( - 1)^{k - a}} + 2$]{
	\begin{minipage}[t]{0.5\linewidth}
	\includegraphics[width=1\hsize]{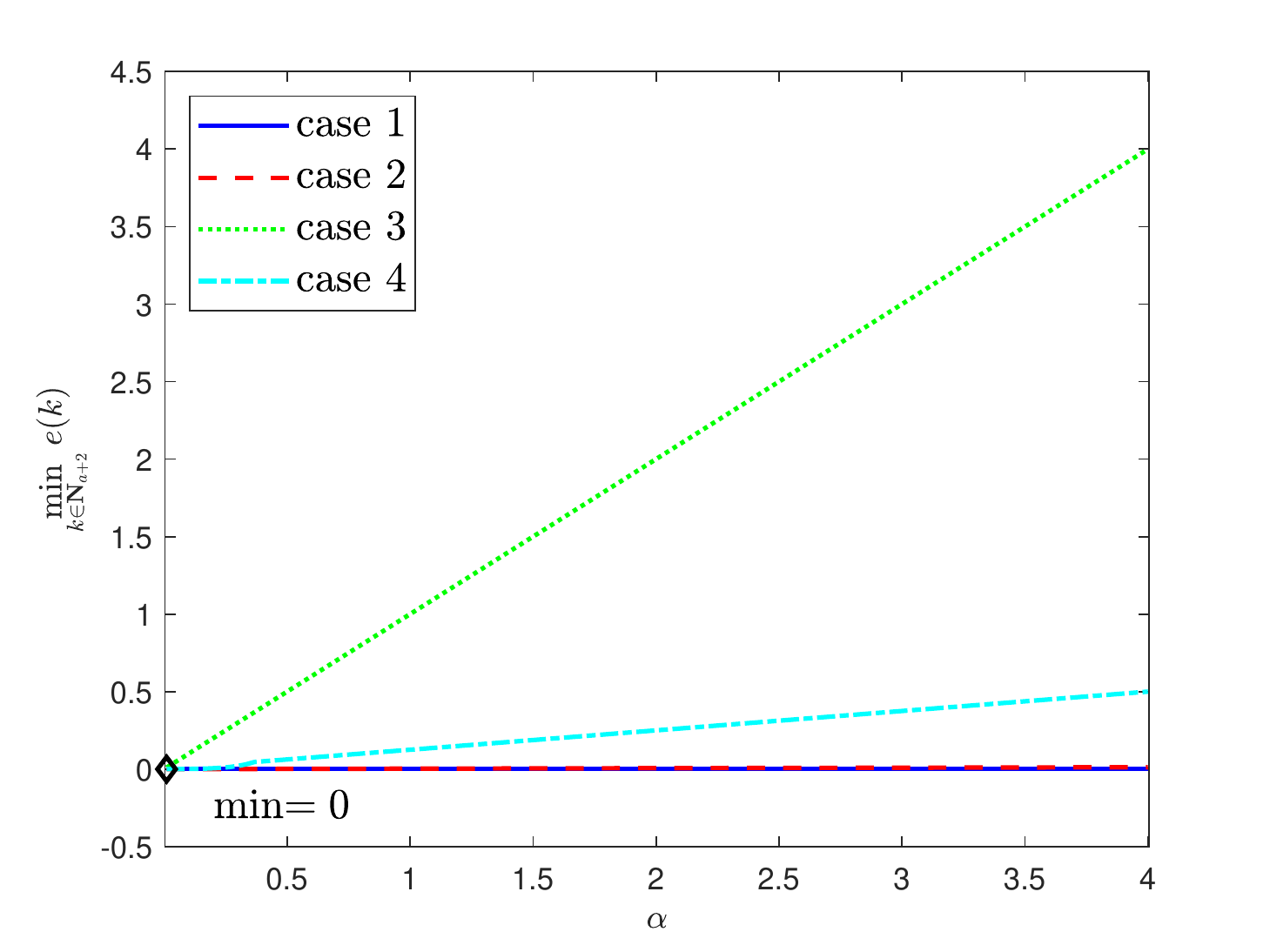}
	\label{Fig9a}
	\end{minipage}%
	}\hspace{-12pt}
	\subfigure[$w(k) = {( - 1)^{k - a}} - 2$]{
	\begin{minipage}[t]{0.5\linewidth}
	\includegraphics[width=1\hsize]{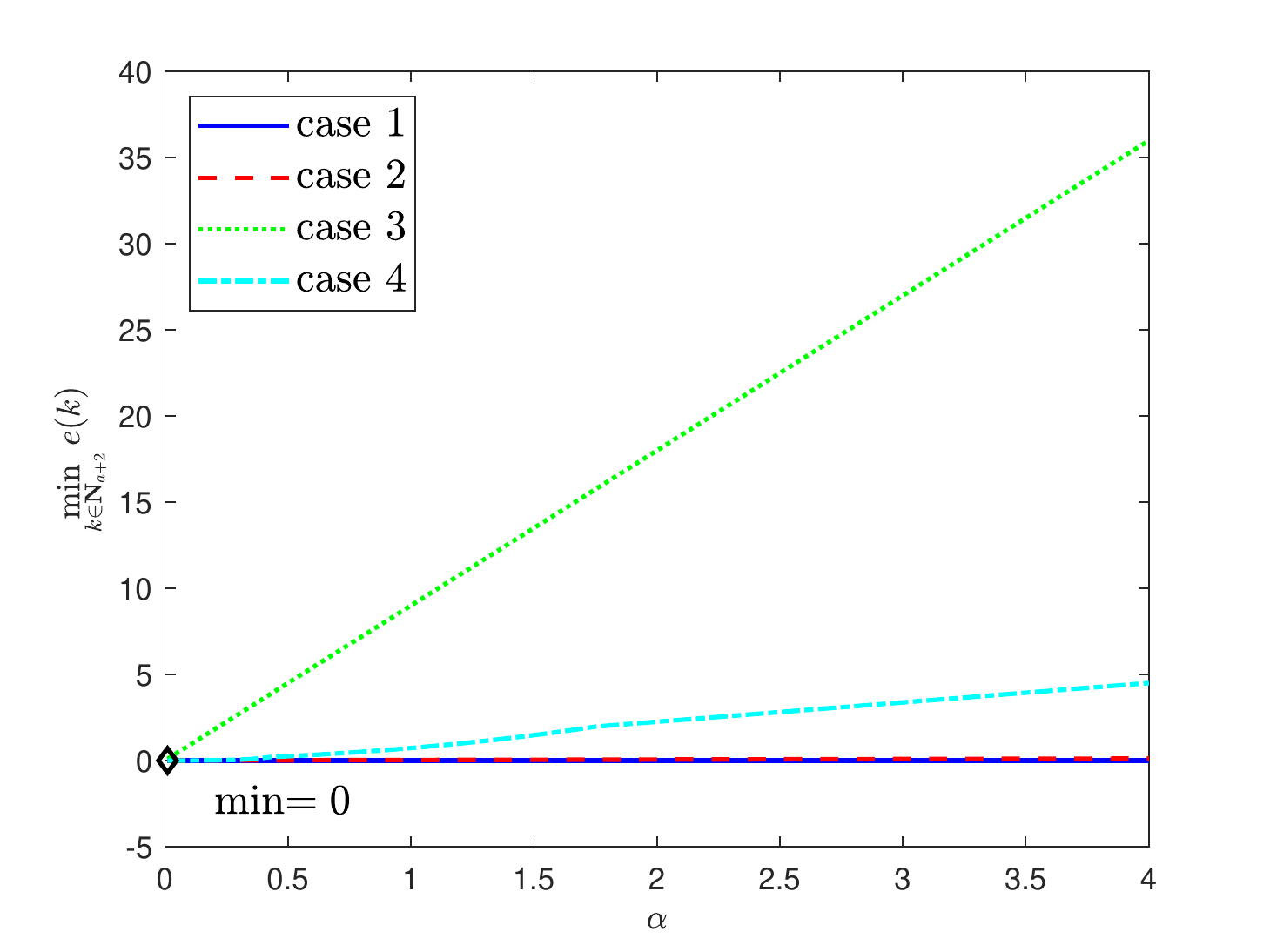}
	\label{Fig9b}
	\end{minipage}%
	}
    \subfigure[$w(k) = \sin (k - a) + 2$]{
	\begin{minipage}[t]{0.5\linewidth}
	\includegraphics[width=1\hsize]{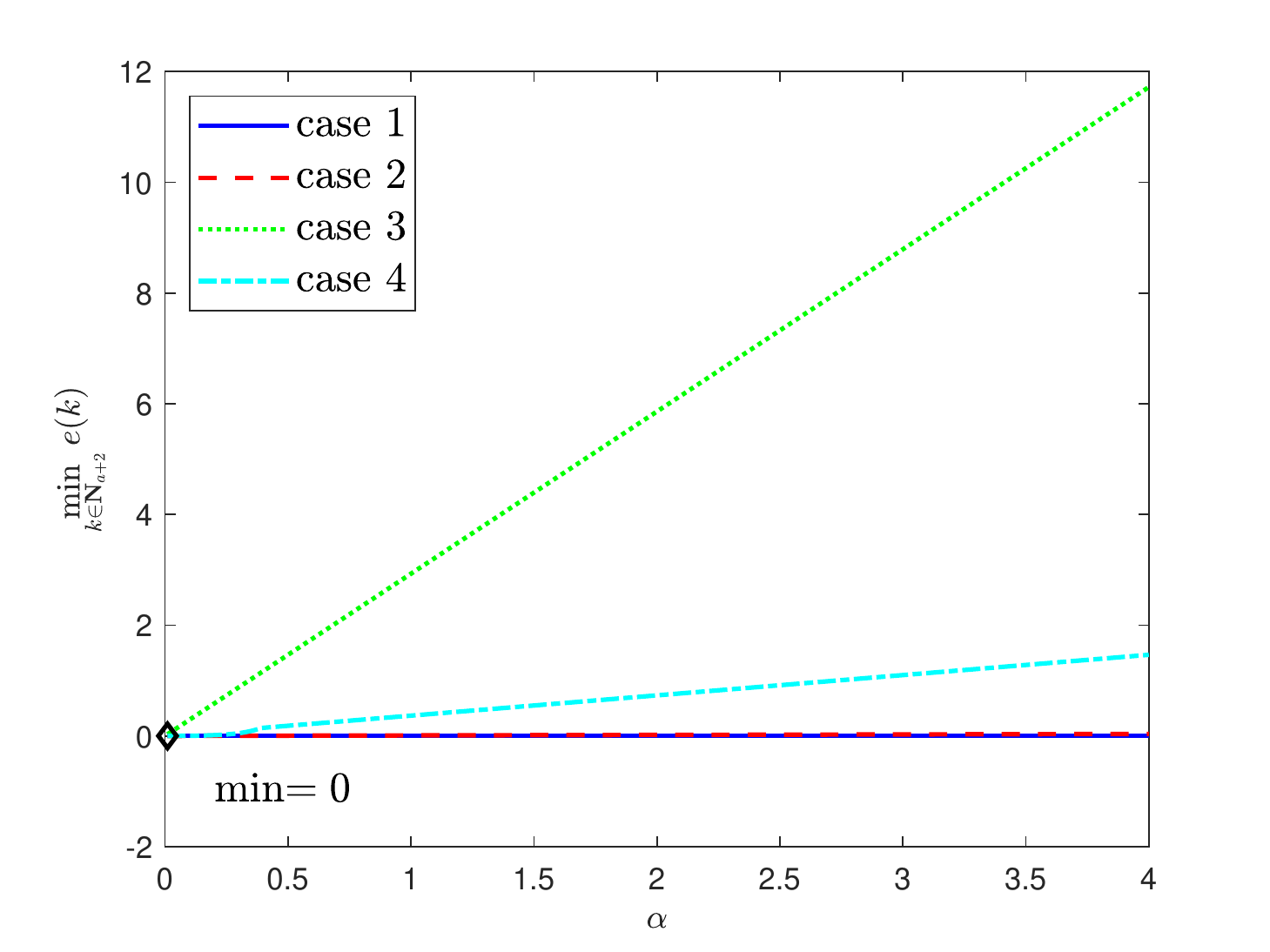}
	\label{Fig9c}
	\end{minipage}%
	}\hspace{-12pt}
	\subfigure[$w(k) = \sin (k - a) - 2$]{
	\begin{minipage}[t]{0.5\linewidth}
	\includegraphics[width=1\hsize]{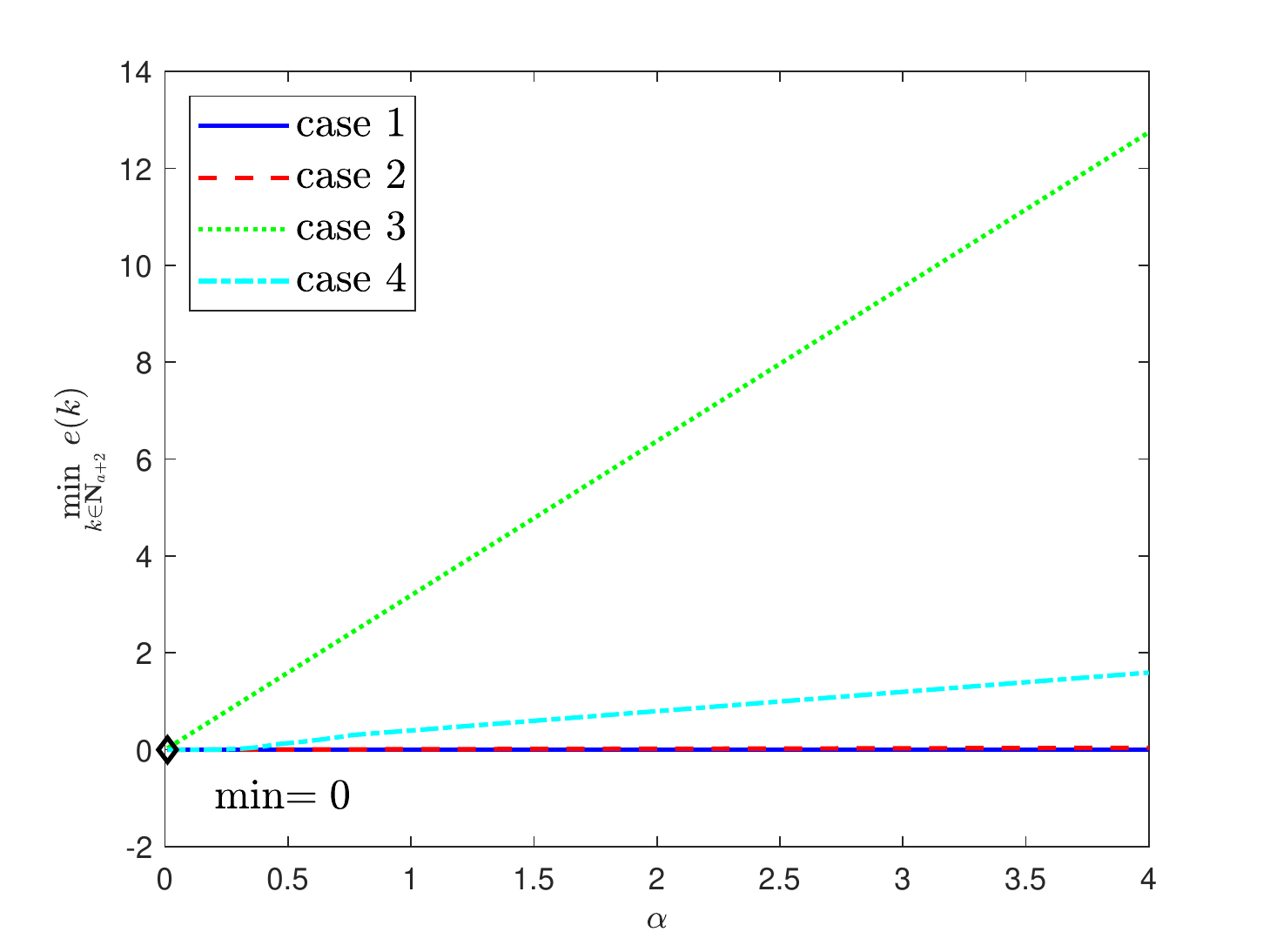}
	\label{Fig9d}
	\end{minipage}%
	}
	\centering
	\caption{The evolution of the error with respect to the order.}
	\label{Fig9}
\end{figure}

The synchronous property of $u(k)$ and $v(k)$ can be confirmed by Figure \ref{Fig8}. In particular, $u(k),v(k)$ are continuously oscillating in {\rm case 1} and {\rm case 2}. $u(k),v(k)$ are monotonically increasing in {\rm case 3}. $u(k),v(k)$ are monotonically decreasing in {\rm case 4}. In Figure \ref{Fig9a} and Figure \ref{Fig9c}, the tempered functions $w(k)$ are positive. In Figure \ref{Fig9b} and Figure \ref{Fig9d}, the tempered functions $w(k)$ are negative. From the quantitative analysis, it can be observed that the minimum of $e(k)$ in all four cases are $0$, which verifies the validity of Theorem \ref{Theorem3.28} firmly.
\end{example}

\begin{example}\label{Example4}
To examine the Cauchy fractional sum inequality, define $e(k):=\big\{{}_a^{\rm G} \nabla_{k}^{-\alpha,w(k)}[u(k) v(k)]\big\}^{2}-{ }_{a}^{\rm G} \nabla_{k}^{-\alpha,w(k)}[u^{2}(k)]{}_{a}^{\rm G} \nabla_{k}^{-\alpha,w(k)}[v^{2}(k)]$. Setting $a=0$, $\alpha=0.01,0.02,\cdots,4$ and considering the following four cases
\[\left\{ \begin{array}{l}
{\rm{case}}\;1:\;u(k)=\sin (10 k)+2, v(k)={\cos (10 k)}-2;\\
{\rm{case}}\;2:\;u(k)=\sin (10 k){(1-\lambda)^{a-k}}, v(k)={\cos (10 k)}{(1-\lambda)^{k-a}}, \lambda=2;\\
{\rm{case}}\;3:\;u(k)=\sin (10 k)+\nu_1, v(k)=\cos (10 k)+\nu_2, \nu_1,\nu_2=\texttt{randn}(\texttt{size}(k));\\
{\rm{case}}\;4:\;u(k)=\nu_1\sin (10 k), v(k)=\nu_2\cos (10 k), \nu_1,\nu_2=\texttt{randn}(\texttt{size}(k)),
\end{array} \right.\]
the simulated results are shown in Figure \ref{Fig10} and Figure \ref{Fig11}.
\begin{figure}[!htbp]
  \centering
  \includegraphics[width=0.5\textwidth]{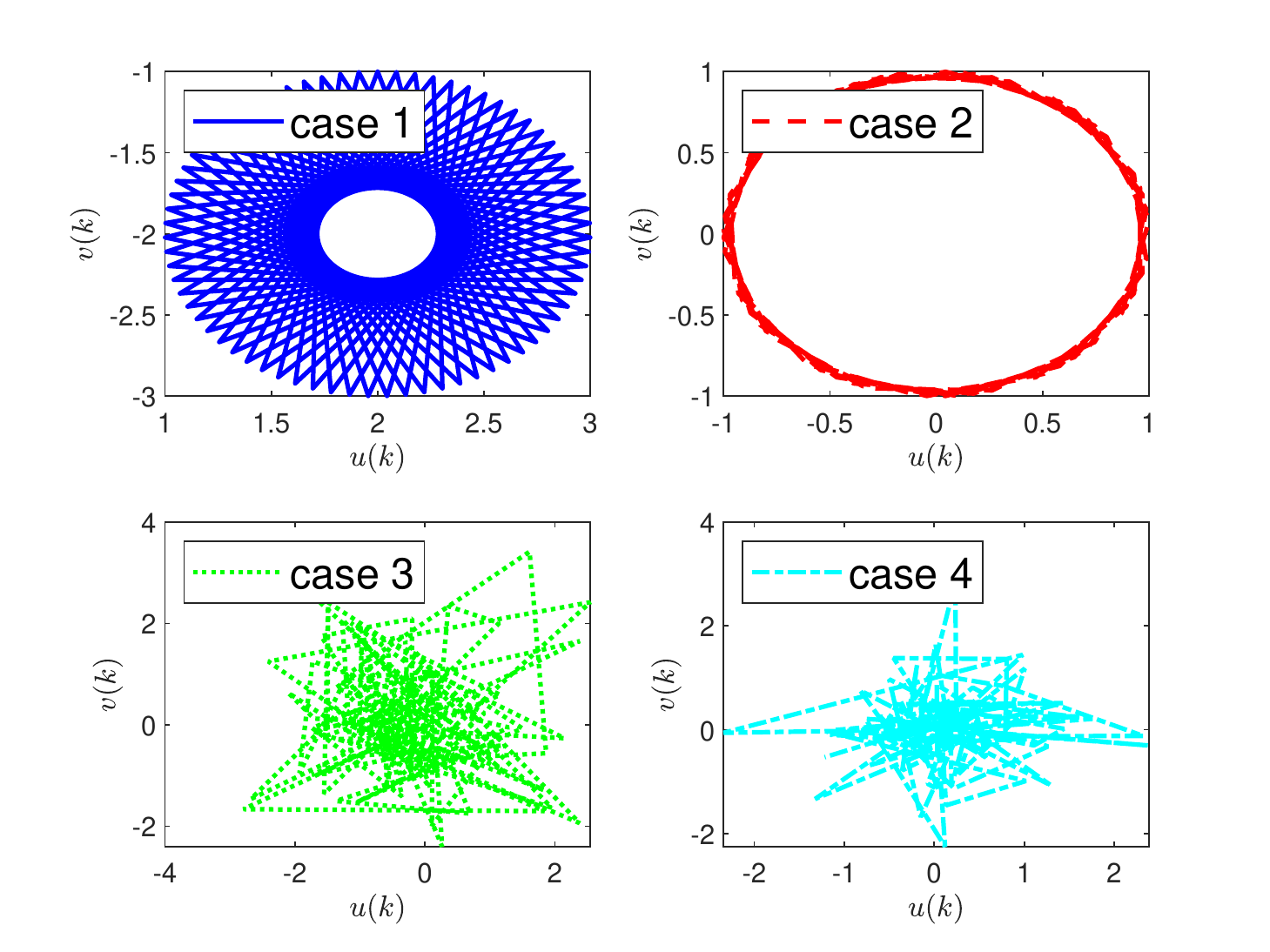}
  \vspace{-10pt}
  \caption{The diagram of $u(k)$ and $v(k)$.}\label{Fig10}
\end{figure}
\begin{figure}[!htbp]
	\centering
	\setlength{\abovecaptionskip}{-2pt}
	\vspace{-10pt}
	\subfigtopskip=-2pt
	\subfigbottomskip=2pt
	\subfigcapskip=-2pt
	\subfigure[$w(k) = {( - 1)^{k - a}} + 2$]{
	\begin{minipage}[t]{0.5\linewidth}
	\includegraphics[width=1\hsize]{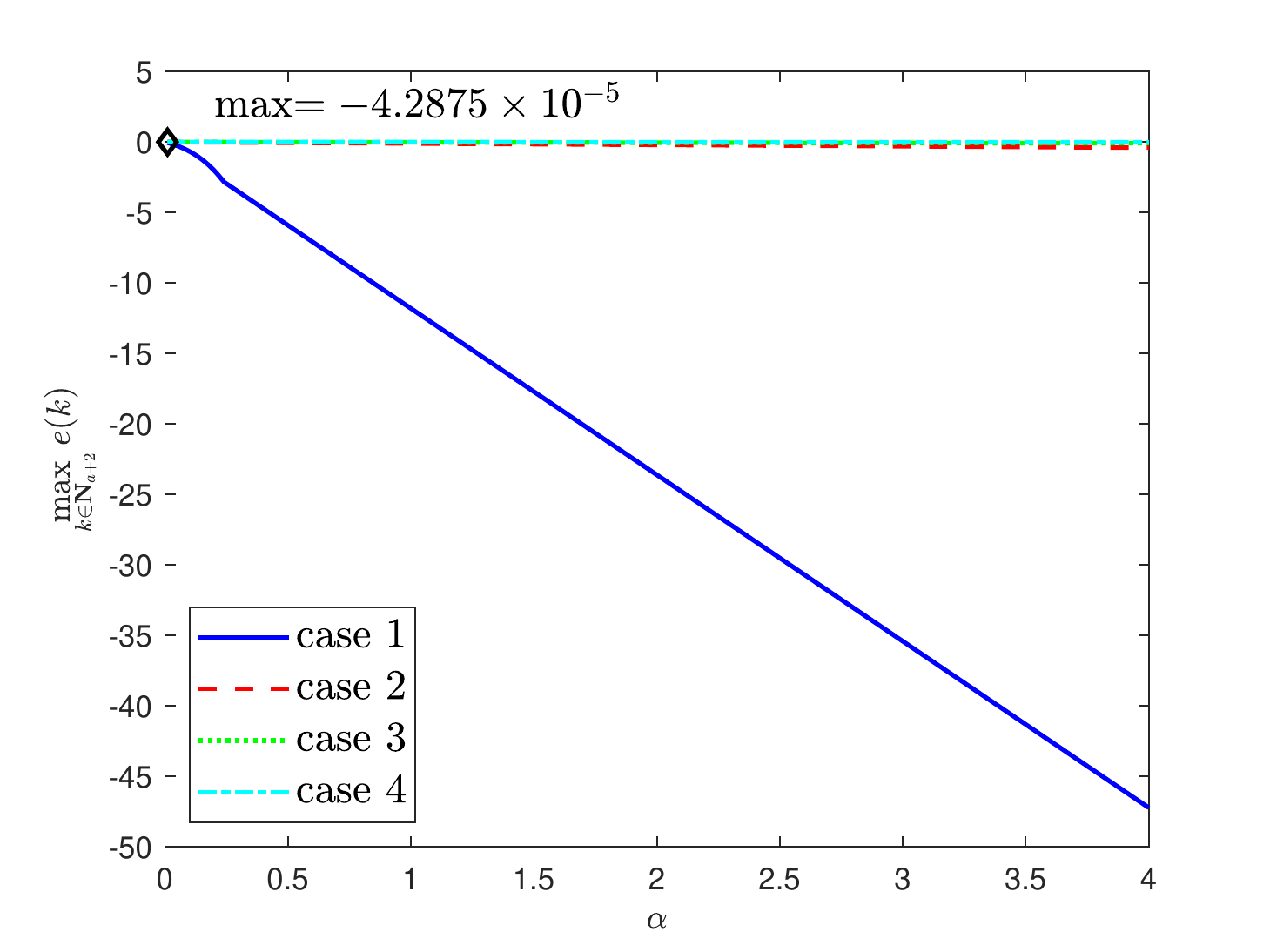}
	\label{Fig11a}
	\end{minipage}%
	}\hspace{-12pt}
	\subfigure[$w(k) = {( - 1)^{k - a}} - 2$]{
	\begin{minipage}[t]{0.5\linewidth}
	\includegraphics[width=1\hsize]{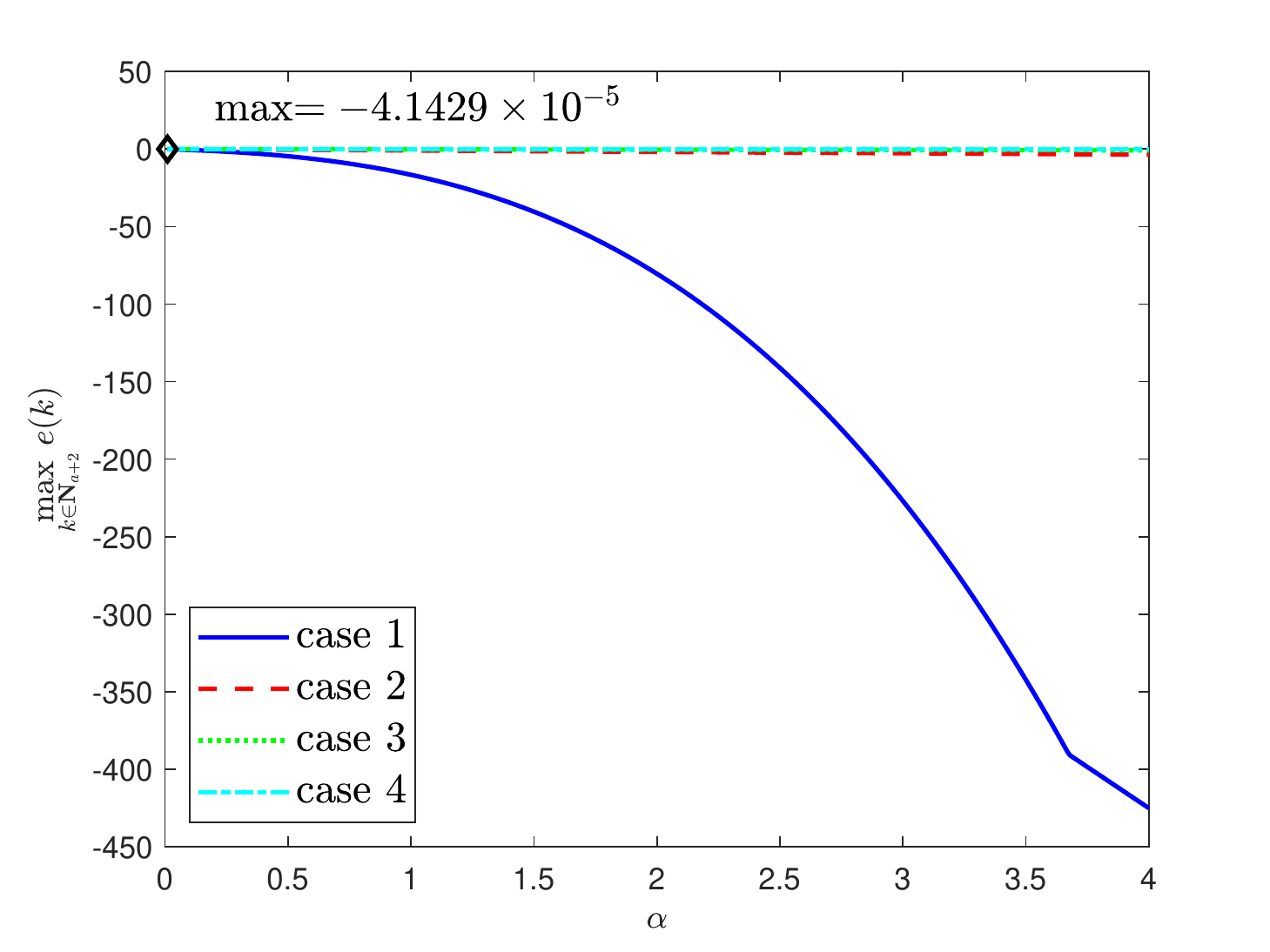}
	\label{Fig11b}
	\end{minipage}%
	}
    \subfigure[$w(k) = \sin (k - a) + 2$]{
	\begin{minipage}[t]{0.5\linewidth}
	\includegraphics[width=1\hsize]{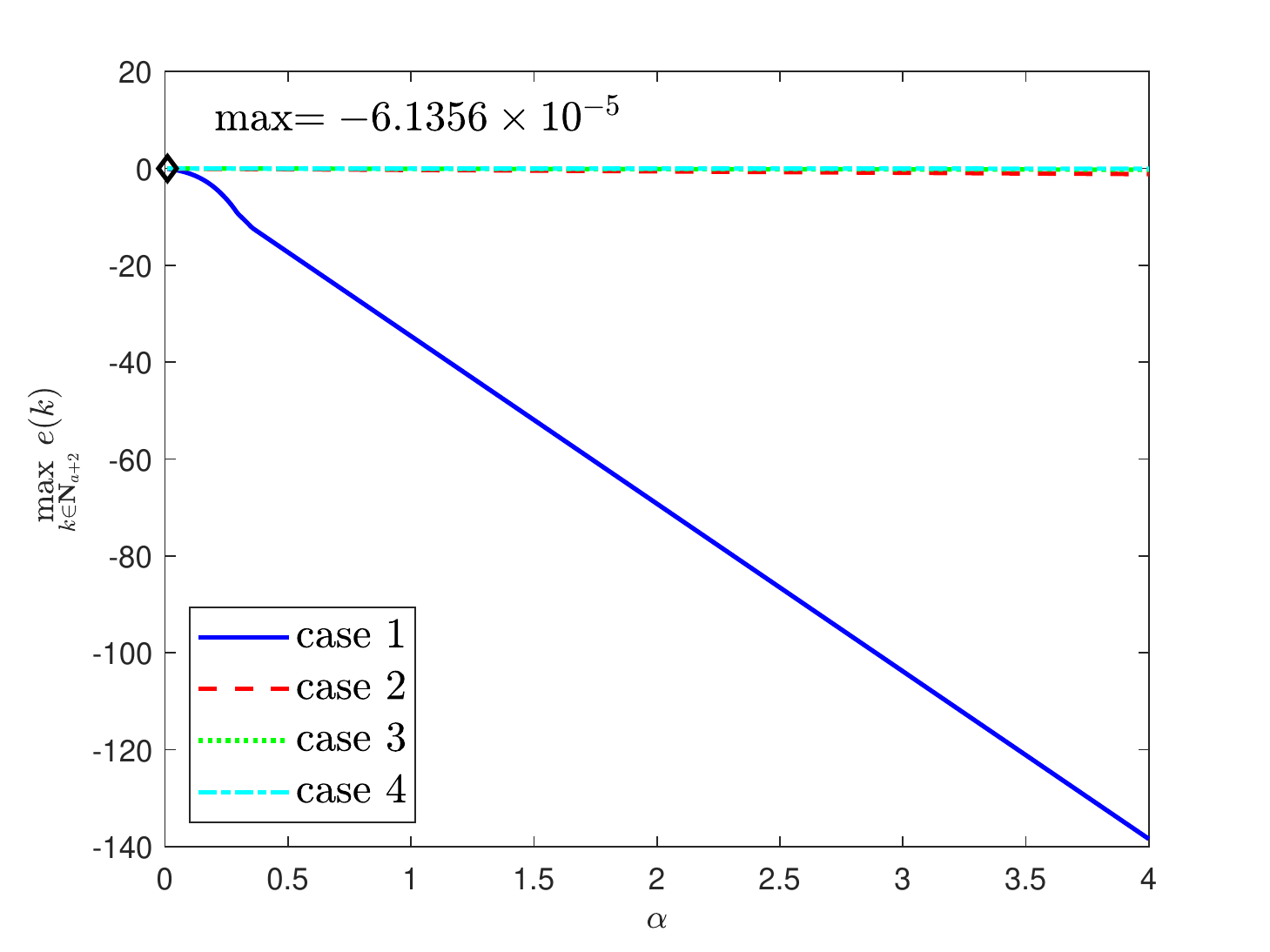}
	\label{Fig11c}
	\end{minipage}%
	}\hspace{-12pt}
	\subfigure[$w(k) = \sin (k - a) - 2$]{
	\begin{minipage}[t]{0.5\linewidth}
	\includegraphics[width=1\hsize]{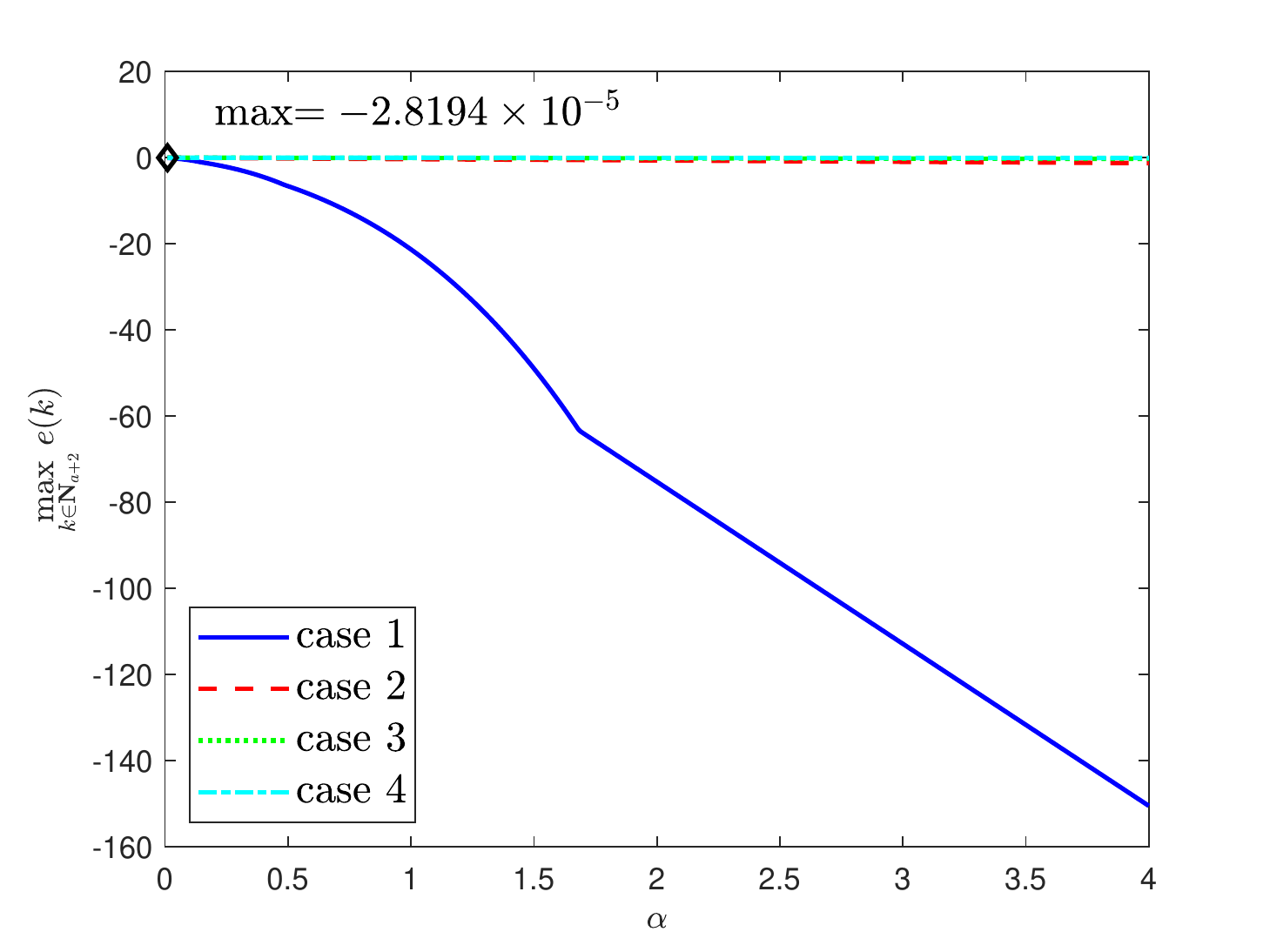}
	\label{Fig11d}
	\end{minipage}%
	}
	\centering
	\caption{The evolution of the error with respect to the order.}
	\label{Fig11}
\end{figure}

Notably, $u(k),v(k)$ are continuously oscillating in {\rm case 1} and {\rm case 2}. $u(k),v(k)$ are continuously oscillating functions combined with the additive white noise in {\rm case 3}. $u(k),v(k)$ are continuously oscillating functions combined with the multiplicative white noise in {\rm case 4}. The tempered functions $w(k)$ are selected as Example \ref{Example3}. It can be found that the maximum of $e(k)$ in all the mentioned cases are negative, which confirms the correctness of Theorem \ref{Theorem3.31} firmly. Since $\mathop {\lim }\limits_{\alpha  \to 0} { }_{a}^{\mathrm{G}} \nabla_{k}^{\alpha} x(k)=x(k)$, it follows $\mathop {\lim }\limits_{\alpha  \to 0} e(k)=0$ which coincides with the obtained simulated results. Besides, for given $u(k),v(k),w(k)$, with the increase of $\alpha$, the value of function $f(x)$ in (\ref{Eq3.178}) increases gradually. As a result, the discriminant $\Delta$ in (\ref{Eq3.180}) decreases gradually, which just matches the variation tendency of the plotted curves.
\end{example}

\begin{example}\label{Example5}
To examine the Jensen fractional sum inequality, define $e(k):=f\big([{}_a^{\rm G}\nabla _k^{ - \alpha,w(k) }1]^{-1}{}_a^{\rm G}\nabla _k^{ - \alpha,w(k) }x(k)\big)-[{}_a^{\rm G}\nabla _k^{ - \alpha ,w(k)}1]^{-1}{}_a^{\rm G}\nabla _k^{ - \alpha,w(k) }f(x(k))$. Setting $a=0$, $\alpha=0.01,0.02,\cdots,4$, $x(k)=-1:0.01:1$ and considering the following four cases
\[\left\{ \begin{array}{l}
{\rm{case}}\;1:\;f(x(k))=|x(k)|^{1.5};\\
{\rm{case}}\;2:\;f(x(k))=x^2(k);\\
{\rm{case}}\;3:\;f(x(k))={\rm e}^{0.5x(k)};\\
{\rm{case}}\;4:\;f(x(k))=\max\{|x(k)|^{1.5},x^2(k),{\rm e}^{0.5x(k)}\},
\end{array} \right.\]
the simulated results are shown in Figure \ref{Fig12} and Figure \ref{Fig13}.
\begin{figure}[!htbp]
  \centering
  \includegraphics[width=0.5\textwidth]{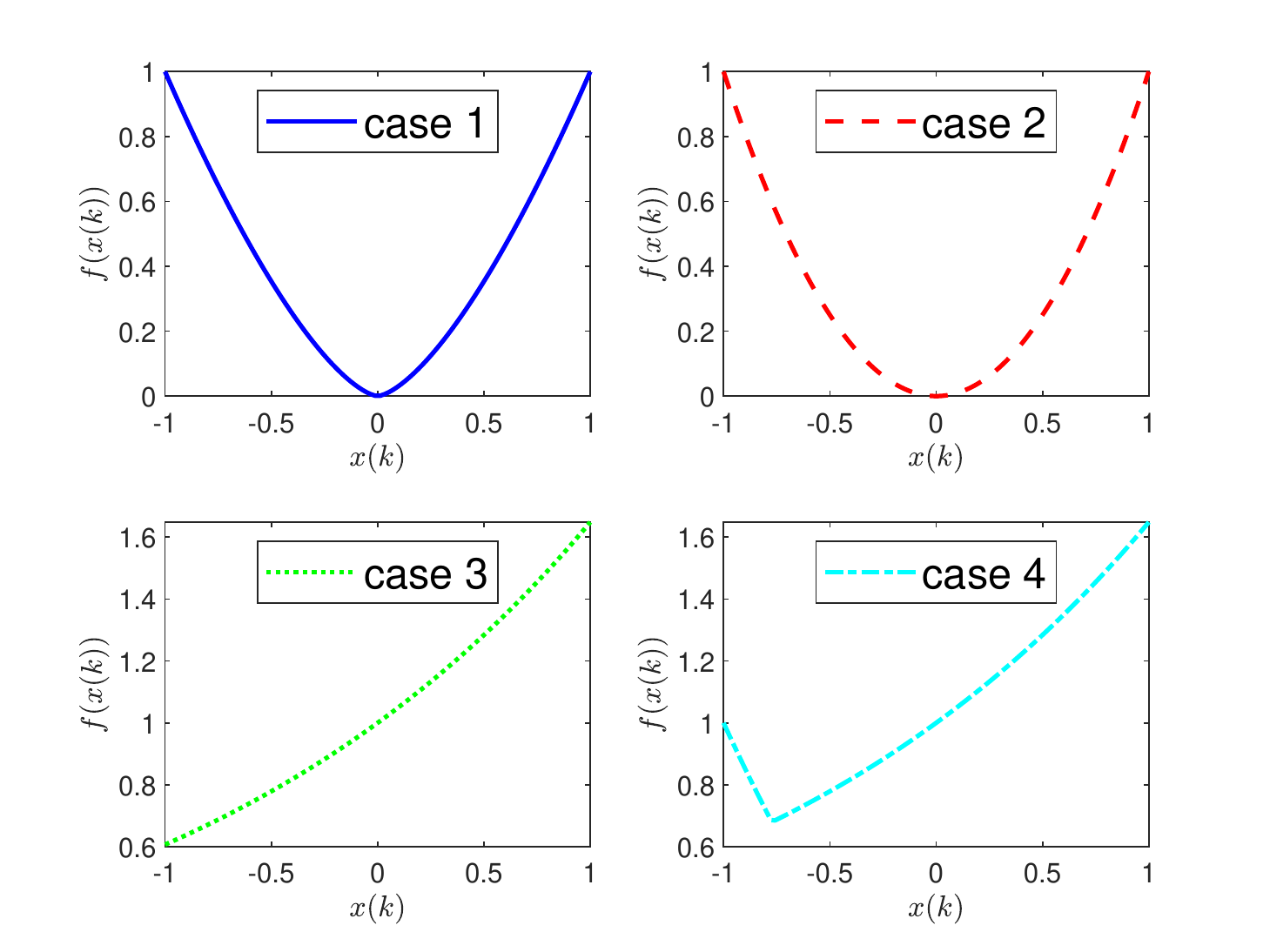}
  \caption{$f(x(k))$ with regard to $x(k)$. }\label{Fig12}
\end{figure}
\begin{figure}[!htbp]
	\centering
	\setlength{\abovecaptionskip}{-2pt}
	\vspace{-10pt}
	\subfigtopskip=-2pt
	\subfigbottomskip=2pt
	\subfigcapskip=-2pt
	\subfigure[$w(k) = {( - 1)^{k - a}} + 2$]{
	\begin{minipage}[t]{0.5\linewidth}
	\includegraphics[width=1\hsize]{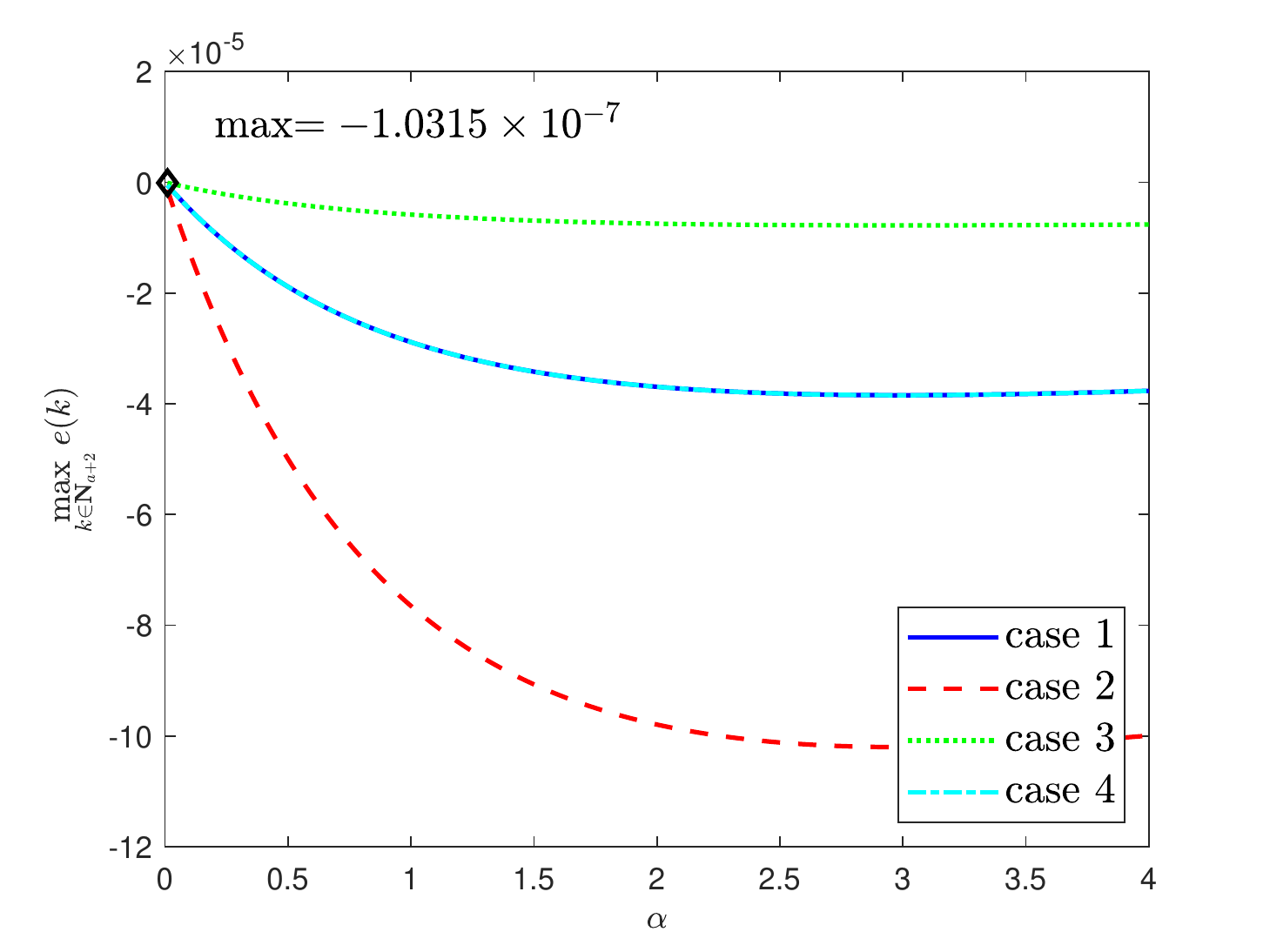}
	\label{Fig13a}
	\end{minipage}%
	}\hspace{-12pt}
	\subfigure[$w(k) = {( - 1)^{k - a}} - 2$]{
	\begin{minipage}[t]{0.5\linewidth}
	\includegraphics[width=1\hsize]{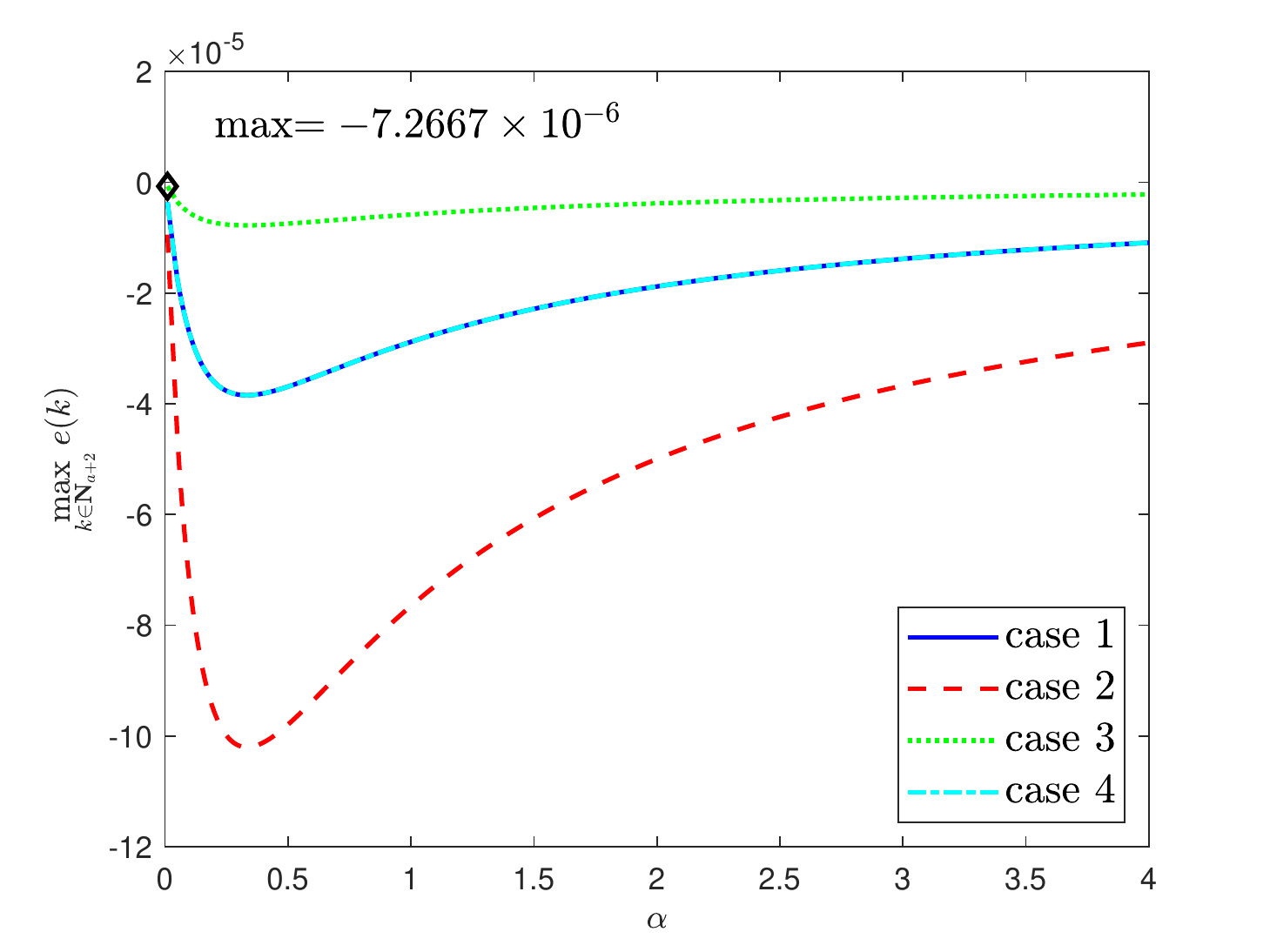}
	\label{Fig13b}
	\end{minipage}%
	}
    \subfigure[$w(k) = \sin (k - a) + 2$]{
	\begin{minipage}[t]{0.5\linewidth}
	\includegraphics[width=1\hsize]{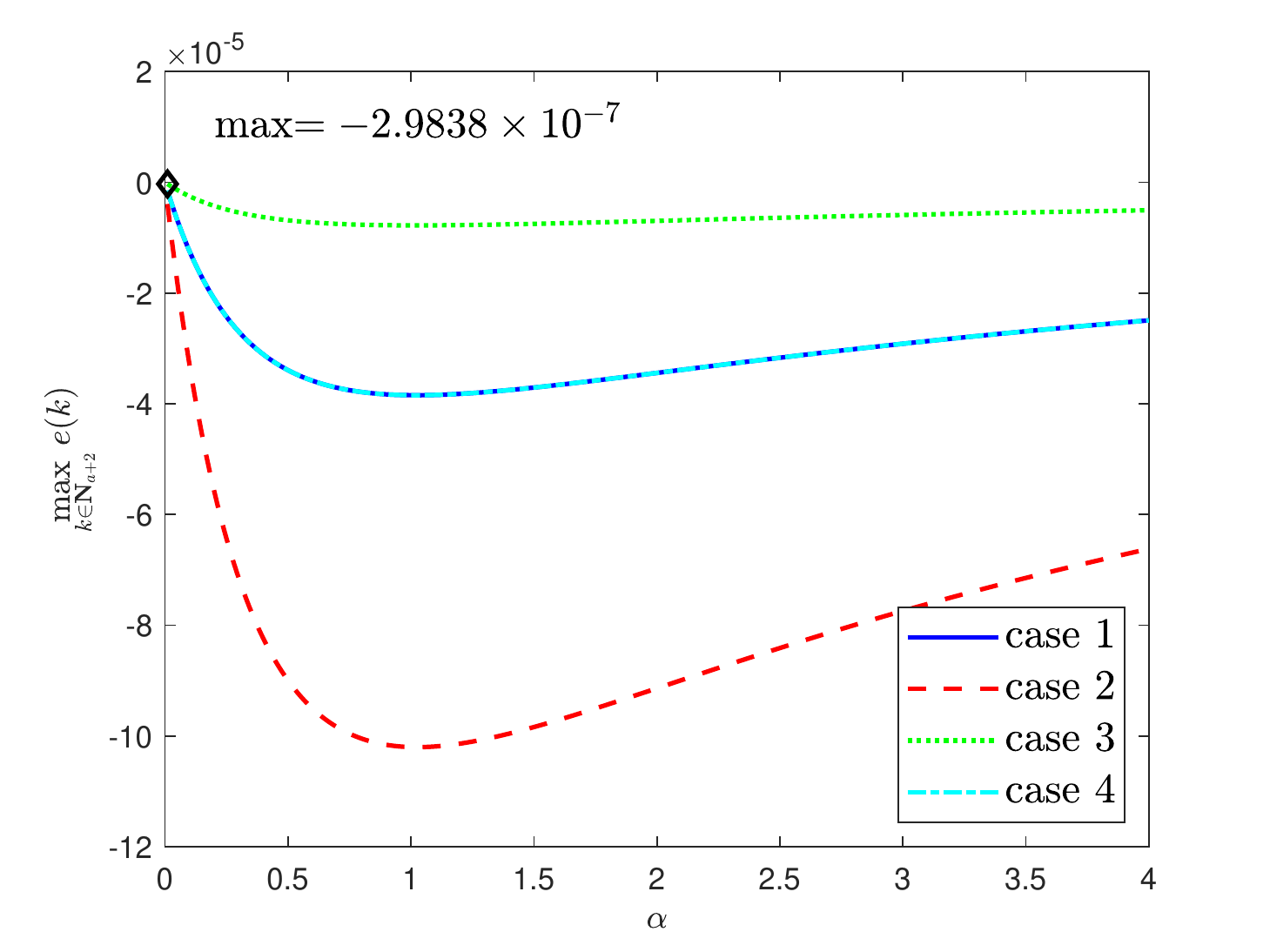}
	\label{Fig13c}
	\end{minipage}%
	}\hspace{-12pt}
	\subfigure[$w(k) = \sin (k - a) - 2$]{
	\begin{minipage}[t]{0.5\linewidth}
	\includegraphics[width=1\hsize]{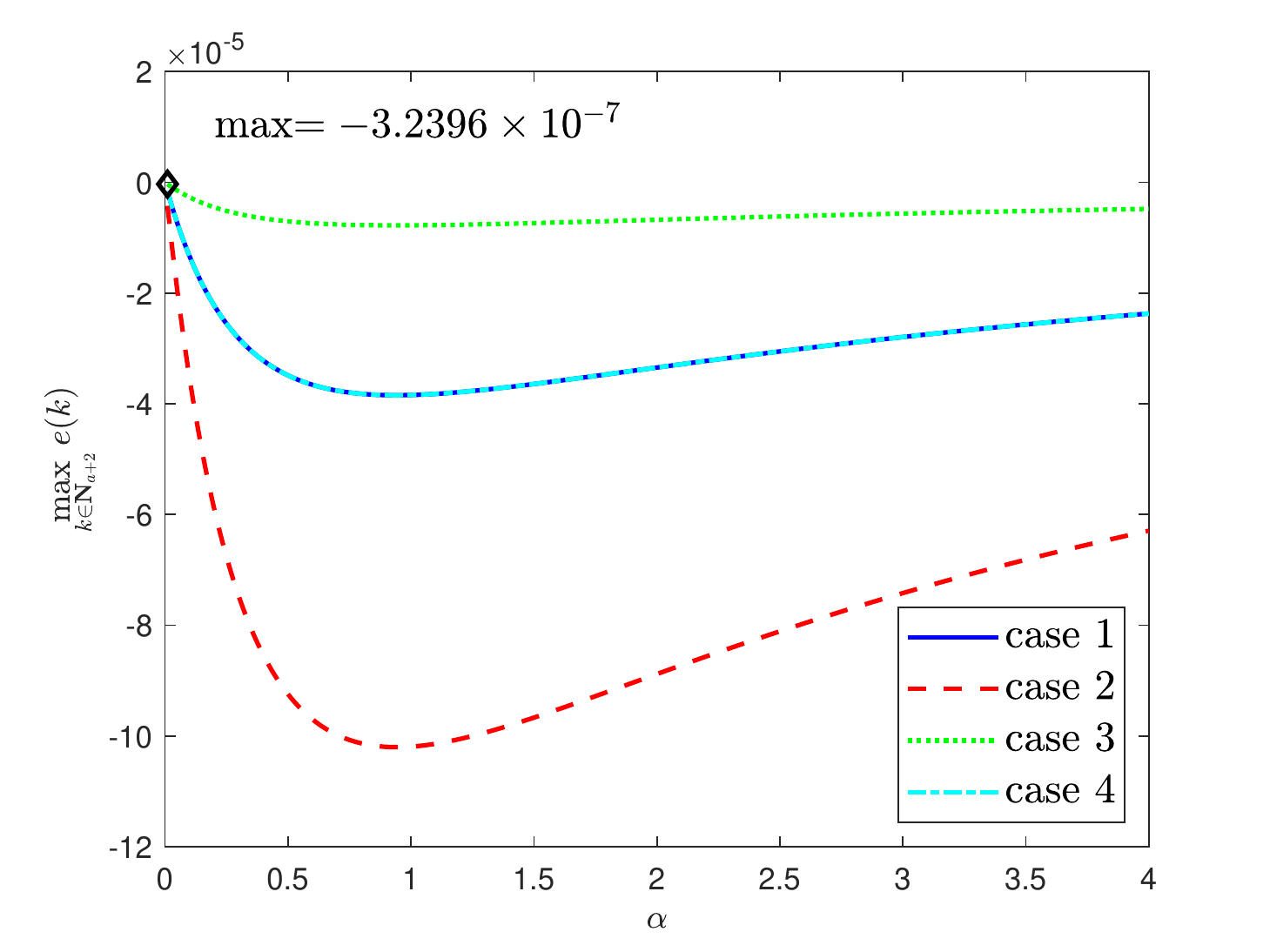}
	\label{Fig13d}
	\end{minipage}%
	}
	\centering
	\caption{The evolution of the error with respect to the order.}
	\label{Fig13}
\end{figure}

From Figure \ref{Fig12}, it can be found that all the mentioned four functions are all convex. The maximum of $e(k)$ in all the mentioned cases are negative, which coincides with the theoretical results in Theorem \ref{Theorem3.32} firmly.
\end{example}

\begin{example}\label{Example6}
To examine the H\"{o}lder fractional sum inequality, define $e(k):={}_a^{\rm G}\nabla _k^{ - \alpha,w(k) }[u(k)v(k)]-[{}_a^{\rm G}\nabla _k^{ - \alpha,w(k) }u^p(k)]^{\frac{1}{p}}[{}_a^{\rm G}\nabla _k^{ - \alpha,w(k) }v^q(k)]^{\frac{1}{q}}$. Setting $a=0$, $\alpha=0.01,0.02,\cdots,4$ and considering the following four cases
\[\left\{ \begin{array}{l}
{\rm{case}}\;1:\;u(k)=1+0.5\sin (10 k),v(k)=1+0.5\cos (10 k);\\
{\rm{case}}\;2:\;u(k)=2+(-1)^{k-a},v(k)=2-(-1)^{k-a};\\
{\rm{case}}\;3:\;u(k)=|\sin(10k)|+0.01,v(k)=|\cos(10k)|+0.01;\\
{\rm{case}}\;4:\;u(k)=|\texttt{randn}(\texttt{size}(k))|+0.01, v(k)=|\texttt{randn}(\texttt{size}(k))|+0.01,
\end{array} \right.\]
the simulated results are shown in Figure \ref{Fig14}-Figure \ref{Fig16}.
\begin{figure}[!htbp]
  \centering
  \includegraphics[width=0.5\textwidth]{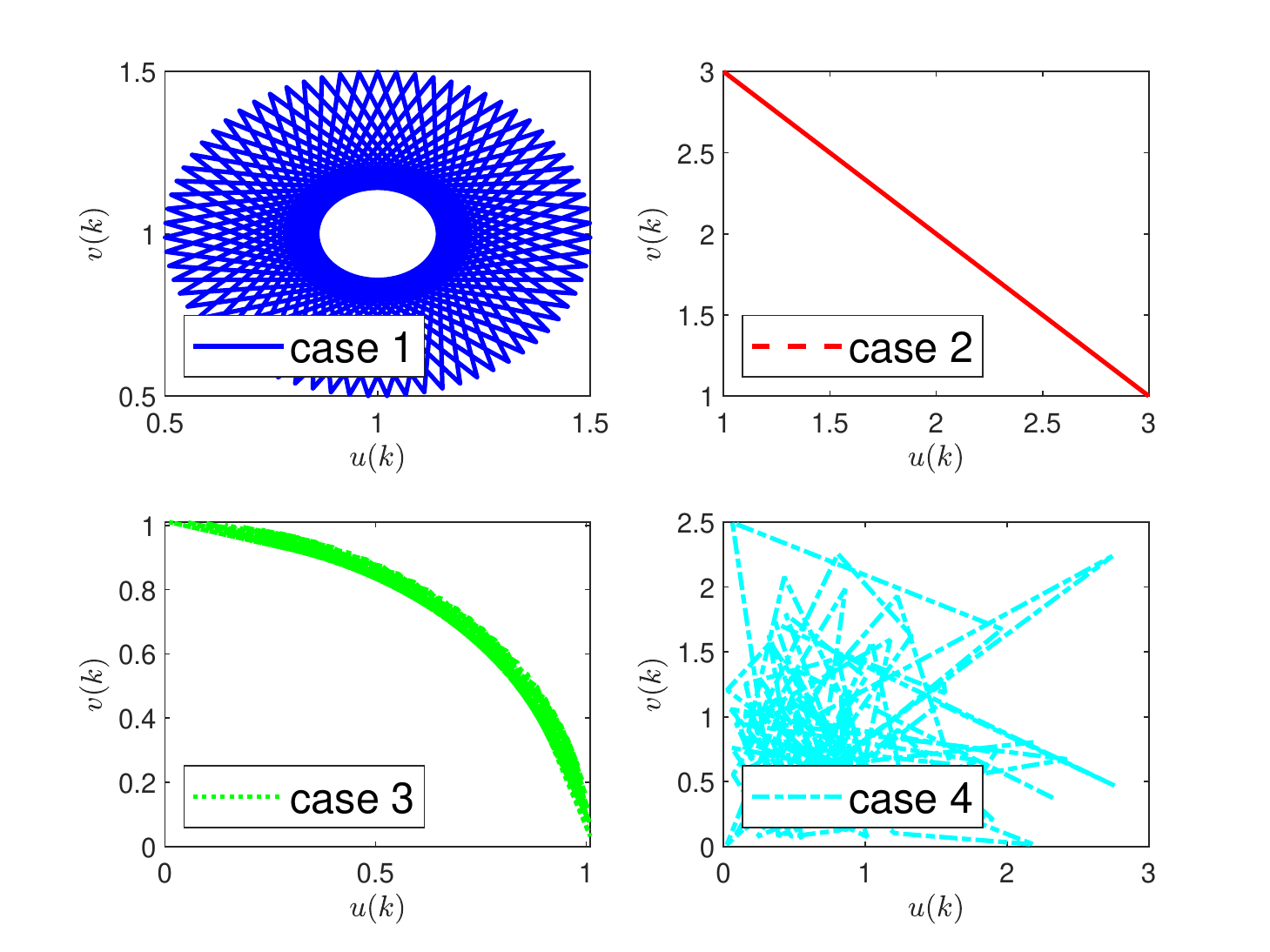}
  \caption{The diagram of $u(k)$ and $v(k)$}\label{Fig14}
\end{figure}
\begin{figure}[!htbp]
	\centering
	\setlength{\abovecaptionskip}{-2pt}
	\vspace{-10pt}
	\subfigtopskip=-2pt
	\subfigbottomskip=2pt
	\subfigcapskip=-2pt
	\subfigure[$w(k) = {( - 1)^{k - a}} + 2$]{
	\begin{minipage}[t]{0.5\linewidth}
	\includegraphics[width=1\hsize]{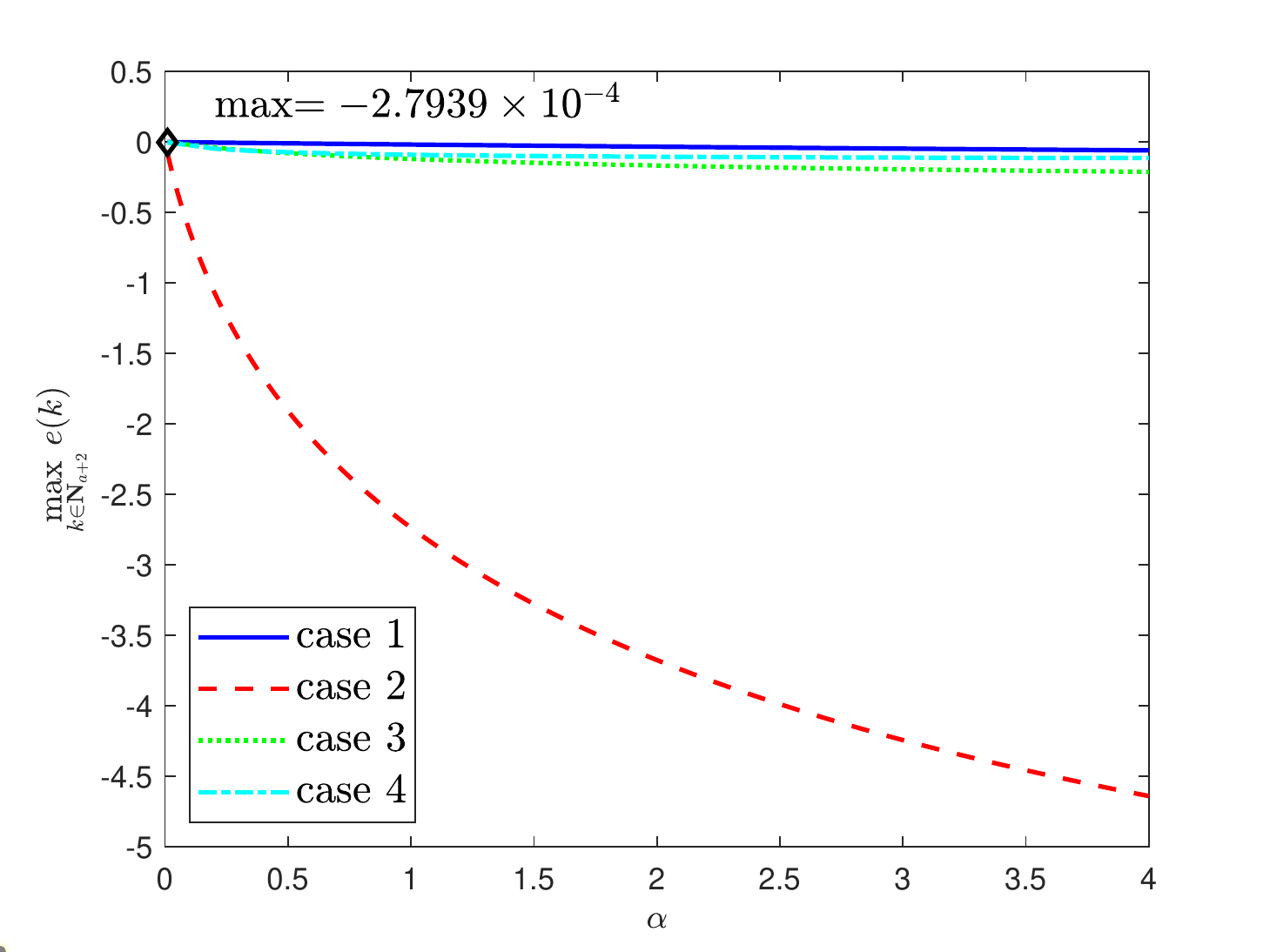}
	\label{Fig15a}
	\end{minipage}%
	}\hspace{-12pt}
	\subfigure[$w(k) = {( - 1)^{k - a}} - 2$]{
	\begin{minipage}[t]{0.5\linewidth}
	\includegraphics[width=1\hsize]{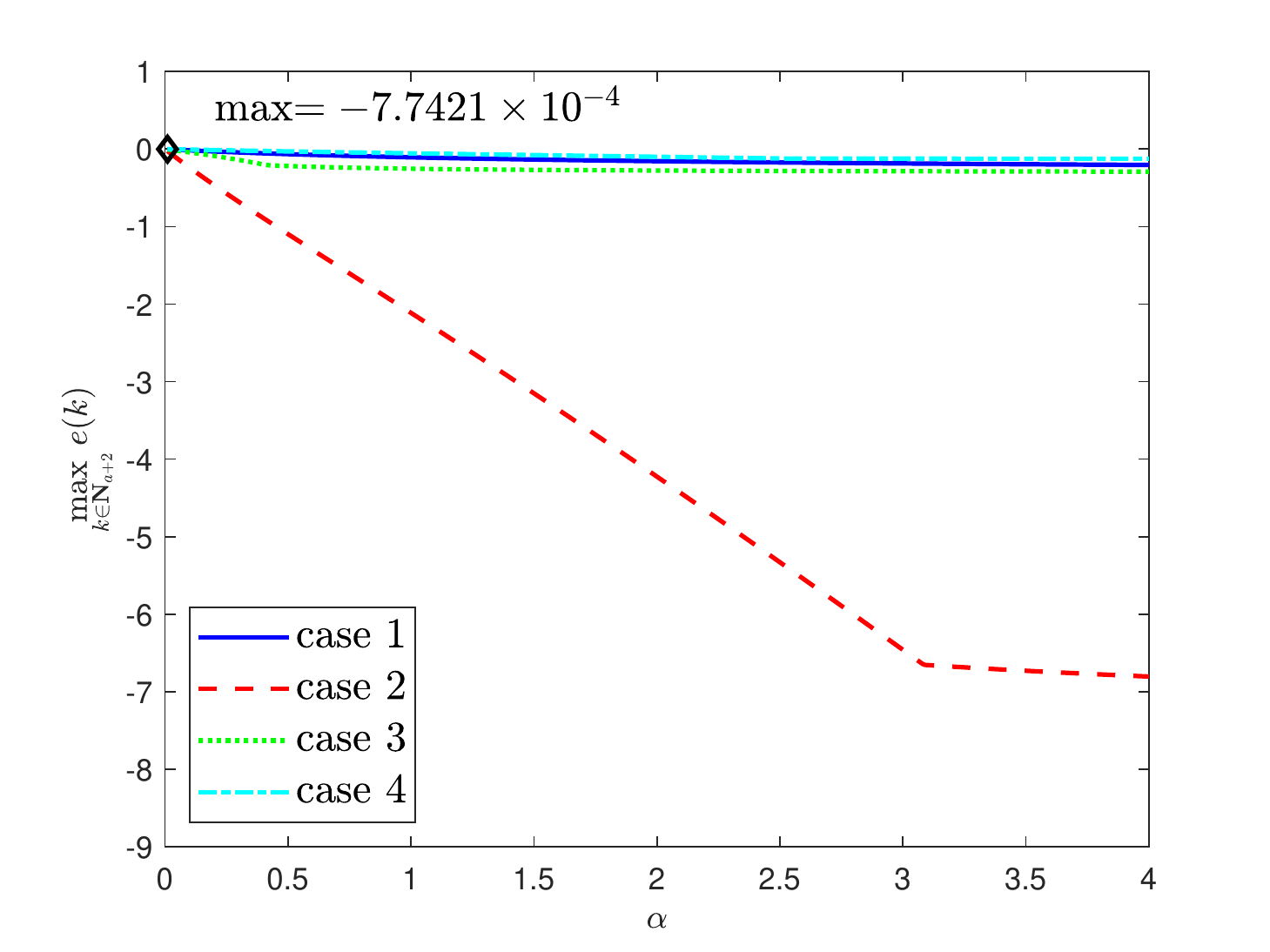}
	\label{Fig15b}
	\end{minipage}%
	}
    \subfigure[$w(k) = \sin (k - a) + 2$]{
	\begin{minipage}[t]{0.5\linewidth}
	\includegraphics[width=1\hsize]{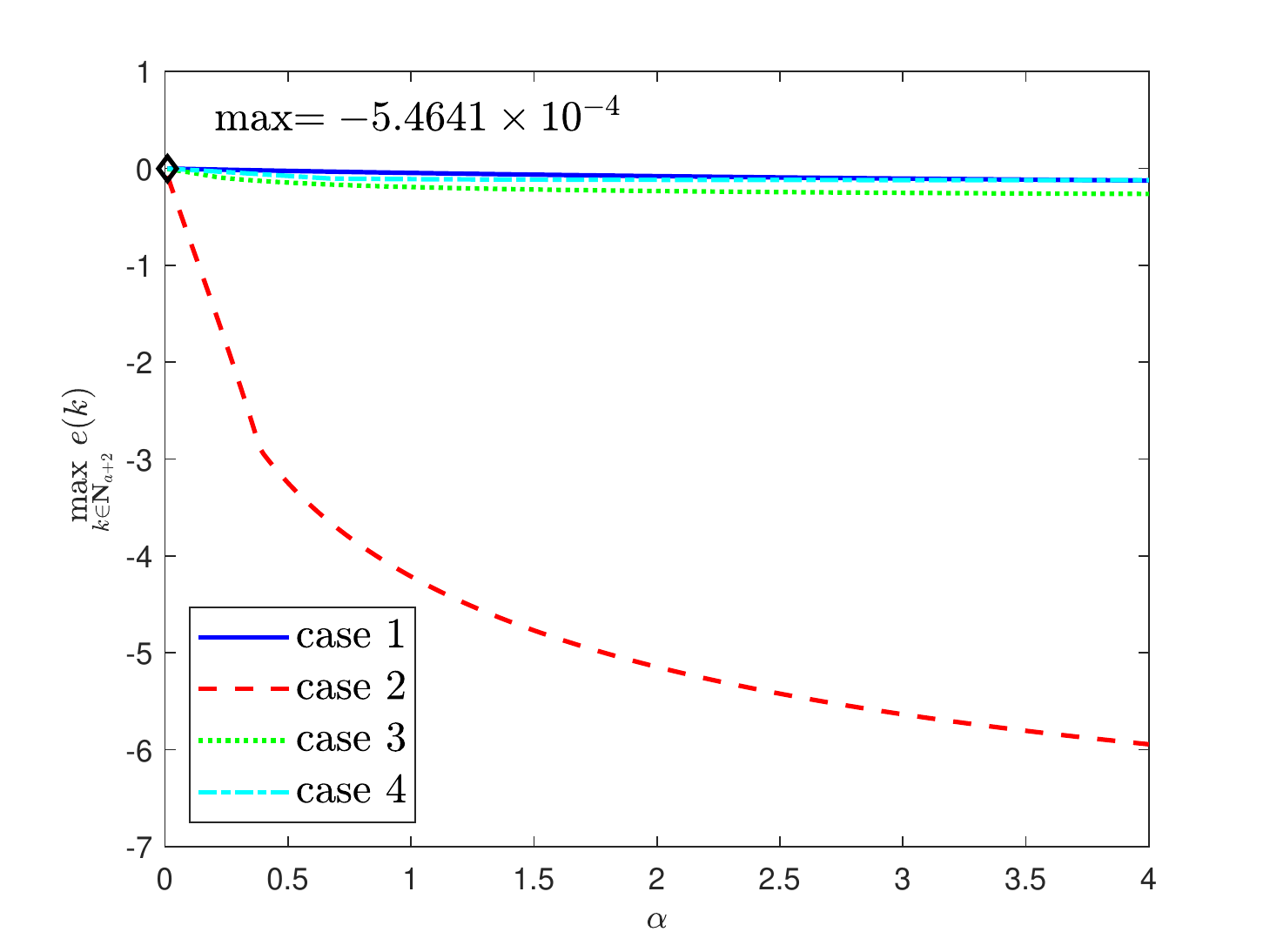}
	\label{Fig15c}
	\end{minipage}%
	}\hspace{-12pt}
	\subfigure[$w(k) = \sin (k - a) - 2$]{
	\begin{minipage}[t]{0.5\linewidth}
	\includegraphics[width=1\hsize]{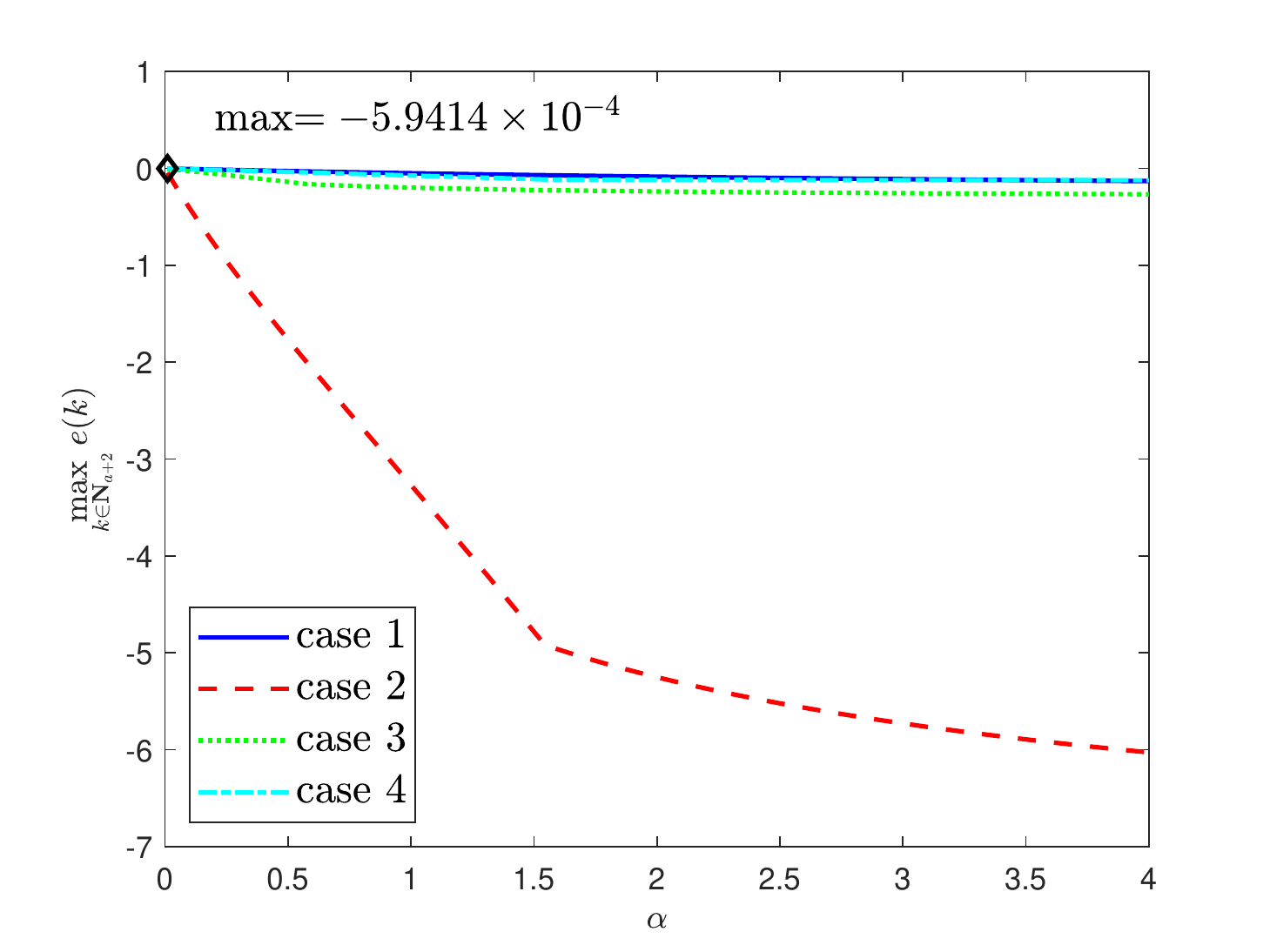}
	\label{Fig15d}
	\end{minipage}%
	}
	\centering
	\caption{The evolution of the error with respect to the order ($p=1.5$).}
	\label{Fig15}
\end{figure}
\begin{figure}[!htbp]
	\centering
	\setlength{\abovecaptionskip}{-2pt}
	\vspace{-10pt}
	\subfigtopskip=-2pt
	\subfigbottomskip=2pt
	\subfigcapskip=-2pt
	\subfigure[$w(k) = {( - 1)^{k - a}} + 2$]{
	\begin{minipage}[t]{0.5\linewidth}
	\includegraphics[width=1\hsize]{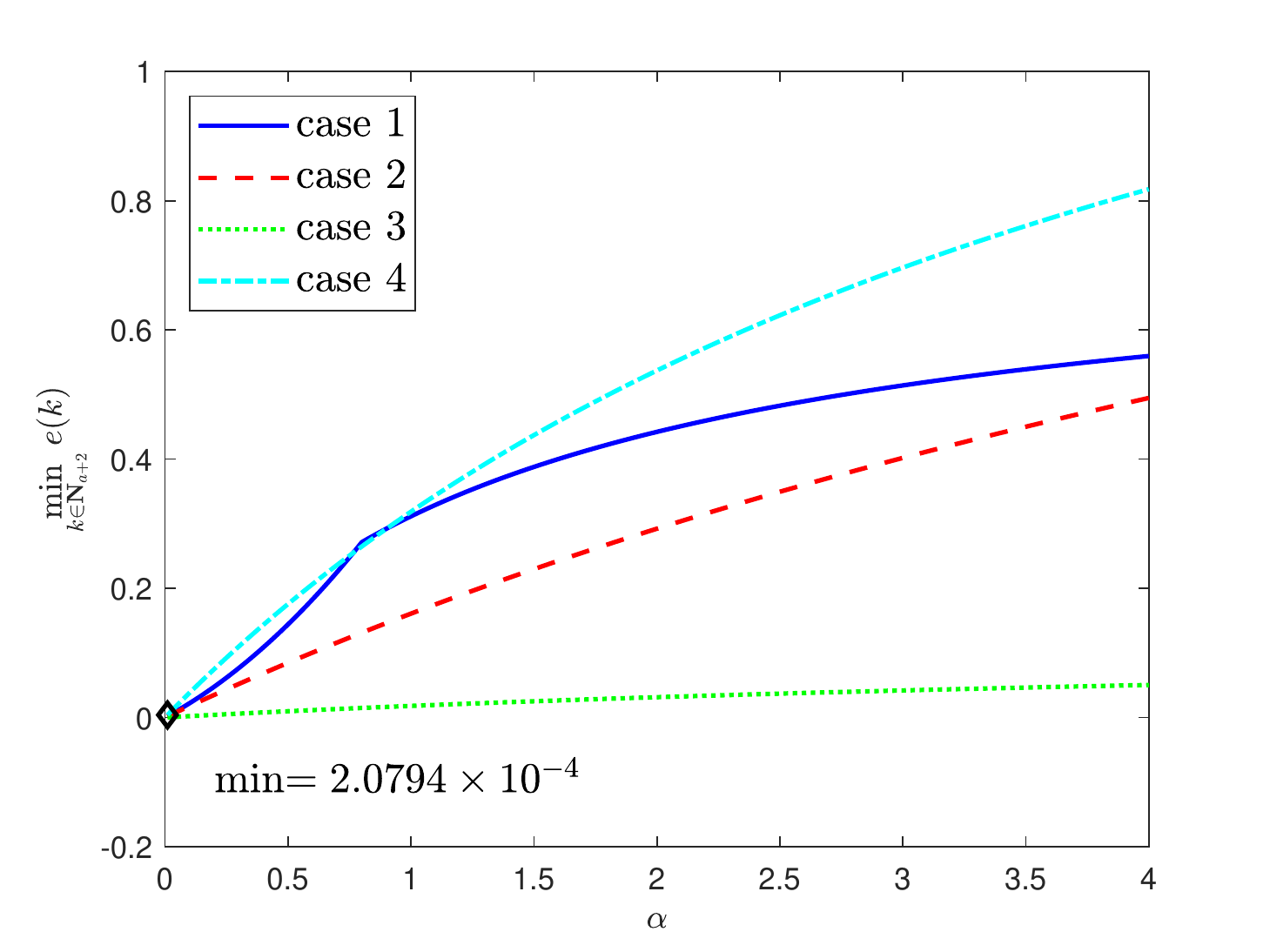}
	\label{Fig16a}
	\end{minipage}%
	}\hspace{-12pt}
	\subfigure[$w(k) = {( - 1)^{k - a}} - 2$]{
	\begin{minipage}[t]{0.5\linewidth}
	\includegraphics[width=1\hsize]{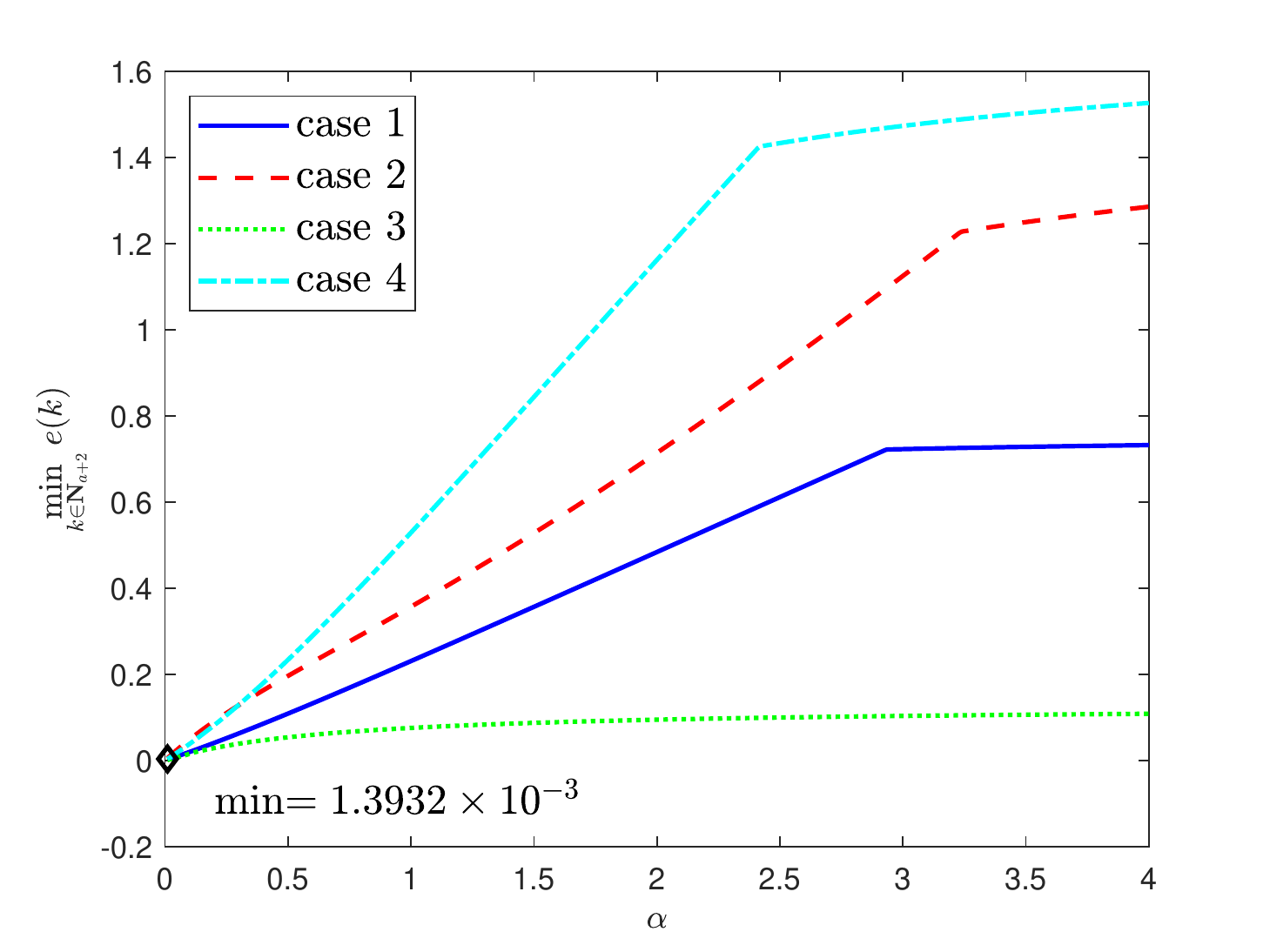}
	\label{Fig16b}
	\end{minipage}%
	}
    \subfigure[$w(k) = \sin (k - a) + 2$]{
	\begin{minipage}[t]{0.5\linewidth}
	\includegraphics[width=1\hsize]{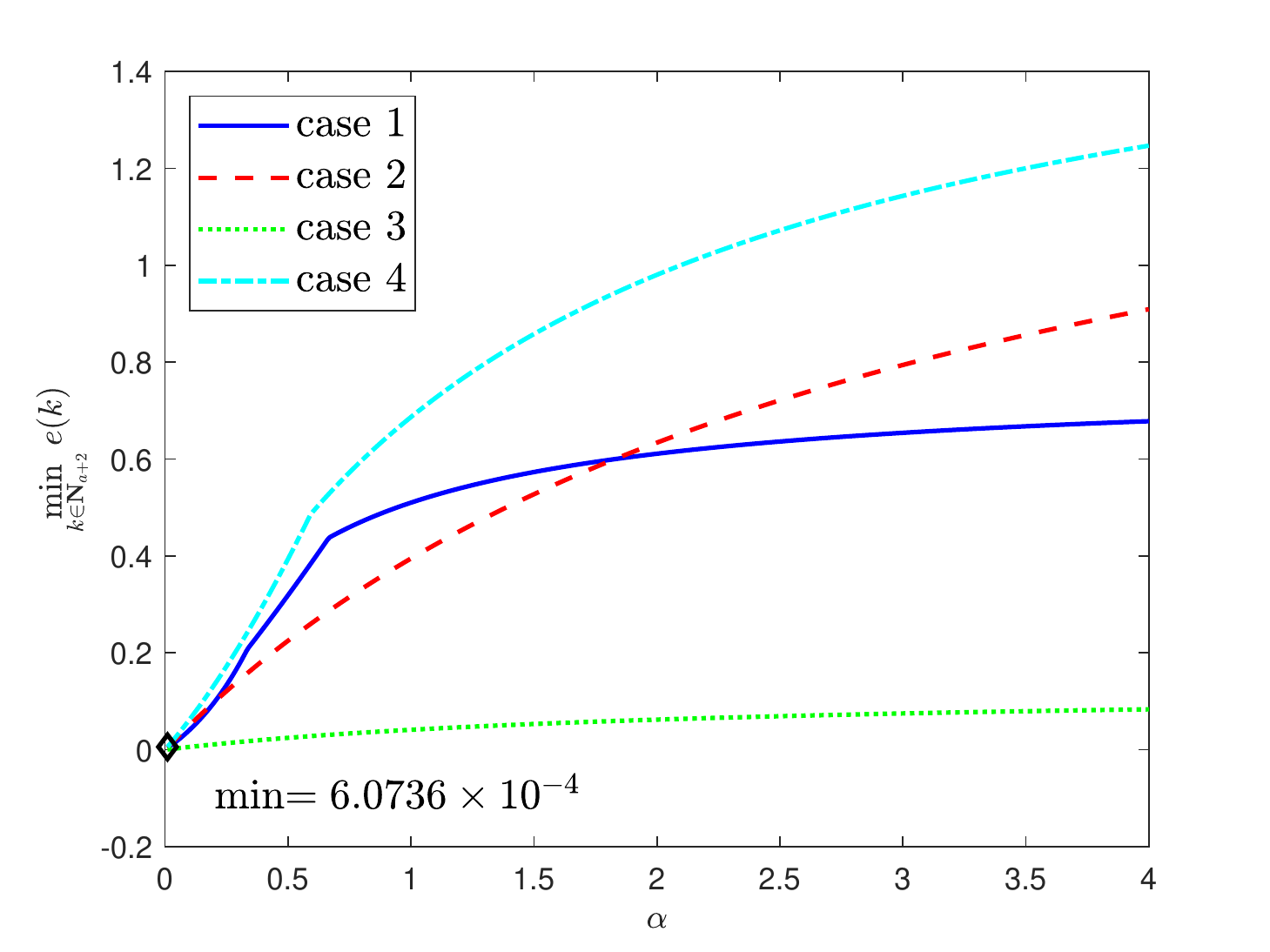}
	\label{Fig16c}
	\end{minipage}%
	}\hspace{-12pt}
	\subfigure[$w(k) = \sin (k - a) - 2$]{
	\begin{minipage}[t]{0.5\linewidth}
	\includegraphics[width=1\hsize]{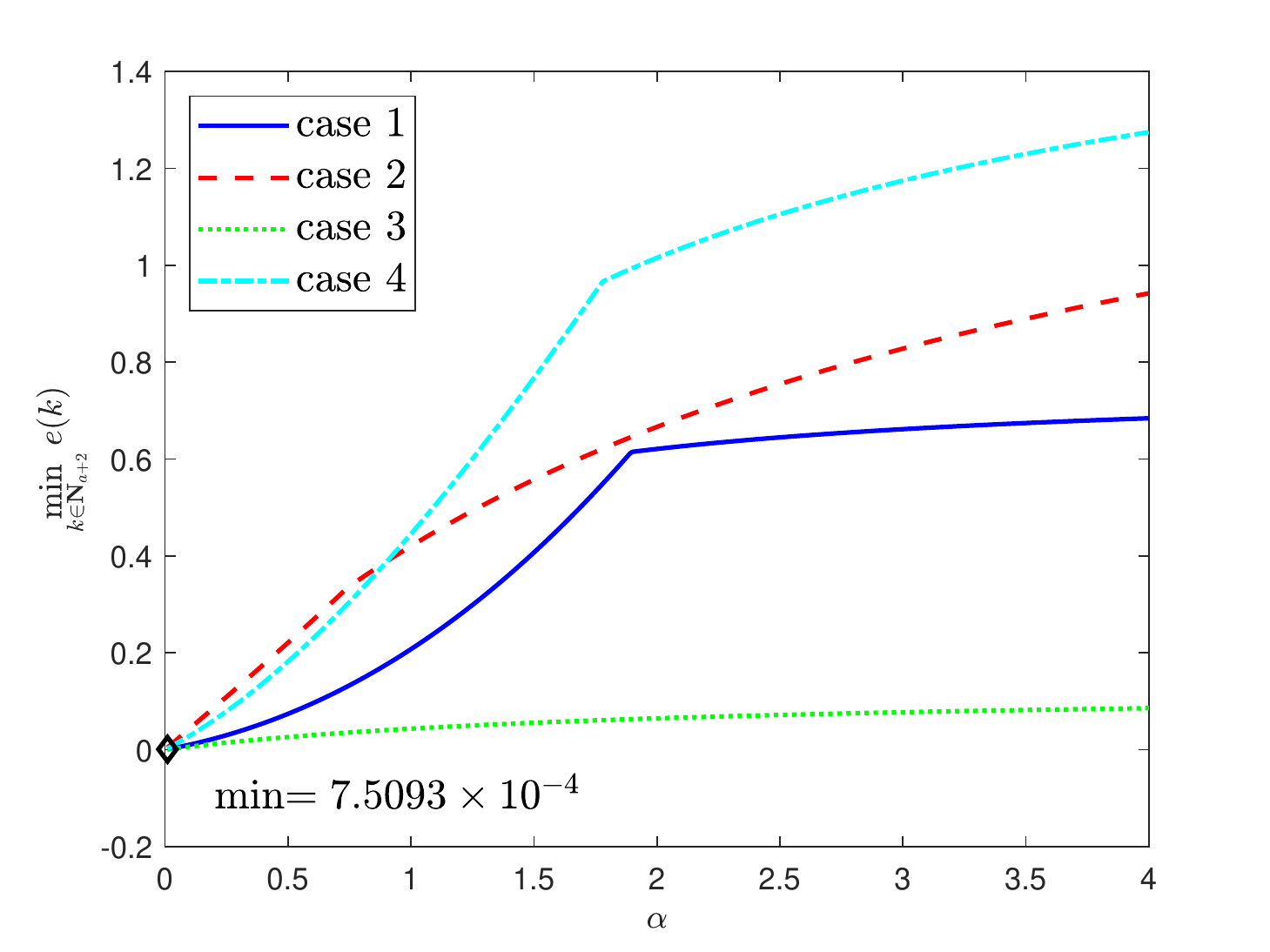}
	\label{Fig16d}
	\end{minipage}%
	}
	\centering
	\caption{The evolution of the error with respect to the order ($p=0.5$).}
	\label{Fig16}
\end{figure}

From the diagram in Figure \ref{Fig15}, it can be observed that $u(k),v(k)$ are positive, which satisfies the needed conditions. When $p=1.5$, the maximum of $e(k)$ is negative. With the increase of $\alpha$, $e(k)$ gets smaller and smaller. When $p=0.5$, the minimum of $e(k)$ is positive. With the increase of $\alpha$, $e(k)$ gets bigger and bigger. All of these illustrate the correctness of Theorem \ref{Theorem3.33}.
\end{example}

\begin{example}\label{Example7}
To examine the Minkonski fractional sum inequality, define $e(k):=\{{}_a^{\rm G}\nabla _k^{ - \alpha ,w(k)}[u(k)+v(k)]^p\}^{\frac{1}{p}}-[{}_a^{\rm G}\nabla _k^{ - \alpha,w(k) }u^p(k)]^{\frac{1}{p}}-[{}_a^{\rm G}\nabla _k^{ - \alpha,w(k) }v^p(k)]^{\frac{1}{p}}$. Setting $a=0$, $\alpha=0.01,0.02,\cdots,4$ and considering the mentioned four cases in Example \ref{Example6}, the simulated results are shown in Figure \ref{Fig17} and Figure \ref{Fig18}.
\begin{figure}[!htbp]
	\centering
	\setlength{\abovecaptionskip}{-2pt}
	\vspace{-10pt}
	\subfigtopskip=-2pt
	\subfigbottomskip=2pt
	\subfigcapskip=-2pt
	\subfigure[$w(k) = {( - 1)^{k - a}} + 2$]{
	\begin{minipage}[t]{0.5\linewidth}
	\includegraphics[width=1\hsize]{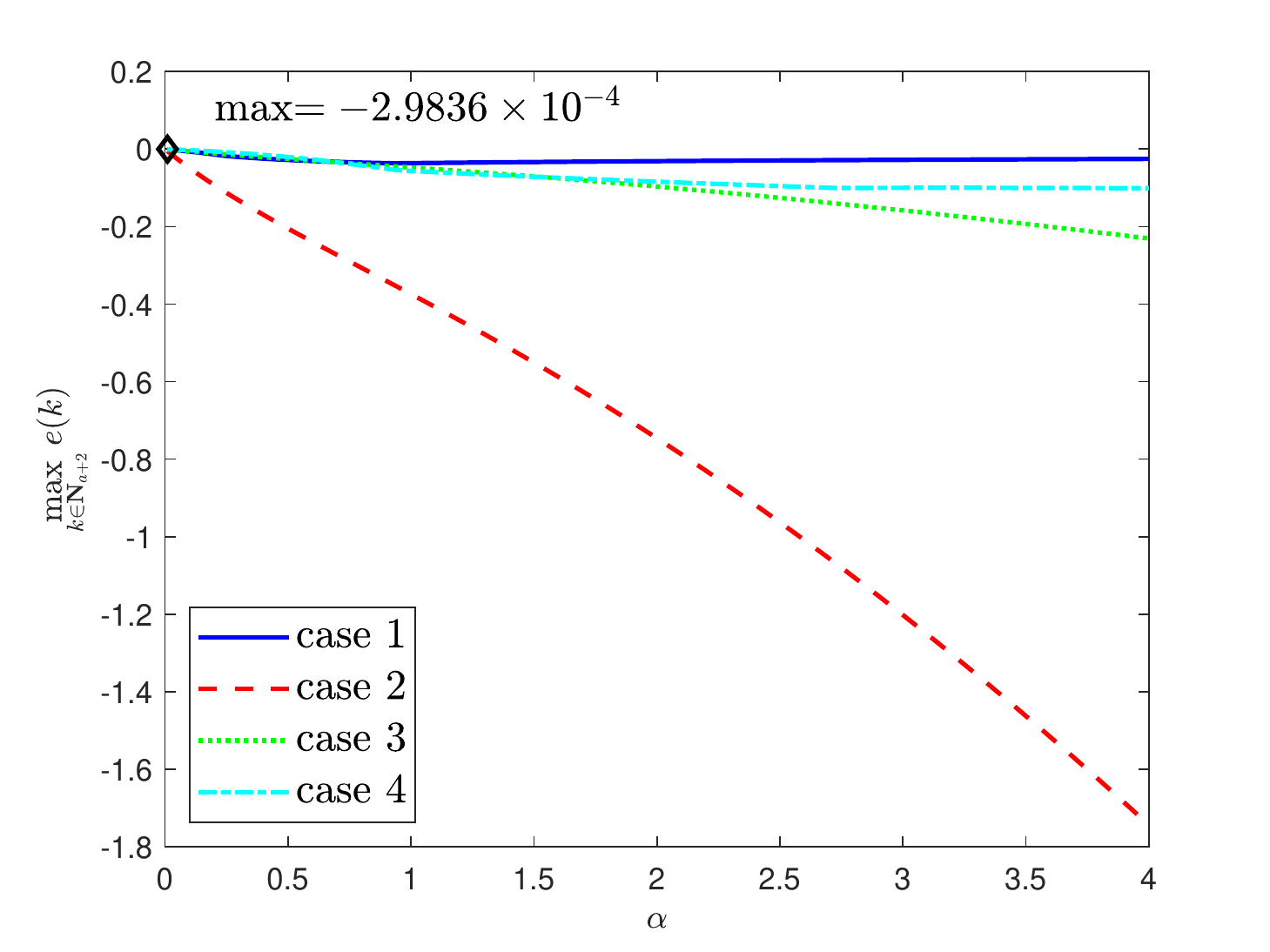}
	\label{Fig17a}
	\end{minipage}%
	}\hspace{-12pt}
	\subfigure[$w(k) = {( - 1)^{k - a}} - 2$]{
	\begin{minipage}[t]{0.5\linewidth}
	\includegraphics[width=1\hsize]{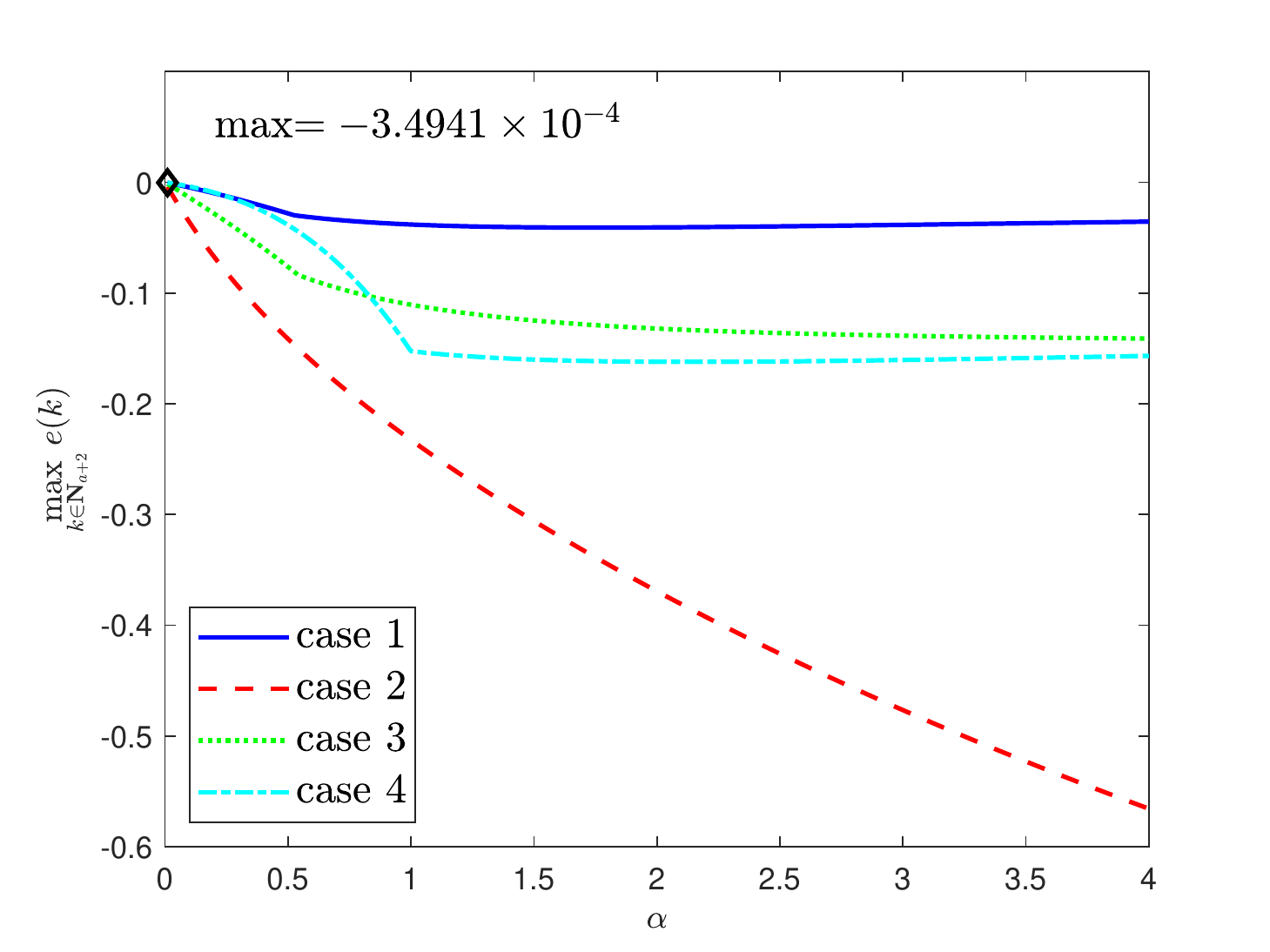}
	\label{Fig17b}
	\end{minipage}%
	}
    \subfigure[$w(k) = \sin (k - a) + 2$]{
	\begin{minipage}[t]{0.5\linewidth}
	\includegraphics[width=1\hsize]{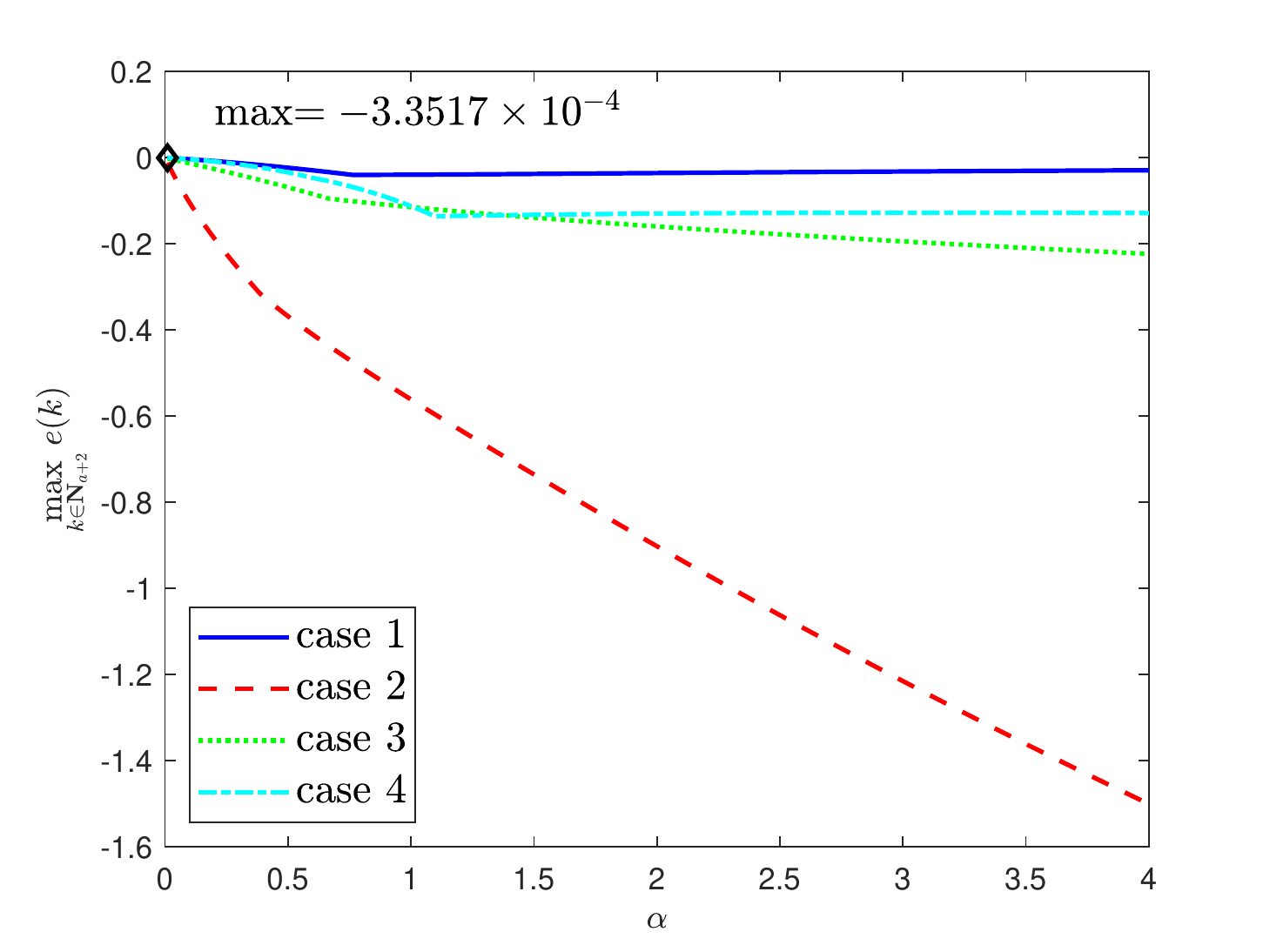}
	\label{Fig17c}
	\end{minipage}%
	}\hspace{-12pt}
	\subfigure[$w(k) = \sin (k - a) - 2$]{
	\begin{minipage}[t]{0.5\linewidth}
	\includegraphics[width=1\hsize]{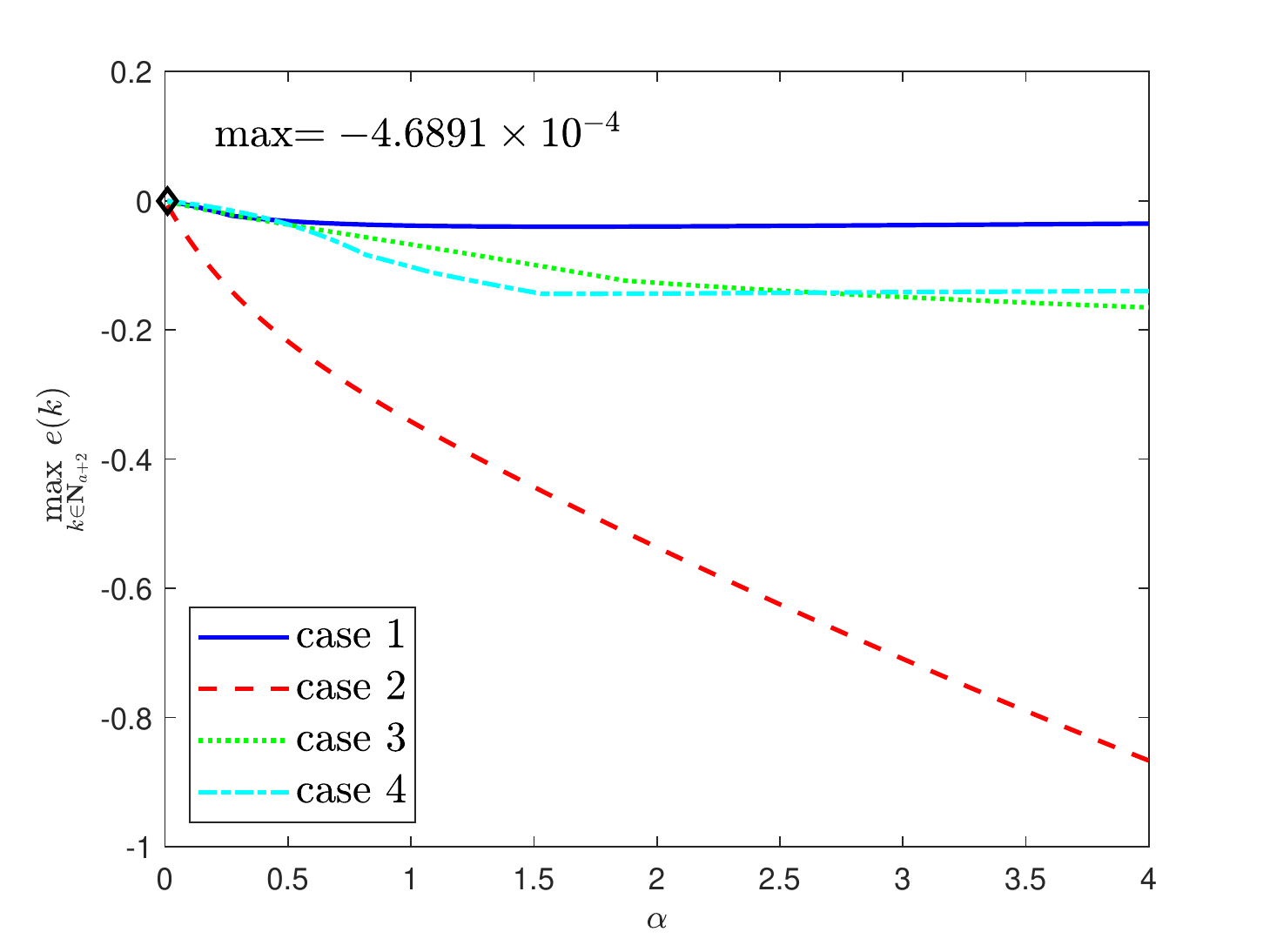}
	\label{Fig17d}
	\end{minipage}%
	}
	\centering
	\caption{The evolution of the error with respect to the order ($p=1.5$).}
	\label{Fig17}
\end{figure}
\begin{figure}[!htbp]
	\centering
	\setlength{\abovecaptionskip}{-2pt}
	\vspace{-10pt}
	\subfigtopskip=-2pt
	\subfigbottomskip=2pt
	\subfigcapskip=-2pt
	\subfigure[$w(k) = {( - 1)^{k - a}} + 2$]{
	\begin{minipage}[t]{0.5\linewidth}
	\includegraphics[width=1\hsize]{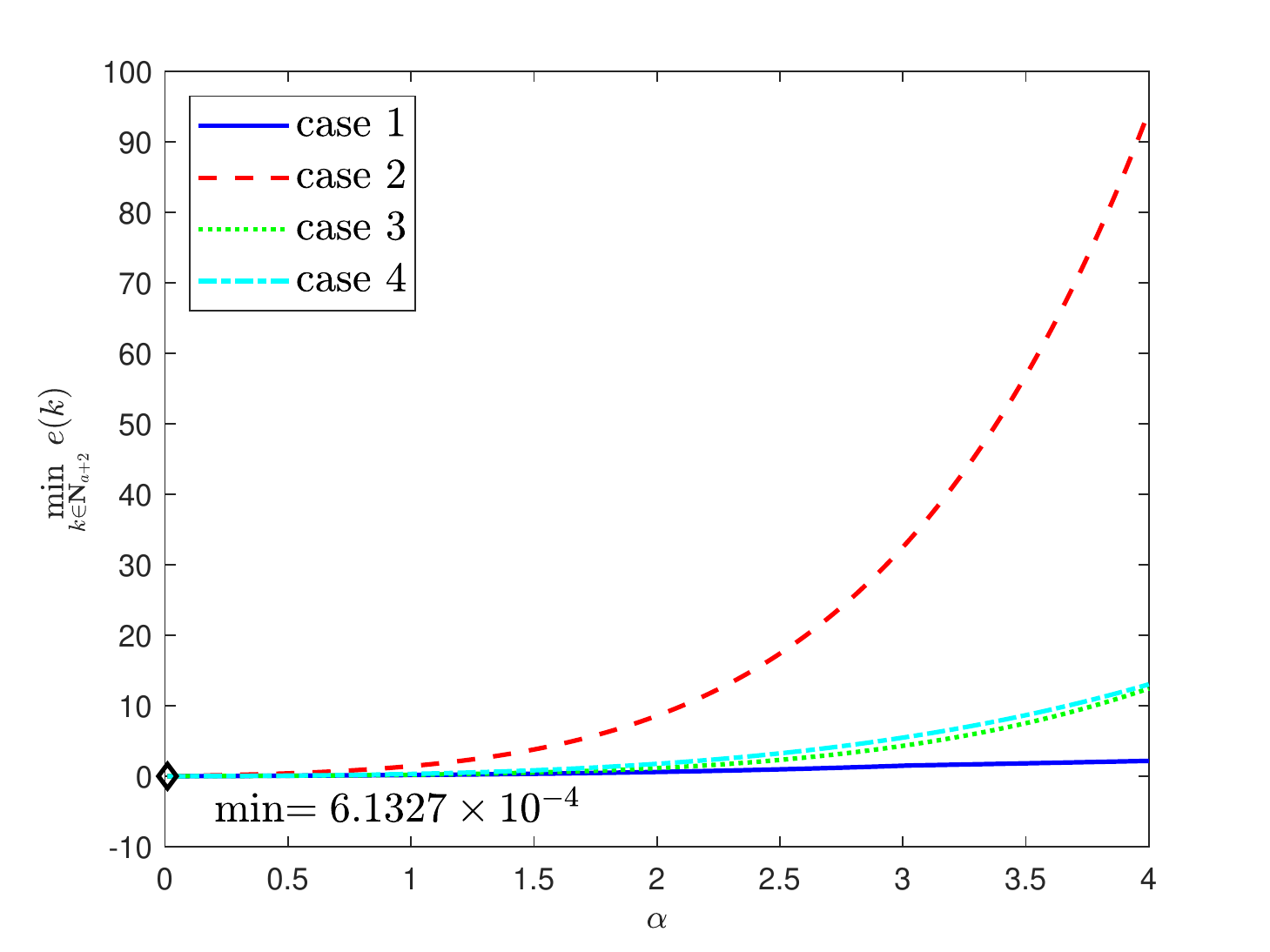}
	\label{Fig18a}
	\end{minipage}%
	}\hspace{-12pt}
	\subfigure[$w(k) = {( - 1)^{k - a}} - 2$]{
	\begin{minipage}[t]{0.5\linewidth}
	\includegraphics[width=1\hsize]{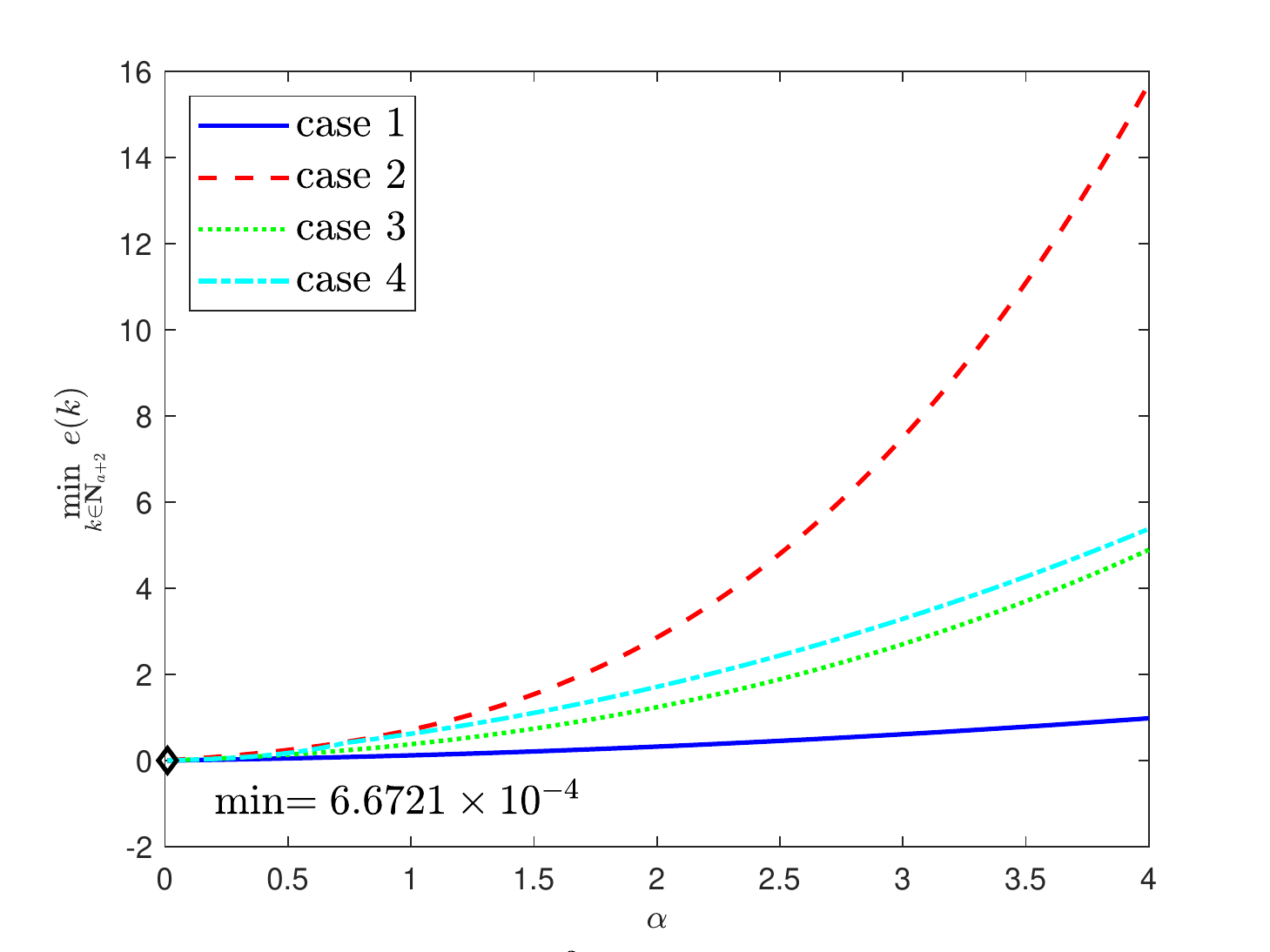}
	\label{Fig18b}
	\end{minipage}%
	}
    \subfigure[$w(k) = \sin (k - a) + 2$]{
	\begin{minipage}[t]{0.5\linewidth}
	\includegraphics[width=1\hsize]{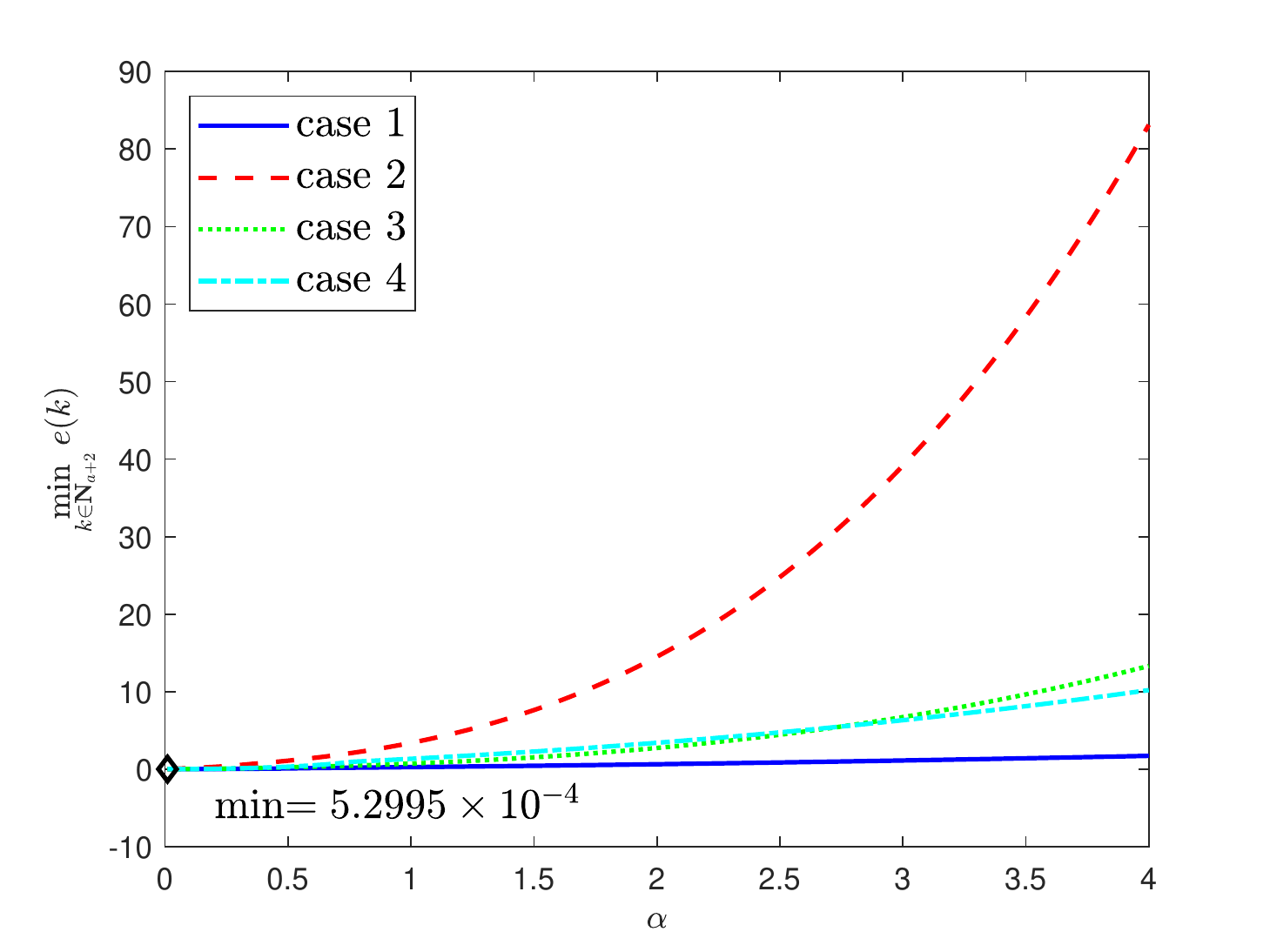}
	\label{Fig18c}
	\end{minipage}%
	}\hspace{-12pt}
	\subfigure[$w(k) = \sin (k - a) - 2$]{
	\begin{minipage}[t]{0.5\linewidth}
	\includegraphics[width=1\hsize]{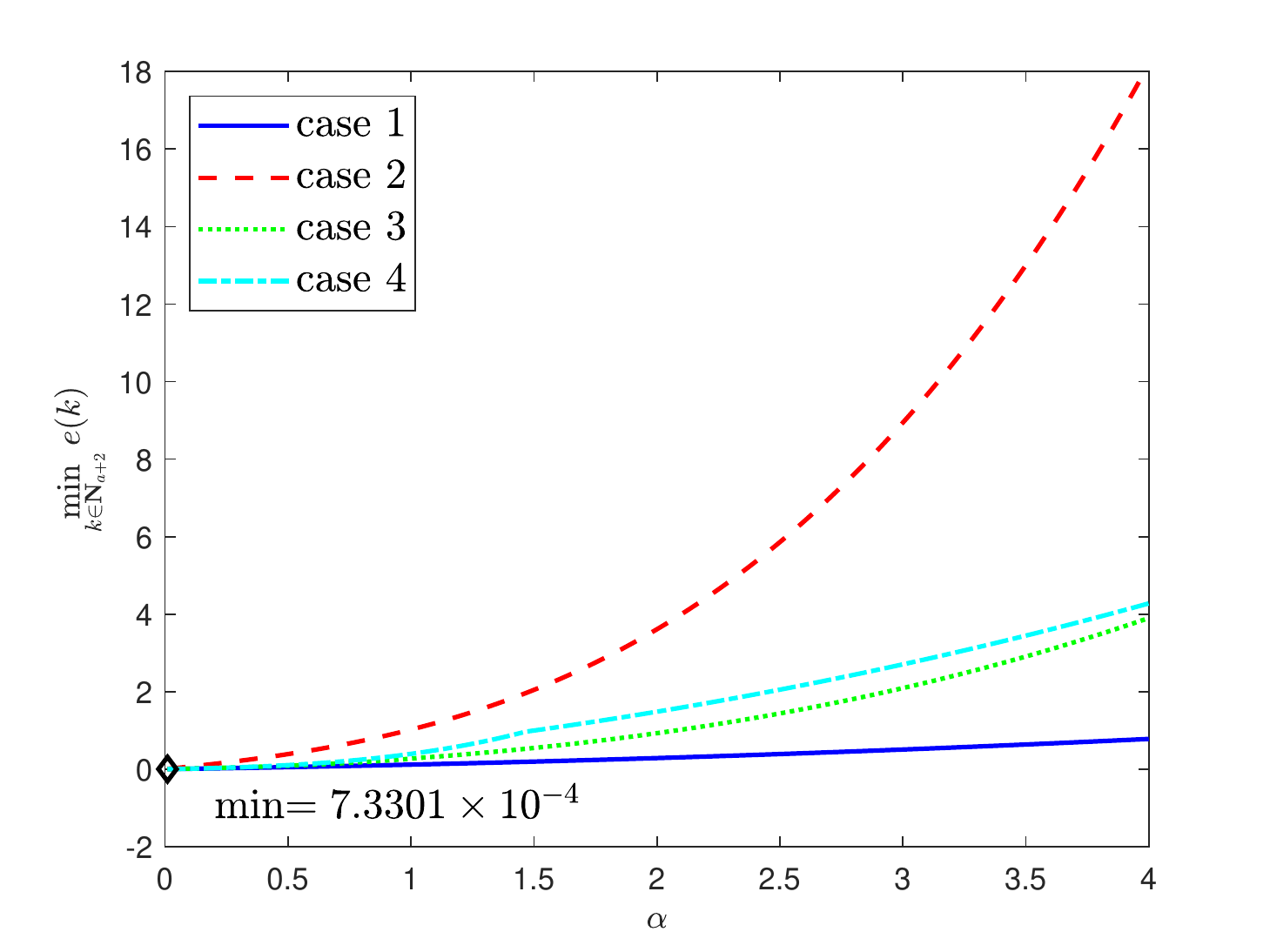}
	\label{Fig18d}
	\end{minipage}%
	}
	\centering
	\caption{The evolution of the error with respect to the order ($p=0.5$).}
	\label{Fig18}
\end{figure}

It can be also found that when $p=1.5$, the maximum of $e(k)$ is negative. When $p=0.5$, the minimum of $e(k)$ is positive. All of obtained simulated results illustrate the correctness of Theorem \ref{Theorem3.34}.
\end{example}

\section{Concluding remarks}\label{Section5}
In this paper, the nabla tempered fractional calculus has been investigated systemically. A series of interesting and promising inequalities are developed rigorously including the monotonicity, the comparison principle, the difference inequality, the sum inequality. These properties have huge potentials in stability analysis and controller design. The proposed properties could greatly enrich the comprehension of nabla tempered fractional calculus and facilitate its applications. In addition, it is hoped that the techniques and consequences of this work could inspire the researchers to explore more interesting sequel in this area.

\bibliographystyle{amsplain}
\bibliography{database}

\end{document}